\RequirePackage{amsthm}
\documentclass[pdflatex,sn-basic]{sn-jnl}
\pdfoutput=1

\usepackage{graphicx}%
\usepackage{multirow}%
\usepackage{amsmath,amssymb,amsfonts}%
\usepackage{amsthm}%
\usepackage{mathrsfs}%
\usepackage[title]{appendix}%
\usepackage{xcolor}%
\usepackage{textcomp}%
\usepackage{manyfoot}%
\usepackage{booktabs}%
\usepackage{algorithm}%
\usepackage{algorithmicx}%
\usepackage{algpseudocode}%
\usepackage{listings}%

\usepackage{bold-extra}

\usepackage{graphics}
\usepackage{latexsym}
\usepackage{ifthen}
\usepackage{array}
\usepackage{multicol}
\usepackage{bbm}

\usepackage{pgf}
\usepackage{pgfplots}
\pgfplotsset{compat=1.15}
\usepackage{tikz, mathtools, stmaryrd} 
\usetikzlibrary{arrows,calc,automata,backgrounds,decorations,intersections,pgfplots.fillbetween,shapes.misc}
\tikzset{every picture/.style={thick,>=angle 60}}
\tikzset{MDPrand/.style={draw,circle,minimum size=11*1.5,inner sep=0}}
\tikzset{MDPcont/.style={draw,rectangle,minimum size=9*1.5,inner sep=0}}
\tikzset{MDPbad/.style={fill=red}}
\definecolor{myorange}{RGB}{255, 128, 0}

\usetikzlibrary{shapes.multipart,shapes.misc}
\usepgflibrary{decorations.pathreplacing}
\usetikzlibrary{arrows,calc,topaths}
\usetikzlibrary{automata,shapes.arrows, backgrounds, intersections, positioning}
\usetikzlibrary{shadows}

\usetikzlibrary{decorations.pathmorphing}
\usetikzlibrary{decorations.shapes}
\usetikzlibrary{shapes}
\usetikzlibrary{backgrounds}
\usetikzlibrary{fit}
\usetikzlibrary{arrows}
\usetikzlibrary{calc}
\usetikzlibrary{mindmap}
\usetikzlibrary{matrix}

\usepackage{mathtools}
\usepackage{thmtools} 
\usepackage{xcolor}

\usepackage{bbding}

\usepackage{hyperref} 
\usepackage{cleveref,thm-restate}
\hypersetup{colorlinks,linkcolor=blue,citecolor=blue,urlcolor=blue}

\usepackage{nicematrix}

\newcommand{\+}[1]{\mathbb{#1}}

\newcommand{\N}{\+{N}}
\newcommand{\R}{\+{R}}
\newcommand{\x}{\times}

\usepackage{wasysym} 
\newcommand{\rsymbol}{R}
\newcommand{\zsymbol}{C}
\newcommand{\zstates}{\states_\zsymbol}
\newcommand{\rstates}{\states_\rsymbol}

\newcommand{\abs}[1]{\lvert#1\rvert}
\newcommand{\card}[1]{\abs{#1}}

\newcommand{\eqby}[2][=]{\stackrel{\mathrm{#2}}{#1}}
\newcommand{\eqdef}{\eqby{def}}
\newcommand{\defeq}{\eqdef}

\newcommand{\eps}{\varepsilon}
\newcommand{\step}[2][]{\xrightarrow[#1]{#2}}

\newcommand{\problemx}[3]{
\par\noindent\underline{\sc#1}\par\nobreak\vskip.2\baselineskip
\begingroup\clubpenalty10000\widowpenalty10000
\setbox0\hbox{\bf INPUT:\ }\setbox1\hbox{\bf QUESTION:\ }
\dimen0=\wd0\ifnum\wd1>\dimen0\dimen0=\wd1\fi
\vskip-\parskip\noindent
\hbox to\dimen0{\box0\hfil}\hangindent\dimen0\hangafter1\ignorespaces#2\par
\vskip-\parskip\noindent
\hbox to\dimen0{\box1\hfil}\hangindent\dimen0\hangafter1\ignorespaces#3\par
\endgroup}

\newcommand{\Runs}[2][]{{\textit{Runs}^{{#1}}_{{#2}}}}
\newcommand{\pRuns}[2][]{{\textit{H}^{{#1}}_{{#2}}}}

\newcommand{\dist}{\mathcal{D}}

\newcommand{\cobuchi}[1]{\textsf{co-B\"uchi}(#1)}

\newcommand{\always}{{\sf G}}
\newcommand{\eventually}{{\sf F}}

\newcommand{\hide}[1]{}

\newcommand{\lrc}[1]{(#1)}

\newcommand{\ignore}[1]{}

\newcommand{\tuple}[1]{\lrc{#1}}

\newcommand{\mdp}{{\mathcal M}}

\newcommand{\mdptupler}{\tuple{\states,\zstates,\rstates,\transition,\probp,r}}
\newcommand{\mdptuple}{\mdptupler}

\newcommand{\states}{S}

\renewcommand{\state}{s}

\newcommand{\transition}{{\longrightarrow}}

\newcommand{\probp}{P}

\newcommand{\complementof}[1]{\overline{#1}}

\newcommand{\play}{\rho}
\newcommand{\playset}{{\mathfrak R}}

\newcommand{\zstrat}{\sigma}

\newcommand{\xstrat}{\tau}

\newcommand{\zstratset}{\Sigma}

\newcommand{\memory}{{\sf M}}
\newcommand{\memconfset}{{\sf M}}
\newcommand{\memconf}{{\sf m}}

\newcommand{\expectval}{{\mathcal E}}
\newcommand{\probm}{{\mathcal P}}

\newcommand{\formula}{{\varphi}}

\newcommand{\valueof}[2]{{\mathtt{val}_{#1}(#2)}}

\mathchardef\mhyphen="2D 

\newcommand{\bubble}[2]{{\sf Bubble}_{#1}(#2)}

\newcommand{\liminfppobj}{\liminf_{\it DP}(\ge 0)}
\newcommand{\limsupppobj}{\limsup_{\it DP}(\ge 0)}

\newcommand{\G}{\always}
\newcommand{\F}{\eventually}

\newcommand{\limsupppexp}{{\mathcal E}(\limsup_{\it DP})}
\newcommand{\liminfppexp}{{\mathcal E}(\liminf_{\it DP})}

\newcommand{\transience}{\mathtt{Transience}}
\newcommand{\reward}{\mathit{r}}

\newcommand{\pmdp}{\mdp_{*}}
\newcommand{\pstates}{\states_{*}}
\newcommand{\pzstates}{\states_{*\zsymbol}}
\newcommand{\prstates}{\states_{*\rsymbol}}
\newcommand{\ptransition}{\transition_{*}}
\newcommand{\pprobp}{\probp_{*}}

\newcommand{\successors}[1]{\mathsf{Succ}({#1})}

\theoremstyle{thmstyleone}%
\newtheorem{theorem}{Theorem}
\newtheorem{proposition}[theorem]{Proposition}%
\newtheorem{lemma}[theorem]{Lemma}
\newtheorem{claim}[theorem]{Claim}
\newtheorem{corollary}[theorem]{Corollary}
\newtheorem{remark}{Remark}%

\newtheorem{definition}{Definition}%

\begin{document}
\title{
Strategy Complexity of Limsup and Liminf Threshold Objectives in Countable MDPs,
with Applications to Optimal Expected Payoffs
}

\author[1]{\fnm{Richard} \sur{Mayr}}
\author[1]{\fnm{Eric} \sur{Munday}}
\affil*[1]{
\orgname{University of Edinburgh}, \orgaddress{\street{10 Crichton Street}, \city{Edinburgh}, \postcode{EH8 9AB}, \country{UK}}}

\abstract{
We study Markov decision processes (MDPs) with a countably infinite
number of states.
The $\limsup$ (resp.\ $\liminf$) threshold objective is to maximize
the probability that the $\limsup$ (resp.\ $\liminf$)
of the infinite sequence of directly seen rewards is non-negative.
We establish the complete picture of the strategy complexity
of these objectives, 
i.e., the upper and lower bounds on
the memory required by $\eps$-optimal (resp.\ optimal) strategies.

We then apply these results to solve two open problems from
\cite[p.43 and p.53]{Sudderth:2020} about the strategy complexity of optimal
strategies for the \emph{expected} $\limsup$ (resp.\ $\liminf$) payoff.
}




\maketitle

\section{Introduction}
\paragraph{Background.}
We consider Markov decision processes (MDPs) with countably infinite
numbers of states and countable action sets.
All runs are of infinite length, i.e., no termination.
MDPs are a standard model for dynamic systems that
exhibit both stochastic and controlled behavior (see, e.g., 
textbooks \cite{DubbinsSavage:2014,Puterman:book,MaitraSudderth:DiscreteGambling,raghavan2012} and references therein).
Some fundamental results and proof techniques for countable MDPs were
established in the framework of Gambling Theory (e.g., 
\cite{DubbinsSavage:2014,MaitraSudderth:DiscreteGambling}). See also Ornstein's
seminal paper on stationary strategies \cite{Ornstein:AMS1969}.
Further applications include control theory (e.g.,
\cite{blondel2000survey,NIPS2004_2569,Ziliotto:2016}),
operations research
and finance (e.g., \cite{schal2002markov,nowak2005,bauerle2011finance,Solan:2014,Flesch:JOTA2020})
artificial intelligence and machine
learning (e.g., \cite{sigaud2013markov,sutton2018reinforcement})
and formal verification (e.g., \cite{ModCheckPrinciples08,EWY2010,BBEKW2010,BBEK:IC2013,
EY:JACM2015,ACMSS2016,ModCheckHB18,KMST2020c}).
The latter works often use countable MDPs to describe unbounded structures in
computational models such as stacks/recursion, counters, queues, etc.
Properties of the long run behavior of MDPs have been extensively studied
(e.g., \cite{Hill:79,sudderth1983gambling,MaitraSudderth:DiscreteGambling,Gimbert2011ComputingOS,Renault:2017,KMST:ICALP2019,Renault:2020,Flesch:JOTA2020,Sudderth:2020}).

A countable MDP can be described as a directed graph
with countably many vertices, where each vertex represents a state.
A directed edge $s \to s'$ from vertex $s$ to vertex $s'$ is also called
a \emph{transition} from $s$ to $s'$.
Many representations of MDPs make an explicit distinction between controlled choices
and random choices by partitioning the states into controlled states and random states.
In a controlled state $s$, the player can choose
a distribution over the set of successor states of $s$, i.e., a distribution
over $\successors{s} \eqdef \{s' \mid \state\transition{}\state'\}$.
In a random state $s$, the next state is chosen according to a predefined
probability distribution over $\successors{s}$.
In a countably infinite MDP, it is possible that, for some state $s$,
the set $\successors{s}$ is infinite, in which case the MDPs is said to be
\emph{infinitely branching}. On the other hand, if $\successors{s}$ is finite
for every state $s$ then the MDP is \emph{finitely branching}. 
In this model, a numeric reward can be assigned to each transition,
or alternatively to each state.

A slightly different representation of MDPs has often been used in gambling theory.
Here the random states are kept implicit. At every state, there is a set of
available actions, and the player chooses a distribution over these actions.
Every action yields a distribution over the states, from which the successor
state is sampled. A numeric reward is assigned to each state.
(Alternatively, each action can be associated with a combined distribution
over states and rewards.)

These two different representations of MDPs are trivially equivalent via
mutual encoding. However, it should be noted that the MDP being finitely
branching in the first model is a stronger condition than requiring
finite action sets in the second model. Even if all action sets are finite,
an action could still yield a distribution over states which has infinite
support. MDPs with finite action sets in the second model
correspond to those in the first model where all controlled states are
finitely branching (while random states can still be infinitely branching).

By fixing a strategy for the player and an initial state, one obtains a probability space
of runs of the MDP. The player's goal is to optimize the expected value of
some objective function on the runs.
The amount/type of memory and randomization that $\eps$-optimal (resp.\ optimal) strategies
need for a given objective is called its \emph{strategy complexity}.

\paragraph{\texorpdfstring{$\limsup$ and $\liminf$  objectives.}{Lim sup and
    Lim inf objectives}}
MDPs are given a reward structure by assigning a real-valued
(resp.\ integer or rational) reward to
each transition. (Alternatively, rewards can be assigned to states.
The two versions can easily be encoded into each other; cf.~\Cref{sec:connections}.)
Every run then induces an infinite sequence of
seen transition rewards $r_0r_1r_2\dots$
(also called \emph{daily payoffs} \cite{Maitra-Sudderth:2003}
or \emph{point payoffs} \cite{MM:CONCUR2021}).
General objectives are defined by real-valued bounded measurable functions on runs.
(In some cases, e.g., the threshold objectives below,
this function is just an indicator function of some measurable
event, i.e., one tries to maximize the probability of the event.)

For $\limsup$ objectives, the payoff of a run is defined
as the $\limsup$ of the daily payoffs, i.e., $\limsup_{n\ge 0} r_n$,
and not the sum or the average.
This $\limsup$ objective comes from two sources. In the gambling theory
of Dubins and Savage \cite{DubbinsSavage:2014},
it corresponds to a nonleavable gambling problem
(see also \cite{Dubins:1989,MaitraSudderth:DiscreteGambling}).
In game theory, it appears in Blackwell's papers on $G_\delta$ games
\cite{Blackwell:1969,Blackwell:1989}, and it has also been
considered in \cite{MaitraSudderth:DiscreteGambling,Maitra-Sudderth:2003}.

The $\limsup$ \emph{threshold objective} is to maximize the
probability that $\limsup_{n\ge 0} r_n \ge 0$.
\footnote{
One could also consider threshold objectives with strict
inequality, i.e., $\limsup_{n\ge 0} r_n > 0$.
For integer rewards this is equivalent to the non-strict objective
$\limsup_{n\ge 0} r_n \ge 1$.
For real rewards, the strategy complexity of $\eps$-optimal strategies
for the strict $\limsup_{n\ge 0} r_n > 0$ objective
is the same as for the non-strict case
$\limsup_{n\ge 0} r_n \ge 0$.
This is because
$(\limsup_{n\ge 0} r_n > 0) = \cup_{k \ge 1} (\limsup_{n\ge 0} r_n \ge 2^{-k})$.
By continuity of measures, for some $k$ depending on $\eps$,
the non-strict objective $(\limsup_{n\ge 0} r_n \ge 2^{-k})$
approximates the strict objective $(\limsup_{n\ge 0} r_n > 0)$ sufficiently closely.
The strategy complexity of optimal strategies for the strict case
is open.
}
Similarly, the $\liminf$ threshold objective is to maximize the
probability that $\liminf_{n\ge 0} r_n \ge 0$.
In the special case where the transition rewards are limited to the integers,
the $\limsup$ (resp.\ $\liminf$) threshold objective
corresponds to the objective of seeing rewards $\ge 0$ infinitely often
(resp.\ to see rewards $<0$ only finitely often).
These are also called B\"uchi (resp.\ co-B\"uchi)
objectives in \cite{KMST:ICALP2019,KMST2020c}, due to their
connections to automata theory and temporal logics \cite{CGP:book,ModCheckHB18}.
However, the general case of \emph{infinite-state} MDPs with rational/real
rewards is more complex.
E.g., the sequence of rewards $-1/2, -1/3, -1/4 \dots$ does satisfy
$\limsup \ge 0$ and $\liminf \ge 0$,
even though all rewards are negative.

A related problem is to maximize the \emph{expected} $\limsup$
(resp.\ $\liminf$) of the daily payoffs in the runs.
This corresponds to nonleavable gambling problems \cite[Section 4]{MaitraSudderth:DiscreteGambling}.
Unlike for the threshold objective, an optimal strategy to maximize the
\emph{expected} $\limsup$ (resp.\ $\liminf$) could accept 
a high probability of a negative $\limsup$ (resp.\ $\liminf$),
provided that the remaining runs have a huge positive $\limsup$ (resp.\ $\liminf$).
It was shown in  \cite{Sudderth:2020} that optimal strategies for the expected $\limsup$,
if they exist, can be chosen as deterministic Markov.
Threshold objectives and expected $\limsup$/$\liminf$ objectives are
closely related; see \Cref{sec:connections,sec:expected}.

Note that the expected $\limsup$ of the daily payoffs
is different from the $\limsup$ of the expected daily payoffs
(and likewise for the $\liminf$).
One could consider $\limsup_{n \ge 0} E(X_n)$ for
random variables $X_n$ that depend on histories of length $n$ and
define $X_n \eqdef r_n$ as the $n$-th daily payoff.
Consider the simple example of a Markov chain with just two runs, each of
probability $1/2$, where the first run has rewards $010101\dots$ ($1$ at odd
steps) and the second run has rewards $101010\dots$ ($1$ at even steps).
Both runs have $\limsup_{n\ge 0} r_n = 1$, and thus $E(\limsup_n r_n) = 1$.
However, $\limsup_{n\ge 0} E(X_n) = \limsup_{n\ge 0} 1/2 = 1/2$.
The $\limsup$/$\liminf$ of the expected daily payoffs is not a topic in this paper.

\paragraph{Strategy Complexity.}
Classes of strategies are defined via the amount and type of memory
used, and whether they are randomized (aka mixed) or deterministic (aka pure).
Some canonical types of memory for strategies are the following:
No memory (also called stationary, memoryless or positional),
finite memory, a step counter (i.e., a discrete clock),
and general infinite memory.
Strategies using only a step counter are also called \emph{Markov strategies}
\cite{Puterman:book}.
Moreover, there can be combinations of these, e.g., a step counter plus some
finite general purpose memory.
Other types of memory are possible, e.g., an unbounded stack or a queue, but they are less
common in the literature.

\ignore{
To establish an upper bound $X$ on the strategy complexity of an objective
in countable MDPs, it suffices to prove that there always exist good
($\eps$-optimal, resp.\ optimal) strategies in some class of strategies $X$.
Lower bounds on the strategy complexity of an objective 
can only be established in the sense of proving that good
strategies for the objective do not exist in some classes $Y$, $Z$, etc.

Note that different classes of strategies are not always comparable,
for several reasons.
First, different types of memory may be incomparable. E.g.,
a step counter uses infinite memory, but it is updated in a very particular
way, and thus it does not subsume a finite general purpose memory.
Second, randomized strategies are more general than deterministic ones if they
use the same memory, but not if the memory is different.
E.g., randomized positional strategies are incomparable to deterministic
strategies with finite memory (or a step counter).
Since strategy classes are not always comparable, there
can be cases with several incomparable upper/lower bounds.

Moreover, there is no weakest type of infinite memory with restricted use.
Hence, upper and lower bounds on the strategy complexity of an objective
can be only be tight \emph{relative} to the
considered alternative strategy classes (e.g., the canonical classes mentioned above).
}

The upper bound tells us that a certain amount/type of memory is sufficient for
a good ($\eps$-optimal, resp.\ optimal)
strategy, while the lower bound tells us what is not sufficient.
By Rand(X) (resp.\ Det(X)) we denote the classes of randomized
(resp.\ deterministic) strategies that use memory of size/type X.
SC denotes a step counter (aka a discrete global clock),
and F denotes arbitrary finite memory.
1-bit is a special case of F, where only 1 bit of memory is used.
Positional (aka memoryless or stationary) means that no memory is used.
E.g., a lower bound $\neg$Rand(SC) means that randomized strategies that use just
a step counter are not sufficient.

\paragraph{Finite-state vs.\ Infinite-state MDPs.}
If an MDP has only finitely many states and transitions,
then there are only finitely many \emph{different} transition rewards,
and thus the $\limsup$ (resp.\ $\liminf$) threshold objective coincides with the
objective of seeing rewards $\ge 0$ infinitely often
(resp.\ of seeing rewards $<0$ only finitely often), even if real-valued
transition rewards are allowed.
In contrast, in an MDP with a countably infinite number of transitions,
one could have a sequence of
rewards $-1/2, -1/3, -1/4 \dots$ that does satisfy $\limsup \ge 0$
and $\liminf \ge 0$ even though
each individual reward is negative.
Thus, in MDPs with a countably infinite number of states/transitions,
the $\limsup$ (resp.\ $\liminf$) threshold objective
is strictly more general than the B\"uchi (resp.\ co-B\"uchi) objective
of seeing certain states/transitions infinitely often (resp.\ finitely often); cf.~\Cref{sec:connections}.
Moreover, optimal strategies need not exist in countably infinite-state MDPs
(not even for much simpler objectives like reachability),
$\eps$-optimal (resp.\ optimal) strategies can require
infinite memory, and computational problems are not defined in general, since
a countable MDP need not be finitely presented.
See \cite{KMSW2017} for a more detailed discussion of the differences between
finite-state and infinite-state MDPs.

\cite{Hill:79,Hill-Pestien:1987,Pestien-Wang:1993} discuss gambling problems with \emph{finitely many controlled
states} but infinite action sets. In our terminology, these are MDPs with
finitely many controlled states but infinitely many random states and also
infinite branching and infinitely many different transitions.
Optimal strategies for $\limsup$ (resp.\ $\liminf$)
need not exist in this case, but $\eps$-optimal strategies can be chosen as
deterministic Markov.
On the other hand, if there are only finitely many states and transitions then 
there always exist optimal
deterministic stationary strategies for the objectives considered in this
paper, because it suffices to visit certain subsets of transitions infinitely
often for $\limsup$ (i.e., repeated reachability), resp.\ to visit certain
sets of transitions only finitely often for $\liminf$ (i.e., eventual safety).

\paragraph{Our contribution.}
We establish the complete picture of
the strategy complexity of the $\limsup$ and $\liminf$ threshold objectives
for \emph{countably infinite-state} MDPs.
\Cref{tab:limsup} shows the upper and lower bounds on the strategy complexity
of the 
$\limsupppobj$ threshold objective, i.e., to maximize the probability that the
$\limsup$ of the daily payoffs is non-negative. These bounds depend on whether
one considers $\eps$-optimal strategies or optimal strategies (where they exist).
Similarly, \Cref{tab:liminf} shows the upper and lower bounds on the strategy complexity
of the $\liminfppobj$  threshold objective, i.e., to maximize the probability that the
$\liminf$ of the daily payoffs is non-negative.
Here the bounds depend on whether the MDPs are infinitely branching or
finitely branching. While the bounds for $\eps$-optimal strategies and optimal
strategies coincide for this objective, the proofs are different.

\begin{table}[tbp]
\begin{tabular}{|l||l|l|}
\hline
$\limsupppobj$ (resp., $\bigcap_{i \in \N} \G \F \, A_i$)  & Upper bound                                                       & Lower bound                                                 \\ \hline
\multirow{2}{*}{$\eps$-optimal strategies} & \multirow{2}{*}{Det(SC + 1-bit) \ref{cor:epsGF}} & $\neg$Rand(SC) \ref{thm:buchippextract} and \\
                                           &                                                                   & $\neg$Rand(F) \ref{thm:ppepstep}           \\ \hline
\multirow{2}{*}{Optimal strategies}                         & Rand(Positional) \ref{cor:GFmr} or               & \multirow{2}{*}{$\neg$Det(F) \ref{thm:limsuplowerdetf}  }   \\
                                           & Det(SC) \ref{cor:GFsc}                           &                                                             \\ \hline
\end{tabular}
\caption{Summary of strategy complexity results for the 
$\limsupppobj$ threshold objective, i.e., to maximize the probability that the
$\limsup$ of the daily payoffs is non-negative.
(The formula $\bigcap_{i \in \N} \G \F \, A_i$ is an equivalent description of
this objective in terms of temporal logic; see \Cref{sec:prelim,sec:connections}.)
Note that the (infinite vs.\ finite) branching degree of the
  MDP does not affect the strategy complexity here.}\label{tab:limsup}
\end{table}

\begin{table}[tbp]
\begin{tabular}{|ll||l|l|}
\hline
\multicolumn{2}{|l||}{$\liminfppobj$ (resp., $\bigcap_{i \in \N} \F \G \, A_i$)}                         & Upper bound     & Lower bound   \\ \hline
\multicolumn{1}{|l|}{\multirow{2}{*}{Infinitely Branching}} & Optimal
                                                              strategies
                                                                                                           &
                                                                                                             Det(SC) \ref{infoptupperpp} & $\neg$Rand(F) \ref{infbranchsteplower} \\
\cmidrule{2-4} 
\multicolumn{1}{|l|}{}                                      & $\eps$-optimal
                                                              strategies &
                                                                           Det(SC)
                                                                           \ref{infepsupperpp}        & $\neg$Rand(F) \ref{infbranchsteplower} \\ \hline
\multicolumn{1}{|l|}{\multirow{2}{*}{Finitely Branching}}   & Optimal
                                                              strategies
                                                                                                           &
                                                                                                             Det(Positional) \ref{finoptupperpp} & Trivial       \\
\cmidrule{2-4} 
\multicolumn{1}{|l|}{}                                      & $\eps$-optimal
                                                              strategies &
                                                                           Det(Positional) \ref{finpointpayoff}  & Trivial       \\ \hline
\end{tabular}
\caption{Summary of strategy complexity results for
the $\liminfppobj$ threshold objective, i.e., to maximize the probability that the
$\liminf$ of the daily payoffs is non-negative.
(The formula $\bigcap_{i \in \N} \F \G \, A_i$ is an equivalent description of
this objective in terms of temporal logic; see \Cref{sec:prelim,sec:connections}.)
The lower bounds for the finitely branching case are trivial, because
deterministic positional strategies are the simplest type.
}\label{tab:liminf}
\end{table}

Our results generalize the results on the strategy complexity of 
the B\"uchi and co-B\"uchi objectives in countably infinite-state MDPs of \cite{KMST:ICALP2019,KMST2020c}.

We then apply our results on the threshold objectives
to solve two open problems from 
\cite[p.43 and p.53]{Sudderth:2020} about the strategy complexity of optimal
strategies for the \emph{expected} $\limsup$ (resp.\ $\liminf$).

Now we discuss our contributions in more detail.

In \Cref{sec:nestedbuchi} we show that $\eps$-optimal strategies for 
the $\limsup$ threshold objective require exactly a step counter plus one bit
of memory, thus extending the result on B\"uchi objectives in
\cite{KMST:ICALP2019}.
Moreover, we show that optimal strategies, if they exist, can be chosen
as deterministic Markov, or alternatively as positional randomized.
In particular this implies that optimal strategies for the \emph{expected} $\limsup$,
if they exist, can also be chosen as positional randomized, which solves the open question in
\cite[p.53]{Sudderth:2020}; see \Cref{sec:expected}.

In \Cref{sec:nestedcobuchi} we show that $\eps$-optimal (resp.\ optimal) strategies for 
the $\liminf$ threshold objective can be chosen as deterministic Markov.
In the special case of finitely branching countable MDPs, even positional deterministic strategies suffice.
The former result is then applied in \Cref{sec:expected}
to show that optimal strategies for the expected $\liminf$, if they exist,
can also be chosen as deterministic Markov, which solves the open question in
\cite[p.43]{Sudderth:2020} (where this property was shown only for MDPs with
\emph{finitely many} controlled states).

\section{Preliminaries}\label{sec:prelim}

Let $\R$ denote the set of real numbers and $\N$ denote the set of
non-negative integers (including $0$).

\paragraph{Markov decision processes.}
A \emph{probability distribution} over a countable set $S$ is a function
$f:\states\mapsto[0,1]$ with $\sum_{\state\in \states}f(\state)=1$.
Let 
$\dist(\states)$ be the set of all probability distributions over $\states$. 
A \emph{Markov decision process} (MDP) $\mdp$ is described by the tuple $\mdptuple$.
The countable set $\states$ of \emph{states} 
is partitioned into a set $\zstates$ of \emph{controlled states} 
and  a set $\rstates$ of \emph{random states}.
The \emph{transition relation} is $\transition\subseteq\states\x\states$.
We  write $\state\transition{}\state'$ if $\tuple{\state,\state'}\in \transition$,
and  refer to~$s'$ as a \emph{successor} of~$s$.
Let $\successors{s} \eqdef \{s' \mid \state\transition{}\state'\}$ be the set of
successor states of $s$.
We assume that every state has at least one successor.
The \emph{probability function}~$\probp$
assigns each random state $s \in \rstates$ a distribution over
its successor states, i.e., $\probp(s) \in \dist(\successors{s})$. 
The reward function $r$ assigns real-valued rewards to transitions.
(We mainly use transition-based rewards in this paper.
Alternatively, one can consider state-based reward functions $r: S \mapsto \R$.
Transition-based rewards can easily be encoded into state-based rewards and
vice-versa; see \Cref{sec:connections}.)

An MDP is \emph{acyclic} if the underlying directed graph~$(S,\transition)$ is acyclic, i.e., 
there is no directed cycle.
It is  \emph{finitely branching} 
if every state has finitely many successors
and \emph{infinitely branching} otherwise.
An MDP without controlled states
($\zstates=\emptyset$) is called a \emph{Markov chain}.

\paragraph{Strategies and probability measures.}
A \emph{run} is an  infinite sequence of states and transitions
$\play = \state_0e_0\state_1e_1\cdots$ 
such that $e_i = (\state_i, \state_{i+1}) \in \transition$
for all~$i\in \mathbb{N}$.
Let $\Runs{\mdp}{\state_0}$
be the set of all runs from state $\state_0$ in the MDP $\mdp$.
A \emph{history} is a finite prefix of a run that ends in some
controlled state $\state \in \zstates$.
Let $\pRuns{\mdp}{\state_0}$ denote the set of all histories starting
in $\state_0$ and let $\pRuns{\mdp}{}$ denote the set of histories from any state in $\mdp$.

For a run $\play = \state_0e_0\state_1e_1\cdots$,  
we write~$\play_s(i)\eqdef\state_i$ for the $i$-th state along~$\play$
and $\play_e(i)\eqdef e_i$ for the $i$-th transition along~$\play$.
Let $\rho_i \eqdef \state_ie_i\state_{i+1}e_{i+1}\cdots$ be the suffix of $\rho$
that starts at state $s_i$.
We sometimes write runs as $\state_0\state_1\cdots$, leaving the transitions implicit.
We say that a run $\play$ \emph{visits} $\state$ if
$\state=\play_s(i)$ for some $i$, and that~$\play$ starts in~$s$ if $\state=\play_s(0)$. 

A \emph{strategy} 
is a function $\zstrat:\pRuns{\mdp} \mapsto \dist(S)$ that 
assigns to each history $w\state$ (where $\state \in \zstates$),
a distribution over the successors $\successors{\state}$ of $\state$.
The set of all strategies  in $\mdp$ is denoted by $\zstratset_\mdp$ 
(we omit the subscript and write~$\zstratset$ if $\mdp$ is clear from the context).
A run~$\state_0e_0\state_1e_1\cdots$ is consistent with a strategy~$\zstrat$
if for all~$i$
either $\state_i \in \zstates$ and $\zstrat(\state_0e_0\state_1e_1\cdots\state_i)(\state_{i+1})>0$,
or
$\state_i \in \rstates$ and $\probp(\state_i)(\state_{i+1})>0$.

An MDP $\mdp=\mdptuple$, an initial state $\state_0\in \states$, and a strategy~$\zstrat$ 
induce a probability space in which the outcomes are runs starting in $\state_0$
 with measure $\probm_{\mdp,\state_0,\zstrat}$
defined as follows.
It is first defined on \emph{cylinders}, i.e., sets of runs of the form
$s_0 e_0 s_1 e_1 \ldots s_n \ldots$ sharing a common finite prefix.
If $s_0 e_0 s_1 e_1 \ldots s_n$
is not consistent with~$\zstrat$ then
$\probm_{\mdp,\state_0,\zstrat}(s_0 e_0 s_1 e_1 \ldots s_n \ldots) \eqdef 0$,
and otherwise
\begin{align*}
\probm_{\mdp,\state_0,\zstrat}(s_0 e_0 s_1 e_1 \ldots
s_n \ldots) 
\eqdef \prod_{i=0}^{n-1} \bar{\zstrat}(s_0 e_0
s_1 \ldots s_i)(s_{i+1})
\end{align*}
where $\bar{\zstrat}$ is the map that
extends~$\zstrat$ by $\bar{\zstrat}(w s) = \probp(s)$
for all histories $w s$.
By Carath\'eodory's extension theorem~\cite{billingsley-1995-probability}, 
this extends uniquely to a probability measure~$\probm_{\mdp,\state_0,\zstrat}$ on 
the Borel $\sigma$-algebra $\?F$ of subsets of~$\Runs{\mdp}{s_0}$ induced by
the cylinders.
A set in $\?F$ is called an \emph{event}.
General objectives are defined by real-valued bounded measurable functions (w.r.t $\?F$).
In some cases (see below), this function is just an \emph{indicator function} of some
event $Y \in \?F$, i.e., one tries to maximize the probability of the event
$Y$. In such cases we identify the objective with the relevant event.
For $Y\in\?F$ we write $\complementof{Y}\eqdef (\Runs{\mdp}{s_0}\setminus Y) \in \?F$ for its complement.
In the case of general objectives, we write $\expectval_{\mdp,\state_0,\zstrat}$
for the expectation wrt.~$\probm_{\mdp,\state_0,\zstrat}$.
We drop the indices if possible without
ambiguity.

\paragraph{The operators ``eventually'' and ``always''.}
For a compact formal notation to specify properties of runs, we use
the operators $\F$ (eventually) and $\G$ (always) and their extensions with
time bounds \cite{ModCheckHB18,CGP:book}.
$\eventually \formula$ denotes all runs $\play$ that have a suffix that
satisfies property $\formula$, i.e., there exists an $i \ge 0$ such that
$\play_i$ satisfies $\formula$.
Similarly, $\eventually^{\le k} \formula$ denotes all runs $\play$
where there exists an $i \le k$ such that
$\play_i$ satisfies $\formula$.
The operator $\always$ (always) is defined as $\neg\eventually\neg$.
Similarly, the operator $\always^{\leq k}$ (for all times until time $k$) 
is defined as $\neg \eventually^{\leq k} \neg$.
Combined time bounds can be specified by $\G^{[m,n]}$
(for all times between $m$ and $n$), i.e.,
a run $\play$ satisfies $\G^{[m,n]} \formula$ iff for all
$i$ with $m \le i \le n$ we have that $\play_i$ satisfies $\formula$.
For sets of states (resp.\ transitions) $X$ we just write $X$ to
denote the property that the first state (resp.\ transition)
of a run is in $X$.
E.g., if $X$ is a set of states
then $\eventually X$ denotes the set of runs that eventually
visit the set $X$.
Similarly, $\always \eventually X$ denotes the set of runs
that visit the set $X$ infinitely often (i.e., always eventually $X$).
Dually, $\eventually\always \neg X$ denotes the set of runs
that visit $X$ only finitely often (i.e., eventually always not $X$).
Sets of runs specified by (combinations of) these operators are measurable
\cite{Vardi:probabilistic}.

\paragraph{Transience and shift invariance.}
Given an MDP $\mdp=\mdptuple$, consider the objective 
$\transience \eqdef \bigwedge_{s \in S} \F \G \neg s$.
That is to say, $\transience$ is the objective to see no state infinitely often. 
An MDP $\mdp$ is called \textit{universally transient} \cite{KMST:Transient-arxiv}
if $\transience$ is satisfied almost surely from every state $\state_0$ under
all strategies, i.e.,
$\forall \state_0\forall \zstrat \in \zstratset_{\mdp} \, \probm_{\mdp, \state_0, \zstrat}(\transience) = 1$.
In particular, all acyclic MDPs are universally transient, but not only these.
E.g., consider the Markov chain for the Gambler's ruin with restarts and a probability $p$ of
winning. 
It is strongly connected and contains cycles, but it is still
universally transient for any $p > 1/2$, but not for $p \le 1/2$.

An event-based objective $\formula$ is called \emph{shift invariant in $\mdp$}
iff for every run
$\rho'\rho$ in $\mdp$ with some finite prefix $\rho'$ we have
$\rho'\rho \in \formula \Leftrightarrow \rho \in \formula$.
An objective is called \emph{shift invariant} if it is shift invariant in every MDP.

\paragraph{Strategy classes.}
Strategies are in general  \emph{randomized} (aka mixed)
in the sense that they take values in $\dist(\states)$. 
A strategy~$\zstrat$ is \emph{deterministic} (aka pure) if $\zstrat(\rho)$ is a Dirac
distribution for all $\rho$.
General strategies can be \emph{history dependent}, while others are
restricted by the size or type of memory they use; see below.
We consider certain classes of strategies:
\begin{itemize}
\item
  A strategy $\zstrat$ is \emph{positional} (also called \emph{memoryless}
  or \emph{stationary})
if its choices depend only on the current state.
We may describe positional strategies as functions $\zstrat: \zstates \mapsto
\dist(\states)$ where $\zstrat(\state) \in \dist(\successors{\state})$.
Memoryless deterministic (resp.\ randomized) strategies are also abbreviated as MD
(resp.\ MR).
\item 
A strategy~$\zstrat$ is \emph{finite memory}~(F) if 
there exists a finite memory~$\memory$ implementing~$\zstrat$.
(See \Cref{app-def} for a formal definition how strategies use memory.)
Hence Rand(F) (resp.\ Det(F)) stands for finite memory randomized
(resp.\ deterministic) strategies.
They are also abbreviated as FR (resp.\ FD).
\item
A step counter strategy bases decisions only on
the current state and the number of steps taken so far, i.e., it uses
an unbounded integer counter that gets incremented by $1$ in every step
(like a discrete clock).
Such strategies are also called \emph{Markov strategies} \cite{Puterman:book}.
Rand(SC) (resp.\ Det(SC)) stands for step counter using randomized
(resp.\ deterministic) strategies.
\end{itemize}
A step counter strategy uses infinite memory, but only in a very restricted
way, since the player has no control over the memory updates.
Thus step counter strategies do \emph{not} subsume finite memory strategies.
Combinations of the above types of memory are possible, e.g.,
Det(SC + 1-bit) stands for deterministic strategies using a step counter plus
one bit of general purpose memory. See also \cite{KMST2020c}.

\paragraph{Optimal and \texorpdfstring{$\eps$-optimal}{epsilon-optimal} strategies.}
Given an objective that is defined as an event~$\formula$
(i.e., the objective function is the indicator function of $\formula$),
the value of state~$s$ in an MDP~$\mdp$, denoted by 
$\valueof{\mdp,\formula}{s}$, is the supremum probability of
achieving~$\formula$, i.e.,
$\valueof{\mdp,\formula}{s} \eqdef\sup_{\sigma \in \Sigma}
\probm_{\mdp,\state,\zstrat}(\formula)$ where $\Sigma$ is the set of all strategies.
Similarly with $\expectval_{\mdp,\state_0,\zstrat}$ for general objectives
defined via bounded measurable functions.
For $\eps\ge 0$ and state~$s\in\states$, we say that a strategy is \emph{$\eps$-optimal} from $s$
iff $\probm_{\mdp,\state,\zstrat}(\formula) \geq \valueof{\mdp,\formula}{s} -\eps$
(resp.\ iff $\expectval_{\mdp,\state,\zstrat}(\formula) \geq
\valueof{\mdp,\formula}{s} -\eps$ for objectives wrt.\ bounded measurable
functions $\formula$).
A $0$-optimal strategy is called \emph{optimal}. 
An optimal strategy for some event-based objective
is \emph{almost-surely winning} if $\valueof{\mdp,\formula}{s} = 1$. 
Considering a positional strategy as a function
$\zstrat: \zstates \mapsto \dist(\states)$
(where $\zstrat(\state) \in \dist(\successors{\state})$)
and $\eps\ge 0$, $\zstrat$
is \emph{uniformly} $\eps$-optimal  (resp.~uniformly optimal) if it is
$\eps$-optimal (resp.~optimal) from \emph{every} state $\state\in \states$.
(A closely related concept is a \emph{stationary} strategy, which
  also bases decisions only on the current state. However, some authors call a
  strategy ``stationary $\eps$-optimal'' if it is $\eps$-optimal from every state (i.e., uniform in our terminology),
and call it ``semi-stationary'' if it is $\eps$-optimal only from some fixed initial state.)

\paragraph{The step counter encoded MDP.}
Given an MDP $\mdp$, we define the MDP $S(\mdp)$ which has a step counter encoded into the state. This will allow us to obtain Markov strategies in $\mdp$ from positional strategies in $S(\mdp)$.

\smallskip
\begin{definition}\label{def:encodestep}
Let $\mdp$ be an MDP with an initial state $s_{0}$. 
We then construct the MDP $S(\mdp) \eqdef (S', \zstates', \rstates', \longrightarrow_{S(\mdp)}, P')$ as follows:
\begin{itemize}
\item The state space of $S(\mdp)$ is 
$S' \eqdef \{ (s,n) \mid s \in S \text{ and } n \in \N \}$.
Note that $S'$ is countable.
We write $s_{0}'$ for the initial state $(s_{0},0)$.
\item $\zstates' \eqdef \{ (s,n) \in S' \mid s \in \zstates \text{ and } n \in \N\}$
and $\rstates' \eqdef S' \setminus \zstates'$.
\item The set of transitions in $S(\mdp)$ is 
\[ \longrightarrow_{S(\mdp)} \eqdef 
\{ 
\left( (s,n),(s',n+1) \right) \mid (s,n),(s',n+1) \in S', 
 s \longrightarrow_{\mdp} s' 
\}.
\]
\item $P': \rstates' \mapsto \mathcal{D}(S')$ is defined such that 
\[
P'(s,n)(s',n+1) \eqdef  
    \begin{cases}
    P(s)(s') & \text{ if } (s,n) \longrightarrow_{S(\mdp)} (s',n+1) \\
    0 & \text{ otherwise }
    \end{cases}
\]

\item If $\mdp$ has rewards, then $r((s,n) \longrightarrow_{S(\mdp)} (s',n+1)) \eqdef r(s \longrightarrow_{\mdp} s')$.
\end{itemize}
\end{definition}

\smallskip
\begin{lemma}\label{steptopoint}
Let $\mdp$ be an MDP with initial state $s_{0}$. Let $\varphi$ be an objective
depending only on the sequence of seen transition rewards (the daily payoffs).
For every finite-memory strategy $\sigma'$ from state $(s_{0},0)$ in $S(\mdp)$ there exists a
corresponding strategy $\sigma$ from state $s_0$ in $\mdp$
which uses the same memory as $\sigma'$ plus a step counter, such that
$\probm_{S(\mdp),(s_0,0),\sigma'}(\varphi)
= \probm_{\mdp,s_0,\sigma}(\varphi)$.
\end{lemma}
\begin{proof}
Let $\sigma'$ be a finite-memory strategy in $S(\mdp)$ 
from state $(s_{0},0)$.
We define a strategy $\sigma$ on $\mdp$ from $s_0$ that uses the same memory
as $\sigma'$ plus a step counter: $\sigma$ plays on $\mdp$ exactly like
$\sigma'$ plays on $S(\mdp)$ by keeping the step counter in its memory instead
of in the state, i.e., at any given state $s$ and step counter value $n$,
$\sigma$ plays exactly as $\sigma'$ plays in state $(s,n)$, and it updates its
memory in the same way.
By our construction of $S(\mdp)$ and the definition of $\sigma$,
the sequences of transition rewards seen by $\sigma'$ in runs on $S(\mdp)$
coincide with the sequences of transition rewards seen by $\sigma$ in runs in $\mdp$.
Hence we obtain
$\probm_{S(\mdp),(s_0,0),\sigma'}(\varphi)
= \probm_{\mdp,s_0,\sigma}(\varphi)$.
\end{proof}

\paragraph{The conditioned MDP.}
We recall the notion of the conditioned MDP $\pmdp$, which will allow us to
lift $\eps$-optimal strategies to \emph{uniformly} $\eps$-optimal strategies.
Moreover, it can be used to lift $\eps$-optimal strategies to optimal ones in some cases.

The intuition is as follows. Given an MDP $\mdp$ and an 
objective $\formula$ that is shift invariant in $\mdp$,
the corresponding conditioned MDP $\pmdp$
adds a losing sink state $s_\bot$ and modifies the probabilities such
that all states, except for $s_\bot$, have value $1$ wrt.\ $\formula$.
In more detail, these modified probabilities make it more likely in $\pmdp$
to go to states that have a high value (wrt.\ $\formula$) in $\mdp$. 

\smallskip
\begin{definition}[{\cite[Def. 12]{KMST:Transient-arxiv}}]
\label{def:conditionedmdp}
For an MDP $\mdp=\mdptuple$ and an objective~$\formula$ that is shift invariant in~$\mdp$,
define the \emph{conditioned version} of~$\mdp$ w.r.t.~$\formula$
to be the MDP $\pmdp = \tuple{\pstates,\pzstates,\prstates,\ptransition,\pprobp}$ with
\begin{align*}
\pzstates \ = \ &\{s \in \zstates \mid \valueof{\mdp}{s} > 0 \} \\
\prstates \ = \ &\{s \in \rstates \mid \valueof{\mdp}{s} > 0 \} \cup
                  \{s_\bot\} \\
                & \cup \{(s,t) \in \mathord{\transition} \mid s \in \zstates,\ \valueof{\mdp}{s} > 0\} \\
\ptransition \ = \ &\{(s,(s,t)) \in (\zstates \times \mathord\transition) \mid \valueof{\mdp}{s}>0,\ s \transition t\}  \cup \mbox{} \\
                   &\{(s,t) \in \rstates \times \states \mid \valueof{\mdp}{s}>0,\ \valueof{\mdp}{t}>0\} \cup \mbox{} \\
                   &\{((s,t),t) \in (\mathord\transition \times S) \mid \valueof{\mdp}{s}>0,\ \valueof{\mdp}{t}>0\} \cup \mbox{} \\
                   &\{((s,t),s_\bot) \in (\mathord\transition \times \{s_\bot\}) \mid \valueof{\mdp}{s}>\valueof{\mdp}{t}\} \\
\pprobp(s)(t) \ = \ &\probp(s)(t) \cdot \frac{\valueof{\mdp}{t}}{\valueof{\mdp}{s}} \hspace{20mm} 
\pprobp((s,t))(t) \ = \ \frac{\valueof{\mdp}{t}}{\valueof{\mdp}{s}} \hspace{20mm}\\
\pprobp((s,t))(s_\bot) \ = \ & 1 - \frac{\valueof{\mdp}{t}}{\valueof{\mdp}{s}} 
\end{align*}
The transition rewards are carried from $\transition$ to $\ptransition$ in the
natural way, i.e., $r(s,(s,t)) = r((s,t), t) = r(s,t)$.
Finally, we add an infinite chain of fresh states and transitions $s_\bot \to s_\bot^1 \to s_\bot^2 \to \cdots$ 
with rewards suitably defined such that it is losing for the objective $\formula$.
\end{definition}

\smallskip
\begin{lemma}[{\cite[Lemma 13.3 and Lemma 16]{KMST:Transient-arxiv}}]\label{lem:conditioned-construction}
Let $\mdp=\mdptuple$ be an MDP, and let $\formula$ be an objective that is
shift invariant in~$\mdp$.
Let $\pmdp = \tuple{\pstates,\pzstates,\prstates,\ptransition,\pprobp}$ be the conditioned version of~$\mdp$ w.r.t.~$\formula$.
Let $s_0 \in \pstates \cap \states$.
Let $\zstrat \in \zstratset_{\pmdp}$, and note that $\zstrat$ can be transformed to a strategy in~$\mdp$ in a natural way.

We have
$\valueof{\mdp}{s_0} \cdot \probm_{\pmdp,s_0,\zstrat}(\formula) = \probm_{\mdp,s_0,\zstrat}(\formula)$.
In particular, $\valueof{\pmdp}{s_0} = 1$, and, for any $\eps \ge 0$,
strategy~$\sigma$ is $\eps$-optimal in~$\pmdp$ if and only if it is
$\eps \valueof{\mdp}{s_0}$-optimal in~$\mdp$.
Moreover, if $\mdp$ is universally transient, then so is $\pmdp$.
\end{lemma}

\smallskip
\begin{lemma}\label{lem:mr-as-uniform}
Let $\mdp=\mdptuple$ be a countable MDP, and $\formula$ an objective that is
shift invariant in~$\mdp$.
Let $\states' \subseteq \states$ be the subset of states that admit an
almost surely winning strategy for $\formula$.

Assume that from every state $\state_0 \in \states'$
there even exists some
almost surely winning deterministic positional (resp.\ randomized positional) strategy.

Then there exists a deterministic positional (resp.\ randomized positional)
strategy $\zstrat$ such that $\zstrat$ 
is almost surely winning from \emph{every}
state in $\states'$.
\end{lemma}
\begin{proof}
Pick an arbitrary state $\state_0^1 \in \states'$ and an a.s.\ winning
deterministic positional (resp.\ randomized positional)
strategy
$\zstrat^1$ from $\state_0^1$. Let $\states_1 \subseteq \states'$ be the set of
states that $\zstrat_1$ reaches with nonzero probability.
Since $\formula$ is shift invariant, $\zstrat^1$
must be a.s.\ winning from every state in $\states_1$.
For the next round pick a state $\state_0^2 \in \states'\setminus\states_1$
(if one exists) and an a.s.\ winning
deterministic positional (resp.\ randomized positional)
strategy
$\zstrat^2$ from $\state_0^2$ and repeat the construction, etc.

Let $\zstrat$ be the
deterministic positional (resp.\ randomized positional)
strategy that plays like $\zstrat^i$ in all states in
$\states_i \setminus \bigcup_{j < i} \states_j$.
Since $\states' = \bigcup_{i>0} \states_i$, the
deterministic positional (resp.\ randomized positional)
strategy $\zstrat$ is
a.s.\ winning from every state in $\states'$.
\end{proof}

\smallskip
\begin{lemma}\label{as-to-opt}
Consider an objective $\formula$ that is shift invariant in every MDP.
Suppose that, in every MDP (resp.\ in every universally transient MDP),
every state that admits an almost surely winning strategy for
$\formula$ even has some almost surely winning
\emph{positional} strategy $\sigma$ for $\formula$.

Then, in every MDP (resp.\ in every universally transient MDP),
there exists some 
positional strategy $\hat{\sigma}$ for $\formula$ such that
$\hat{\sigma}$ is optimal from \emph{every} state that admits an optimal strategy.
If $\sigma$ can always be chosen as deterministic, then $\hat{\sigma}$
can be chosen as deterministic.
\end{lemma}
\begin{proof}
Let $\mdp=\mdptuple$ be an MDP.
Consider the conditioned version $\pmdp = \tuple{\pstates,\pzstates,\prstates,\ptransition,\pprobp}$
of~$\mdp$ w.r.t.~$\formula$, and let $\eps=0$.
(Recall that $0$-optimal means optimal in general, and a.s.\ winning in the
case where the value of the start state is one.)
By \Cref{lem:conditioned-construction}, if $\mdp$ is universally transient
then $\pmdp$ is universally transient.
Moreover, every state that has an
optimal strategy in $\mdp$ has an a.s.\ winning strategy in $\pmdp$ and
vice-versa.
By applying \Cref{lem:mr-as-uniform} to $\pmdp$,
we obtain some
deterministic positional (resp.\ randomized positional)
strategy $\hat{\zstrat}$ that is a.s.\ winning from \emph{every}
state in $\pmdp$ that admits an a.s.\ winning strategy.
By \Cref{lem:conditioned-construction}, the same
deterministic positional (resp.\ randomized positional)
strategy $\hat{\zstrat}$ is optimal from
\emph{every} state in $\mdp$ that admits an optimal strategy.
\end{proof}

The following \Cref{epsilontooptimal} is a general result concerning shift
invariant objectives. We use this result to lift $\eps$-optimal upper bounds
to optimal upper bounds when the two bounds coincide.
Note that its preconditions are different from those of \Cref{as-to-opt}.

\smallskip
\begin{theorem}(\cite[Theorem 7]{KMST:Transient-arxiv})
\label{epsilontooptimal}
Let $\mdp=\mdptuple$ be a countable MDP, and let $\formula$ be an objective that is shift invariant in~$\mdp$.
Suppose that for every $s \in S$ there exist $\eps$-optimal MD strategies for~$\formula$.
Then:
\begin{enumerate}
\item \label{item:uniformeps} There exist uniformly $\eps$-optimal MD strategies for~$\formula$.
\item \label{item:epstooptimal} There exists an MD strategy that is optimal
  from every state that admits an optimal strategy.
\end{enumerate}
\end{theorem}

\smallskip
\begin{remark}\label{rem:uniformity}
By and large, our results are stated in terms of
$\eps$-optimal or optimal strategies from a given start state $s_0$.
Since all of the objectives that we consider are shift invariant,
it follows that \Cref{epsilontooptimal}
applies everywhere. It would be very repetitive to state a second theorem after every upper bound saying that \Cref{epsilontooptimal} applies and that therefore uniform strategies also exist. We choose therefore to save on repetitions by issuing the blanket statement that \Cref{epsilontooptimal} applies everywhere that you would expect and the relevant uniformity results hold despite not being explicitly stated.
\end{remark}

\section{Objectives and how they are connected}\label{sec:connections}

We study six objectives that are naturally presented as two
sets of three objectives which are dual to each other. They are as follows.
\begin{enumerate}
\item
The $\limsupppobj$ threshold objective aims to maximize the
probability that the $\limsup$ of the \emph{daily payoffs} (the immediate
transition rewards, not their sum or average) is $\ge 0$, i.e., 
$\limsupppobj \defeq
\{\rho \mid \limsup_{n\in \N} \reward(\rho_e(n)) \ge 0\}$.
\item
  Given a monotone decreasing (w.r.t.\ set inclusion) sequence of
  sets of transitions $\{A_i\}_{i \in \N}$, the objective
  $\bigcap_{i \in \N} \G \F \, A_i$
  aims to maximize the probability that, for all $i$,
  the set $A_i$ is visited infinitely often.
\item
The $\liminfppobj$ threshold objective aims to maximize the
probability that the $\liminf$ of the daily payoffs is $\ge 0$, i.e., 
$\liminfppobj \defeq
\{\rho \mid \liminf_{n\in \N} \reward(\rho_e(n)) \ge 0\}$.
\item
  Given a monotone decreasing (w.r.t.\ set inclusion)
  sequence of sets of transitions
  $\{A_i\}_{i \in \N}$, the objective
  $\bigcap_{i \in \N} \F \G \, A_i$
  aims to maximize the probability that, for all $i$,
  eventually only transitions in $A_i$ are visited.
\item
The $\limsupppexp$ objective aims to maximize the
\emph{expectation} of the $\limsup$ of the daily payoffs, i.e., for runs $\rho$ we want to maximize
$\expectval( \limsup_{n\in \N} \reward(\rho_e(n)))$.
\item
The $\liminfppexp$ objective aims to maximize the
\emph{expectation} of the $\liminf$ of the daily payoffs, i.e., for runs $\rho$ we want to maximize
$\expectval( \liminf_{n\in \N} \reward(\rho_e(n)))$.
\end{enumerate}

Note that the 1st objective, $\limsupppobj$, refers only to transition rewards,
while the 2nd objective, $\bigcap_{i \in \N} \G \F \, A_i$, does not refer to
the rewards at all, but instead refers to the sets $A_i$.
While these two objectives are different, we show that they have the same
strategy complexity, via mutual encodings that do not change the structure of
the MDP; cf.~\Cref{lem:nestedbuchiislimsup} and \Cref{lem:limsupisnestedbuchi}.

Similarly, the 3rd objective, $\liminfppobj$, has the same strategy
complexity as the 4th objective $\bigcap_{i \in \N} \F \G \, A_i$,
by \Cref{lem:nestedcobuchiisliminf} and \Cref{lem:liminfisnestedcobuchi}.

\smallskip
\begin{lemma}\label{lem:nestedbuchiislimsup}
Let $\mdp=\mdptuple$ be an MDP with the rewards $r$ as yet unspecified.
For every monotone decreasing (w.r.t.\ set inclusion) sequence of
sets of transitions $\{A_i\}_{i \in \N}$
there exists a reward function $r$ such that  
$\bigcap_{i \in \N} \G \F \, A_i = \limsupppobj$ in
$\mdp$.
\end{lemma}
\begin{proof}
We define the transition based reward function $r$ such that 
\begin{equation*}
r(t) \eqdef 
\begin{cases}
      0 	   & \text{if } \forall i\ t \in A_i \\
      -2^{-i} & \text{if } i= \max \{ i \in \N \mid \ t \in A_i \} \\
      -1        & \text{ otherwise}
\end{cases}
\end{equation*}
Now we show that $\bigcap_{i \in \N} \G \F \, A_i = \limsupppobj$ in $\mdp$.
\begin{enumerate}
\item
  Consider a run $\rho \in \bigcap_{i \in \N} \G \F \, A_i$.
  By definition, $\rho$ always eventually visits transitions in $A_i$
  for every $i$.
  By construction of $r$, $\rho$ therefore always eventually visits
  transitions with
  reward $-2^{-i}$ for every $i$ and thus $\rho \in \limsupppobj$.
\item
  Consider a run $\rho \in \limsupppobj$.
  By construction of $r$, this means that the $\limsup$ of the seen rewards
  must be $0$ (since $r$ never assigns rewards $>0$). 
  There are two cases.
  In the first case, $\rho$ visits infinitely many transitions with reward $0$,
  each of which is trivially in $A_i$ for all $i$, and thus $\rho \in \bigcap_{i \in \N} \G \F \, A_i$.
  In the second case, the sequence of rewards seen by $\rho$ gets arbitrarily close to $0$ from below.
  Therefore $\rho$ must visit infinitely many transitions $t \in A_i$ for
every $i$, and thus also $\rho \in \bigcap_{i \in \N} \G \F \, A_i$. \qedhere
\end{enumerate}
\end{proof}

\begin{lemma}\label{lem:limsupisnestedbuchi}
Let $\mdp=\mdptuple$ be a countable MDP.
There exists a
monotone decreasing (w.r.t.\ set inclusion) sequence of
sets of transitions $\{A_i\}_{i \in \N}$
such that 
$\limsupppobj = \bigcap_{i \in \N} \G \F \, A_i$ in $\mdp$.
\end{lemma}
\begin{proof}
Let $A_i \eqdef \{ t \in \transition \mid \ r(t) \geq -2^{-i} \}$.
The rest of the proof is very similar to \Cref{lem:nestedbuchiislimsup}.
\end{proof}

Symmetrically to the situation for $\bigcap_{i \in \N} \G \F \, A_i$ and $\limsupppobj$, the objectives
$\bigcap_{i \in \N} \F \G \, A_i$ and $\liminfppobj$ can also be mutually encoded.

\smallskip
\begin{lemma}\label{lem:nestedcobuchiisliminf}
Let $\mdp=\mdptuple$ be an MDP with the rewards $r$ as yet unspecified.
For every monotone decreasing (w.r.t.\ set inclusion) sequence of
sets of transitions $\{A_i\}_{i \in \N}$
there exists a reward function $r$ such that  
$\bigcap_{i \in \N} \F \G \, A_i = \liminfppobj$
in $\mdp$.
\end{lemma}
\begin{proof}
Let
\begin{equation*}
r(t) \eqdef 
\begin{cases}
      0 	   & \text{if } \forall i\ t \in A_i \\
      -2^{-i} & \text{if } i= \max \{ i \in \N \mid \ t \in A_i \} \\
      -1        & \text{ otherwise}
\end{cases}
\end{equation*}  
The rest of the proof is very similar to \Cref{lem:nestedbuchiislimsup}.
\end{proof}

\begin{lemma}\label{lem:liminfisnestedcobuchi}
Let $\mdp\mdptuple$ be a countable MDP.
There exists a
monotone decreasing (w.r.t.\ set inclusion) sequence of
sets of transitions $\{A_i\}_{i \in \N}$
such that 
$\liminfppobj = \bigcap_{i \in \N} \F \G \, A_i$ in $\mdp$.
\end{lemma}
\begin{proof}
Let $A_i \eqdef \{ t \in \transition \mid \ r(t) \geq -2^{-i} \}$.
The rest of the proof is very similar to \Cref{lem:nestedbuchiislimsup}.
\end{proof}

Built into the above objectives 1-4 is a certain notion of progress.
E.g., the sequence $-1/2, -1/3, -1/4 \dots$ satisfies $\limsup \ge 0$ and $\liminf \ge 0$.
In order to succeed, a strategy must strive to see better and better sets
$A_i$ (i.e., for larger and larger $i$ and thus larger and larger transition rewards), 
leaving behind those sets (or rewards) that are no longer good enough. 
In order to make this progress happen, there must be some sort of driving
force behind the strategy to make it play better and better as time
progresses.
This can come either
in the form of the MDP's underlying acyclicity or universal transience,
or via the strategy's step counter, 
or we can create it ourselves by suitably modifying the MDP or 
by exploiting the finite branching degree of the MDP in order to define a
function that measures the distance from the start state.
Many of the strategies defined in the following sections exploit this
intuition.

On the other hand, the objectives become a bit simpler if the transition rewards are restricted to the integers.
In that case, the objective $\limsupppobj$ just corresponds to the objective to see
transitions with reward $\ge 0$ infinitely often, i.e.,
$\always\eventually\, A$ where $A = \{ t \in \transition \mid \ r(t) \ge 0 \}$.
The latter objective $\always\eventually\, A$ is also called a B\"uchi objective
\cite{KMST:ICALP2019,KMST2020c,ModCheckHB18,CGP:book}.
Similarly, for integer rewards, $\liminfppobj$ corresponds to the co-B\"uchi
objective $\eventually\always\, A$.


\paragraph{Relation to the expected payoff.}

We now show how the (strategy complexity of)
expected payoff objectives $\limsupppexp$ and $\liminfppexp$ relate to the threshold
objectives.

The expected payoff objectives are more natural to define in the context of state based rewards. This is because 
we want to be able to refer to the value of a state and compare that value to
the reward of that state itself. This is less important for the threshold
objectives, since the rewards and the values are measured on two different
scales as it were.
It is trivial that transition based rewards and state based
rewards can be encoded into each other; cf.~\Cref{app:states-vs-transitions}
for a formal treatment.

The following \Cref{thm:explimsuptothreshold} shows that the strategy
complexity of optimal strategies for the \emph{expected} payoff objectives $\limsupppexp$ (resp.\ $\liminfppexp$)
in countable MDPs is upper-bounded by the strategy complexity of
optimal strategies for the \emph{threshold} payoff objectives $\limsupppobj$ (resp.\ $\liminfppobj$).
Following \cite{Sudderth:2020}[Sections 4 and 5], the main idea is that for
each MDP where
optimal strategies for $\limsupppexp$ (resp.\ $\liminfppexp$) exist, one can construct a
derived MDP with a $\limsupppobj$ (resp.\ $\liminfppobj$) objective,
and then carry optimal strategies from the derived MDP back to the original MDP.

\Cref{thm:explimsuptothreshold} will allow us to obtain 
upper bounds on the strategy complexity of optimal strategies for 
$\limsupppexp$ (resp.\ $\liminfppexp$) objectives
for free from the results on the strategy complexity of optimal
$\limsupppobj$ and $\liminfppobj$ strategies; cf.~\Cref{sec:expected}.

First we need to recall the notion of \emph{equalizing strategies}.
There are many equivalent ways of defining equalizing strategies.
Below we choose the one most appropriate for our application.
See \cite{MaitraSudderth:DiscreteGambling}[Theorem 4.7.7] for
more characterizations of equalizing strategies.
In this context we consider $\limsup$ (resp.\ $\liminf$) objectives
about the infinite sequence of daily payoffs, where the daily payoffs
here are given as state rewards (rather than transition rewards).
Intuitively, a strategy is equalizing if the $\limsup$ (resp.\ $\liminf$)
of the \emph{difference} between the value of the current state
and the reward of the current state converges to zero almost surely. 

\smallskip
\begin{definition}[see {\cite{DubbinsSavage:2014}[Theorem 3.7.2]}, {\cite{sudderth1983gambling}[Lemma 5]}]
\label{def:equalizing}
Let $\mdp = \mdptupler$ be a countable MDP with initial state $\state_0$.
Define sets
$A_i \eqdef \{ s \in S \mid r(s) \geq \valueof{\mdp,\limsupppexp}{s} - 2^{-i} \}$
(resp.\ $A_i \eqdef \{ s \in S \mid r(s) \geq \valueof{\mdp,\liminfppexp}{s} - 2^{-i} \}$)
for all $i \in \N$.  
A strategy $\zstrat$ from $\state_0$ is \emph{equalizing}
for $\limsupppexp$ (resp.\ $\liminfppexp$)
if and only if
\[
\probm_{\mdp,\state_0,\zstrat}\left(\bigcap_{i \in \N}\G \F \, A_i\right) = 1
\quad\quad
\left(
\mbox{resp.\ }
\probm_{\mdp,\state_0,\zstrat}\left(\bigcap_{i \in \N}\F \G \, A_i\right) = 1
\right).  
\]
\end{definition}

\begin{theorem}\label{thm:explimsuptothreshold}
In countable MDPs $\mdp = \mdptupler$, the strategy complexity of
optimal strategies for $\limsupppexp$ (resp.\ $\liminfppexp$),
where they exist, is upper-bounded by 
the strategy complexity of optimal strategies for
$\limsupppobj$ (resp.\ $\liminfppobj$).
\end{theorem}
\begin{proof}
For clarity, we present the proof for $\limsupppexp$ here.
(The corresponding proof for $\liminfppexp$ is very similar; see below).  

Let $\mdp = (\states,\zstates,\rstates,\transition,\probp,r)$
be a countable MDP with bounded state based rewards where optimal strategies
from $\state_0$ exist for $\limsupppexp$.
Let $\zstrat$ be an optimal strategy for $\limsupppexp$ from $\state_0$.
Let $\mdp^r$ be the sub-MDP of $\mdp$ composed only of those states and transitions used by $\sigma$ with positive probability.
In particular, $\zstrat$ is optimal for $\limsupppexp$ from $\state_0$ also in $\mdp^r$.

Since $\zstrat$ is optimal, all controlled transitions $s \to s'$ in $\mdp^r$
are such that $\valueof{\mdp^r, \limsupppexp}{s} = \valueof{\mdp^r, \limsupppexp}{s'}$.
Hence, all strategies in $\mdp^r$ are value preserving
(also called \emph{thrifty} in \cite{DubbinsSavage:2014,Sudderth:2020}). 
Let $A_i \eqdef \{ s \in S \mid r(s) \geq \valueof{\mdp^r,\limsupppexp}{s} -
2^{-i} \}$ for all $i \in \N$.
By \Cref{def:equalizing}, 
$\zstrat$ is equalizing if and only if 
$\probm_{\mdp^r,\state_0,\zstrat}(\bigcap_{i \in \N}\G \F \, A_i) = 1$.
By \cite{MaitraSudderth:DiscreteGambling}[Theorem 4.7.2],
strategies are optimal if and only if they are equalizing and value preserving
(aka thrifty).
I.e., strategies are optimal for 
$\limsupppexp$ in $\mdp^r$ if and only if they are optimal for 
$\bigcap_{i \in \N} \G \F\, A_i$.

We now define a new state based reward function $u$ as follows:
\[
u(s) \eqdef 
\begin{cases}
-1 & \text{if } \forall i \in \N, s \notin A_i, \\
0  & \text{if } \forall i \in \N, s \in A_i , \\
-2^{- \max \{ i \in \N \,\mid\, s \in A_i \} } & \text{otherwise}				
\end{cases} 
\]
Now we let $\mdp^u$ be the MDP that is like $\mdp^r$ except that it uses
reward function $u$ instead of reward function $r$.
So $\mdp^u$ is a sub-MDP of $\mdp$, but with a different reward function.
We claim that a strategy $\tau$ from $\state_0$
is optimal for $\limsupppexp$ in $\mdp^r$ if and only if it is optimal for
$\limsupppobj$ in $\mdp^u$.
(In particular this claim implies that optimal strategies for $\limsupppobj$ from
$\state_0$ exist in $\mdp^u$.)
Notice that strategies defined on $\mdp^r$ are also defined on $\mdp^u$ and
vice-versa, since the only difference between the two MDPs is the reward function. 
The claim follows immediately from the observation that runs satisfy
$\limsupppobj$ in $\mdp^u$ if and only if they satisfy $\bigcap_{i \in \N} \G \F\, A_i$.
(This is the state based rewards version of \Cref{lem:nestedbuchiislimsup}.)
Hence a strategy $\tau$ from $\state_0$ is optimal for $\limsupppobj$ in $\mdp^u$
if and only if it is optimal for $\bigcap_{i \in \N} \G \F\, A_i$ in $\mdp^u$
if and only if it is optimal for $\limsupppexp$ in $\mdp^r$
if and only if it is optimal for $\limsupppexp$ in $\mdp$.

Thus, we have reduced the problem of finding optimal strategies for
$\limsupppexp$ to the problem of finding optimal strategies for $\limsupppobj$,
and the upper bound on the strategy complexity follows.

To obtain the corresponding proof for $\liminfppexp$, simply replace $\limsup$ with $\liminf$ and 
$\bigcap_{i \in \N}\G \F \, A_i$ with $\bigcap_{i \in \N} \F \G \, A_i$
and \cite{MaitraSudderth:DiscreteGambling}[Theorem 4.7.2] with \cite{sudderth1983gambling}[Lemma 5].
\end{proof}

\begin{figure}
\begin{center}
    \begin{tikzpicture}
    
    \node[draw] (S1) at (0,0){};
    \node[draw] (S2) at (2,0){};
    \node[draw] (S3) at (4,0){};
    \node[draw] (S4) at (6,0){};
    
    \node[draw] (T1) at (0,-1){};
    \node[draw] (T2) at (2,-1){};
    \node[draw] (T3) at (4,-1){};
    \node[draw] (T4) at (6,-1){};
    
    \node (I1) at (7,0){};
    
    \draw[->,>=latex] (S1) edge node[above, midway]{$+0$}  (S2)
    (S2) edge node[above, midway]{$+0$}  (S3)
    (S1) edge  (T1)
    (S2) edge  (T2)
    (S3) edge  (T3)
    (S4) edge  (T4)
    (T1) edge[loop below] node[below, midway]{$+1-2^{-1}$}  (T1)
    (T2) edge[loop below] node[below, midway]{$+1-2^{-2}$}  (T2)
    (T3) edge[loop below] node[below, midway]{$+1-2^{-3}$}  (T3)
    (T4) edge[loop below] node[below, midway]{$+1-2^{-i}$}  (T4);
    
    \draw[->,>=latex,dotted,thick] (S3) edge (S4)
    (S4) edge (I1);

    \end{tikzpicture}
\end{center}  

\caption{
All strategies are optimal for $\liminfppobj$ and $\limsupppobj$, yet there
are no optimal
strategies for $\liminfppexp$ or $\limsupppexp$.
Every threshold $< 1$ has an optimal strategy, yet thresholds $\geq 1$ have no optimal strategies.
}
\label{incomparable}
\end{figure}

\begin{remark}
The reduction in \Cref{thm:explimsuptothreshold} does not work in the reverse direction.
I.e., we cannot reduce the problem of finding optimal strategies for
$\limsupppobj$ (resp.\ $\liminfppobj$) to the problem of finding optimal
strategies for  
$\limsupppexp$ (resp.\ $\liminfppexp$).
\Cref{incomparable} gives an example of an MDP where optimal strategies exist
for
$\limsupppobj$ (resp.\ $\liminfppobj$),
but do not exist for $\limsupppexp$ (resp.\ $\liminfppexp$).
\end{remark}

\section{Strategy complexity of \texorpdfstring{$\bigcap_{i \in \N} \G \F A_i$
    and $\limsupppobj$}{lim sup}}\label{sec:nestedbuchi}

In light of \Cref{lem:nestedbuchiislimsup} and \Cref{lem:limsupisnestedbuchi}, we present strategy complexity
results on $\bigcap_{i \in \N} \G \F A_i$ and $\limsupppobj$ together and interchangeably. 

First we show that, for these objectives, the branching degree of the MDP does
not matter (unlike for $\liminfppobj$; cf.~\Cref{sec:nestedcobuchi}).
Note that positional strategies are a special case of finite-memory
strategies, i.e., with only one memory mode. Thus the following lemma also allows to
carry positional strategies between MDPs.

\smallskip
\begin{lemma}\label{lem:branchingreplacement}
Consider an infinitely branching MDP $\mdp=\mdptuple$ and the $\limsupppobj$
objective. There exists a corresponding binary branching MDP
$\mdp' = \tuple{\states',\zstates',\rstates',\transition',\probp',r'}$
with $\states\subseteq \states'$, $\zstates\subseteq \zstates'$,
$\rstates\subseteq \rstates'$
such that 
\begin{enumerate}
\item
  $\mdp'$ is acyclic (resp.\ universally transient) iff
  $\mdp$ is acyclic (resp.\ universally transient).
\item
The value of states is preserved, i.e.,
\[
\forall\state \in \states\ 
\valueof{\mdp,\limsupppobj}{\state}=\valueof{\mdp',\limsupppobj}{\state}.
\]
\item
Finite memory randomized (resp.\ deterministic) strategies $\zstrat'$ in $\mdp'$
can be carried back to $\mdp$. I.e., for
every Rand(F) (resp.\ Det(F)) strategy
$\zstrat'$ from some $\state_0 \in \states$ in $\mdp'$ there exists a
corresponding strategy $\zstrat$ from $\state_0$ in $\mdp$ with the same memory and
randomization such that
\[
\probm_{\mdp, \state_0, \zstrat}(\limsupppobj) \ge
\probm_{\mdp', \state_0, \zstrat'}(\limsupppobj).
\]
\end{enumerate}
\end{lemma}
\begin{proof}
We construct $\mdp'$ from $\mdp$ by replacing every infinitely branching controlled state
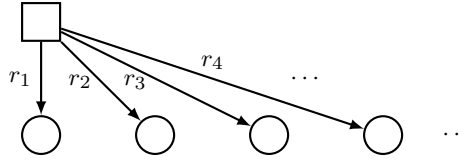
\begin{center}
    \begin{tikzpicture}
    
    \node[draw, , minimum height=5mm, minimum width=5mm] (S1) at (0,0){};
    
	\node (I1) at (3.5,-0.75) {\textbf{$\cdots$}};    
	\node (I2) at (5.5,-1.5) {\textbf{$\cdots$}};    
    
    \node[draw, circle, minimum height=5mm] (T1) at (0,-1.5){};
    \node[draw, circle, minimum height=5mm] (T2) at (1.5,-1.5){};
    \node[draw, circle, minimum height=5mm] (T3) at (3,-1.5){};
    \node[draw, circle, minimum height=5mm] (T4) at (4.5,-1.5){};
      
     \draw[->,>=latex] (S1) edge node[left, midway]{$r_1$}  (T1)
     (S1) edge node[left, midway]{$r_2$} (T2)
     (S1) edge node[left, midway, yshift=-1pt]{$r_3$} (T3)
     (S1) edge node[above, midway]{$r_4$} (T4);

    \end{tikzpicture}
\end{center}    

by

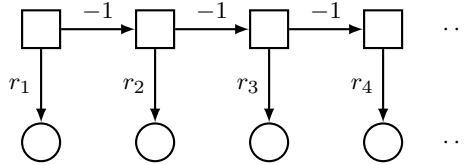
\begin{center}
    \begin{tikzpicture}
    
    \node[draw, , minimum height=5mm, minimum width=5mm] (S1) at (0,0){};
    \node[draw, , minimum height=5mm, minimum width=5mm] (S2) at (1.5,0){};
    \node[draw, , minimum height=5mm, minimum width=5mm] (S3) at (3,0){};
    \node[draw, , minimum height=5mm, minimum width=5mm] (S4) at (4.5,0){};
    
	\node (I1) at (5.5,0) {\textbf{$\cdots$}};    
	\node (I2) at (5.5,-1.5) {\textbf{$\cdots$}};    
    
    \node[draw, circle, minimum height=5mm] (T1) at (0,-1.5){};
    \node[draw, circle, minimum height=5mm] (T2) at (1.5,-1.5){};
    \node[draw, circle, minimum height=5mm] (T3) at (3,-1.5){};
    \node[draw, circle, minimum height=5mm] (T4) at (4.5,-1.5){};
      
     \draw[->,>=latex] (S1) edge node[left, midway]{$r_1$} (T1)
     (S2) edge node[left, midway]{$r_2$} (T2)
     (S3) edge node[left, midway]{$r_3$} (T3)
     (S4) edge node[left, midway]{$r_4$}(T4)
     (S1) edge node[above, midway]{$-1$} (S2)
     (S2) edge node[above, midway]{$-1$} (S3)
     (S3) edge node[above, midway]{$-1$} (S4);

    \end{tikzpicture}
\end{center}    
The above construction is also called the \emph{ladder gadget}.
It includes negative rewards $-1$ on the ladder
transitions (top row) to ensure that strategies must leave the ladder
eventually (since runs that stay on the ladder do not satisfy $\limsupppobj$).

Infinitely branching random states with $x \step{p_i}{} y_i$ for all $i \in \N$
are replaced by a gadget
$x \step{1}{} z_1, z_i \step{1-p_i'} z_{i+1}, z_i \step{p_i'} y_i$ for all
$i \in \N$, with fresh random states $z_i$ and suitably adjusted probabilities
$p_i'$ to ensure that the gadget
is left at state $y_i$ with probability $p_i$, i.e.,
$p_i' = p_i/(\prod_{j=1}^{i-1}(1-p_j'))$.
Almost all the new auxiliary transitions have reward $-1$, except for the last
step to $y_i$, which is given the same reward as the original transition 
$x \step{p_i}{} y_i$.

Finally, each non-binary finite branching is replaced by a binary tree
where the new auxiliary transitions have reward $-1$
and the last step has the original reward.

Thus $\mdp'$ has branching degree $\le 2$.

Item 1.\ follows from the fact that the
auxiliary gadgets in $\mdp'$ are acyclic.

Towards item 2., note that for every strategy in $\mdp'$
from some state $\state \in \states$
there is a corresponding
strategy in $\mdp$ from $\state$ that attains \emph{at least} as much for $\limsupppobj$, and
vice-versa. 
This is because runs that stay on the ladder gadget forever do not satisfy
$\limsupppobj$
and thus the new additions in $\mdp'$ confer no advantage.
So the corresponding strategy just imitates the outcome of the branching in
the other MDP.
Thus
\[
\forall\state \in \states\  \valueof{\mdp,\limsupppobj}{\state}=\valueof{\mdp',\limsupppobj}{\state}.
\]
Towards item 3., we show that finite memory randomized (resp.\ deterministic) strategies can be carried
from $\mdp'$ to $\mdp$ with the same memory and randomization.

Let $\state_0 \in \states$ be the initial state and let $\memory$ be the finite set of memory modes used by $\zstrat'$.
Consider some state $x$ in $\mdp$ that is infinitely branching to states $y_i$ for
$i \in \N$, and its associated ladder gadget in $\mdp'$.
Whenever a run in $\mdp'$ according to $\zstrat'$ reaches $x$ in some
memory mode $\alpha \in \memory$ there exists
a combined probability distribution $d$ over the set of outcomes.
I.e., $d(y_i,\alpha_j)$ is the probability that we exit the gadget at state
$y_i$ in memory mode $\alpha_j$.
It is also possible that we do not exit the gadget at all, but in this case we
do not care about the memory modes, since this run will not satisfy
$\limsupppobj$ anyway. Let $p \eqdef 1-\sum_{i,j} d(y_i,\alpha_j)$ be the
probability of not exiting the gadget.

We now define the corresponding strategy $\zstrat$ in $\mdp$ with the same
memory as $\zstrat'$.
Whenever $\zstrat$ is in state
$x$ in memory mode $\alpha \in \memory$ then it plays as follows.

If $x$ is a random state then $p=0$, i.e., we will exit the gadget almost
surely. Moreover, $\zstrat$ does not select the next state, but can only update
its memory. By construction of $\mdp'$ we will go to state $y_i$ in $\mdp$
with the same probability as in $\mdp'$. Upon arriving at some state $y_i$,
$\zstrat$ will update its memory to match the distribution $d$, i.e.,
it will set its memory to mode $\alpha_j$ with probability
$d(y_i,\alpha_j)/(\sum_j d(y_i,\alpha_j))$.

If $x$ is a controlled state then $\zstrat$ selects the successor state $y_i$
and new memory mode $\alpha_j$ as follows.
It selects $y_1$ and $\alpha_1$ with probability $d(y_1,\alpha_1)+p$
(special case),
and otherwise selects $y_i$ and $\alpha_j$ with probability $d(y_i,\alpha_j)$
(normal case).
The special case is needed to obtain a distribution in case $p>0$.

The construction for states that are finitely, but non-binary, branching
in $\mdp$ is similar but easier (since the gadget will surely be exited).

In all other states, $\zstrat$ does the same in $\mdp$ as $\zstrat'$ in
$\mdp'$.

We observe that $\zstrat$ is deterministic iff $\zstrat'$ is deterministic.

Finally, we observe that $\zstrat$ attains at least as much for the objective $\limsupppobj$
in $\mdp$ as $\zstrat'$ in $\mdp'$.
The additionally seen rewards of $-1$ in the gadgets in $\mdp'$ make no difference for
$\limsupppobj$ if the gadget is exited.
The runs that stay forever in some gadget in $\mdp'$ do not satisfy
$\limsupppobj$ anyway, and thus the different extra corresponding runs in $\mdp$
(by the extra probability $+p$) via 
the special case definition above cannot do worse.
All other runs in $\mdp'$ are mimicked by runs in $\mdp$
(w.r.t.\ seen states and memory modes) and the probabilities coincide.
Thus $
\probm_{\mdp, \state_0, \zstrat}(\limsupppobj) \ge
\probm_{\mdp', \state_0, \zstrat'}(\limsupppobj).
$
\end{proof}

Note that one cannot encode infinitely branching MDPs into finitely
branching MDPs in general.
E.g., even for $\liminfppobj$ it \emph{does} make a difference whether the MDP is
finitely branching or infinitely branching; cf.~\Cref{sec:nestedcobuchi}
and \Cref{tab:liminf}.

\bigskip
\begin{remark}[Integer rewards vs.\ real rewards]\label{rem:integer-rewards}
If the transition rewards are restricted to the integers, then 
the objective $\limsupppobj$ just corresponds to the objective to visit
transitions with reward $\ge 0$ infinitely often, i.e.,
$\always\eventually\, A$ where $A \eqdef \{ t \in \transition \mid \ r(t) \ge 0 \}$.
Thus, instead of $\bigcap_{i \in \N} \G \F \, A_i$, it suffices to consider the
simpler B\"uchi objective $\always\eventually\, A$.
The strategy complexity of the B\"uchi objective in countable MDPs has been
studied in \cite{KMSW2017,KMST:ICALP2019,KMST2020c}.

Optimal strategies for B\"uchi, where they exist, can be chosen as
Det(Positional), but this does not carry over to $\bigcap_{i \in \N} \G \F \, A_i$,
where even Det(F) is not sufficient (\Cref{thm:limsuplowerdetf}),
and instead optimal strategies can be chosen as Det(SC) or Rand(Positional); cf.~\Cref{tab:limsup}.

$\eps$-optimal strategies for B\"uchi can be chosen as Det(SC + 1-bit),
but not as Rand(SC) or Rand(F).
The lower bounds trivially carry over to $\bigcap_{i \in \N} \G \F \, A_i$.
The Det(SC + 1-bit) upper bound also carries over, but it requires a more complex
proof (\Cref{cor:epsGF}).
\end{remark}


\subsection{Upper Bounds}

We choose to present the upper bounds in terms of
$\bigcap_{i \in \N} \G \F \, A_i$,
because it is more natural.
First we present the results for optimal strategies,
and then those for $\eps$-optimal strategies.

\subsubsection{Strategy Complexity of Optimal Strategies}

\begin{theorem}\label{thm:optnestedbuchiunivtrans}
  Let $\mdp=\mdptuple$ be a
  universally transient countable MDP
  with initial state $\state_{0}$ with a 
  $\bigcap_{i \in \N} \G \F \, A_i$ objective.
  If there exists an almost surely winning strategy from $s_0$,
  then there also exists an almost surely winning deterministic positional strategy.
\end{theorem}
\begin{proof}[Proof outline]
  W.l.o.g.\ we can assume that $\mdp$ is finitely branching, and it
  suffices to consider a sub-MDP where every state admits an almost surely
  winning strategy.
  A run satisfies $\bigcap_{i \in \N} \G \F \, A_i$
  iff it visits infinitely many transitions in $A_i$ for every $i \in \N$.
  We partition the state space into infinitely many finite regions
  in the shape of expanding rings around the initial state $\state_0$,
  where membership in each ring
  is defined via certain lower and upper bounds on the length of the shortest path from
  $\state_0$.
  Since $\bigcap_{i \in \N} \G \F \, A_i \subseteq \eventually \, A_i$ for all $i$,
  every state has value one for the reachability objective $\eventually \, A_i$.
  Since the rings are chosen sufficiently large,
  we can fix a deterministic positional strategy inside the $i$-th ring
  that ensures that $A_i$ is visited with probability $\geq 1/2$ already inside the
  $i$-th ring whenever the $i$-th ring is entered from the previous $(i-1)$-th
  ring.
  Since our MDP is universally transient and every ring is finite, it follows
  that almost all runs (except for a nullset) eventually reach the $i$-th ring for every $i$.
  Hence, for every $i \in \N$ we obtain a $\ge 1/2$ chance of visiting $A_i$.
  Since the sets $A_i$ form a monotone decreasing chain, we
  satisfy $\bigcap_{i \in \N} \G \F \, A_i$ almost surely.
\end{proof}
  
\begin{proof}
W.l.o.g.\ we can assume that $\mdp$ is finitely branching by \Cref{lem:branchingreplacement}.
Assume furthermore that there exists an almost surely winning strategy
$\zstrat'$ from $\state_{0}$ in $\mdp$.
Let $\states' \subseteq \states$ be the subset of states that are visited
under $\zstrat'$. By restricting $\mdp$ to $\states'$, we obtain a sub-MDP
$\mdp'$ where all states are almost surely winning for $\bigcap_{i \in \N} \G \F \, A_i$.

By construction of $\mdp'$, we have
$\forall \state \in \states'\, \valueof{\mdp',\bigcap_{i \in \N} \G \F \, A_i}{\state} = 1$.
Note that $\G \F \, A_{i} \subseteq \F\, A_{i}$ for all $i$.
Hence we have that $\valueof{\mdp',\F\, A_{i}}{\state} \geq \valueof{\mdp',\bigcap_{i \in \N} \G \F \, A_i}{\state} = 1$ for all $s \in S'$ and all $i \in \N$.
Let $d : \states \to \N$ be a distance function where $d(\state)$ is the length of the shortest path from $\state_{0}$ to $\state$.
Let $\bubble{n}{\state_{0}} \eqdef \{\state \in \states'\mid d(\state) \leq n \}$.

We inductively construct
an increasing sequence of numbers $0 = n_0 < n_1 < n_2 < \dots$ and
a sequence of MDPs $\mdp_i$ for $i=0,1,2,\dots$ that are derived from $\mdp'$
by fixing choices in larger and larger finite subspaces $\bubble{n_i}{\state_{0}}$.
These will satisfy the following properties:
\begin{itemize}
\item
  For all $i \ge 0$, $\mdp_i$ is universally transient and
  all states in $\mdp_i$ are almost surely winning for $\bigcap_{i \in \N} \G \F \, A_i$.
\item
  For all $i,j \ge 1$ we have the following.
  
Let $\states_i \eqdef \bubble{n_i}{\state_{0}}$, 
$X_i \eqdef \states_i \setminus \states_{i-1}$ and
let
$A_i^j \eqdef \{(\state_1 \transition \state_2) \in A_i \mid \state_1,\state_2 \in X_j\}$
be the transitions in $A_i$ that happen completely inside the subspace $X_j$.

In $\mdp_{i+1}$, under all strategies, for all $j \le i$,
from any state $\state \in \states_j$ the probability of seeing a transition in
$A_{j+1}^{j+1}$ is $\ge 1/2$.
(This property will hold under all strategies, because in $\mdp_{i+1}$ the relevant
choices will be fixed already.)
Formally, for all $j \le i$ we have
\begin{equation}\label{eq:lem:ppalmostupper-1}
  \forall \tau\ \forall \state \in \states_j
  \, \probm_{\mdp_{i+1}, \state, \tau}(\F\, A_{j+1}^{j+1}) \ge 1/2
\end{equation}
\end{itemize}

In the base case $i=0$ we have $\mdp_0 \eqdef \mdp'$, $n_0 \eqdef 0$ and
$\states_0 = \{\state_0\}$.
The first property holds by construction of $\mdp_0 = \mdp'$ and the second is
vacuously true (since it assumes $i \ge 1$).

Induction step $i \transition i+1$:
We define $\mdp_{i+1}$ inductively as follows.

Consider the MDP $\mdp_i$.
Since $\mdp_i$ (like $\mdp$) is finitely branching, 
the set $\states_{i}$ is finite.
By induction hypothesis, all states in $\mdp_i$ are almost surely winning
for $\bigcap_{i \in \N} \G \F \, A_i$.
Moreover, since $\mdp_i$ (like $\mdp$) is universally transient,
under all strategies,
the set of runs which stays within a finite set forever is a nullset.
In particular this holds for the finite set $\states_i = \cup_{j \le i} X_j$. Therefore
$ \G \F \, A_{i+1}$
is equal to
$\G \F(A_{i+1} \setminus \cup_{j \le i} A_{i+1}^j)$
up to a nullset (by universal transience) and we have
$\G \F(A_{i+1} \setminus \cup_{j \le i} A_{i+1}^j)
\subseteq
\F (A_{i+1} \setminus \cup_{j \le i} A_{i+1}^j)
$.
Thus every state in $\mdp_i$ is also almost surely
winning for the reachability objective $\F(A_{i+1} \setminus \cup_{j \le i} A_{i+1}^j)$.

In countable MDPs there exist uniform $\eps$-optimal deterministic positional strategies for
reachability \cite{Ornstein:AMS1969}.
Since all states in $\mdp_i$ have value $1$ for the reachability objective
$\F(A_{i+1} \setminus \cup_{j \le i} A_{i+1}^j)$,
we can pick a uniform deterministic positional strategy $\zstrat_{i+1}$
that is $1/4$-optimal from every $\state \in \states_{i}$, i.e.,
\[
\forall \state \in \states_{i}\,
\probm_{\mdp_i,\state,\zstrat_{i+1}}(\F(A_{i+1} \setminus \cup_{j \le i} A_{i+1}^j))
\ge 1 - \dfrac{1}{4} = \dfrac{3}{4}.
\]
Since $\F = \bigcup_{k \in \N} \F^{\leq k}$,
by continuity of measures from below we have 
\[
\forall \state \in \states_{i}\, \lim_{k \to \infty}
\probm_{\mdp_i,\state,\zstrat_{i+1}}(\F^{\le k}(A_{i+1} \setminus \cup_{j \le i} A_{i+1}^j)) \ge \dfrac{3}{4}.
\]
Thus for all states $\state \in \states_{i}$ there is a $k(\state)$
s.t.\ $\probm_{\mdp_i,\state,\zstrat_{i+1}}(\F^{\le k(\state)}(A_{i+1} \setminus \cup_{j \le i} A_{i+1}^j)) \ge 1/2$.
Since $\states_i$ is finite,
$k_{i+1} \eqdef \max\{k(\state) \mid \state \in \states_i\}$ is finite.
Hence
\[
\forall \state \in \states_{i}\,
\probm_{\mdp_i,\state,\zstrat_{i+1}}(\F^{\le k_{i+1}}(A_{i+1} \setminus \cup_{j \le i} A_{i+1}^j)) \ge 1/2.
\]
Let $n_{i+1} \eqdef n_i + k_{i+1}$.
All transitions in $A_{i+1}\setminus \cup_{j \le i} A_{i+1}^j$ that are reachable from
$\states_i$ in $\le k_{i+1}$ steps are contained in $A_{i+1}^{i+1}$
by the definitions of $n_{i+1}$, $X_{i+1}$ and $A_{i+1}^{i+1}$.
Thus
\begin{equation}\label{eq:T-iplus-iplus}
\forall \state \in \states_{i}\,
\probm_{\mdp_i,\state,\zstrat_{i+1}}(\F\, A_{i+1}^{i+1}) \ge 1/2.
\end{equation}
We obtain $\mdp_{i+1}$ from $\mdp_i$ by fixing the (remaining) choices inside
$\states_{i+1} = \bubble{n_{i+1}}{s_{0}}$ according to $\zstrat_{i+1}$.
From \eqref{eq:T-iplus-iplus} and the definition of $\mdp_{i+1}$
we obtain \eqref{eq:lem:ppalmostupper-1} for the case of $j=i$.
The other cases of \eqref{eq:lem:ppalmostupper-1} where $j < i$ follow from the induction hypothesis.

$\mdp_{i+1}$ is universally transient, since $\mdp_i$ is universally transient.
Thus only a nullset of runs stays in the finite set $\states_{i+1}$ forever.
Thus, since $\bigcap_{i \in \N} \G \F \, A_i$ is shift invariant, even in $\mdp_{i+1}$ all
states are almost surely winning for $\bigcap_{i \in \N} \G \F \, A_i$.
This concludes the induction step.

Let $\zstrat^{*}$ be the deterministic positional strategy on $\mdp'$ that is compatible
with the MDPs $\mdp_i$ for all $i$.
I.e.\ $\zstrat^{*}$ plays like $\zstrat_{i}$ on $X_i$ for all $i$.
We'll show that $\zstrat^{*}$ is almost surely winning for $\bigcap_{i \in \N} \G \F \, A_i$ from
$\state_{0}$ in $\mdp'$. To this end, we show that the dual objective
$\neg \bigcap_{i \in \N} \G \F \, A_i$ is a nullset under $\zstrat^{*}$.

We know that
\begin{equation}\label{eq:neg-limsup}
\neg\bigcap_{i \in \N} \G \F \, A_i = \bigcup_{i \in \N} \F\G\, \neg A_{i}.
\end{equation}

By the definition of the sets $A_i^j$ we have
$\bigcup_{j \ge i} A_j^j \subseteq A_i$ and thus
$\neg A_i \subseteq \bigcap_{j \ge i} \neg A_j^j$.
Moreover, we cannot reach $A_j^j$ in fewer than $j$ steps from $\state_0$.
Hence, for any $k \in \N$ we have
\begin{equation}\label{eq:remove-k-prefix}
\F^{\leq k} \G\,\neg A_{i}
\subseteq
\F^{\leq k} \G\, \left(\cap_{j \ge i} \neg A_j^j\right)
\subseteq
\G\, \left(\cap_{j > \max\{ i, k \}} \neg A_j^j\right).
\end{equation}

Since $\mdp'$ (like $\mdp$) is universally transient, only a nullset of runs stays in a finite set forever, in particular each of the finite sets $\states_m$
for every $m \in \N$.
Thus every run from $\state_0$ must eventually reach the set $X_m$ for each
$m \in \N$.
This holds under every strategy, and thus in particular
under the strategy $\zstrat^*$, i.e., 
\begin{equation}\label{eq:acyc-reach-Xi}
\probm_{\mdp', s_{0}, \zstrat^*}(\cap_{m \in \N}\, \F\, X_m) = 1
\end{equation}

By the construction of $\zstrat^{*}$, for any $i,k \in \N$ we obtain the following
\begin{equation}\label{eq:always-not-Tjj}
\begin{array}{ll}
\probm_{\mdp', s_{0}, \zstrat^{*}}(\G\, (\cap_{j > \max\{ i, k \}} \neg A_j^j))\\
=
\probm_{\mdp', s_{0}, \zstrat^{*}}(\G\, (\cap_{j > \max\{ i, k \}}
\neg A_j^j) \cap \cap_{m \in \N} \F\, X_m) & \text{by \eqref{eq:acyc-reach-Xi}}\\
\le
\probm_{\mdp', s_{0}, \zstrat^{*}}(
\cap_{j > \max\{ i, k \}} \F (X_{j-1} \cap \G \neg A_j^j))
& \text{set inclusion}\\
\le
\prod_{j > \max\{ i, k \}} \max_{\state \in X_{j-1}}\probm_{\mdp',\state, \zstrat^{*}}(\G\neg A_j^j)
& \text{$X_{j-1}$ finite}
\\
=
\prod_{j > \max\{ i, k \}} \max_{\state \in X_{j-1}}\probm_{\mdp_j,  \state, \zstrat_j}(\neg\F\, A_j^j)
& \text{def.\ $\zstrat^{*}$, $\G$}
\\
\leq
\prod_{j > \max\{ i, k \}} 1/2 & \text{by \eqref{eq:T-iplus-iplus}}\\
= 0.
\end{array}
\end{equation}
Hence for all $i,k \in \N$,
\begin{equation}\label{eq:avoid-Tu-nullset}
\begin{array}{ll}
\probm_{\mdp', \state_{0}, \zstrat^{*}}(\F\G\, \neg A_{i})\\
= \lim_{k \to \infty} \probm_{\mdp', \state_{0}, \zstrat^{*}}(\F^{\le
  k}\G\, \neg A_{i}) & \text{cont.\ measures}\\
\le \lim_{k \to \infty} \probm_{\mdp', \state_{0}, \zstrat^{*}}(\G\,
(\bigcap_{j > \max\{ i, k \}} \neg A_j^j)) & \text{by \eqref{eq:remove-k-prefix}}\\
= 0 & \text{by \eqref{eq:always-not-Tjj}}
\end{array}
\end{equation}
Finally, we show that $\zstrat^*$ is almost surely winning for
$\bigcap_{i \in \N} \G \F \, A_i$
from $\state_0$ in $\mdp$.
\begin{align*}
 & \probm_{\mdp, \state_{0}, \zstrat^{*}}(\bigcap_{i \in \N} \G \F \, A_i) \\
 & = \probm_{\mdp', \state_{0}, \zstrat^{*}}(\bigcap_{i \in \N} \G \F \, A_i) & \text{def.\ of $\mdp'$}\\
 & = 1 - \probm_{\mdp', \state_{0}, \zstrat^{*}}(\neg \bigcap_{i \in \N} \G \F \, A_i) & \text{duality}\\
 & \ge 1 - \sum_{i \in \N} \probm_{\mdp', \state_{0}, \zstrat^{*}}(\F\G\,
  \neg A_{i}) & \text{\eqref{eq:neg-limsup} and union bound}\\
 & = 1  & \text{by \eqref{eq:avoid-Tu-nullset}} 
\end{align*} 
\end{proof}

\begin{corollary}\label{cor:GFsc}
Consider a countable MDP $\mdp$ with initial state $\state_{0}$ and a
$\bigcap_{i \in \N}\G \F \, A_i$ objective.
\begin{enumerate}
\item
If $\mdp$ is universally transient then there exists a deterministic positional strategy that is optimal from every
state that admits an optimal strategy.
\item
In general, even if $\mdp$ is not universally transient,
if there exists an optimal strategy from $\state_0$, then there
also exists an optimal deterministic Markov strategy.
\end{enumerate}
\end{corollary}
\begin{proof}
Towards item 1., 
since $\mdp$ is universally transient and $\bigcap_{i \in \N}\G \F \, A_i$ is
shift invariant, it follows
from \Cref{thm:optnestedbuchiunivtrans} and \Cref{as-to-opt} that there exists a single deterministic positional strategy that is
optimal from every state that has an optimal strategy.

Towards item 2., 
consider the MDP $S(\mdp)$ that is derived from $\mdp$ by encoding the step
counter from $\state_0$ into the states; cf.~\Cref{def:encodestep}.
$S(\mdp)$ is trivially universally transient.
We assume that there exists an optimal strategy from $\state_0$.
Then we can use \Cref{thm:optnestedbuchiunivtrans} on $S(\mdp)$ and \Cref{steptopoint}
to obtain an optimal deterministic Markov strategy from $\state_0$ in $\mdp$.
\end{proof}

\bigskip

Deterministic Markov strategies as in \Cref{cor:GFsc} are not the only type of `simple’ optimal
strategies for $\bigcap_{i \in \N} \G \F \, A_i$. Alternatively, optimal strategies (if they exist) can be chosen
as positional randomized, i.e., one can trade the step counter for randomization.
We start by considering the special case of almost surely winning strategies.

\smallskip
\begin{theorem}\label{thm:nestedbuchimr}
  Let $\mdp=\mdptuple$ be a
  countable MDP
  with initial state $\state_{0}$
  and a $\bigcap_{i \in \N} \G \F \, A_i$ objective.
  If there exists an almost surely winning strategy from $\state_0$,
  then there also exists an almost surely winning
  randomized positional strategy.
\end{theorem}
\begin{proof}[Proof outline]
  W.l.o.g.\ one can assume that $\mdp$ is finitely branching, and
  that every state admits an almost surely winning strategy.
A run satisfies $\bigcap_{i \in \N} \G \F \, A_i$
iff it visits infinitely many transitions
in $A_i$ for every $i \in \N$.
We partition the state space into infinitely many finite regions
in the shape of expanding rings around the initial state $\state_0$,
where membership in each ring
is defined via certain lower and upper bounds on the length of the shortest path from
$\state_0$.
Inside each ring, we fix a randomized positional strategy
that is a weighted combination of countably infinitely many different
deterministic positional strategies.
Inside the $i$-th ring, the strategy focuses with high probability
on visiting a transition in $A_i$, such that $A_i$ is visited with probability
$\ge 1/2$ in the $i$-th ring.
However, for each other index $j \in \N \setminus \{i\}$,
the strategy also retains some small fixed positive chance $q_j >0$ to visit
the set $A_j$.
The set of runs from $\state_0$ can then be partitioned into the
following two types.
Non-transient runs stay inside some finite subspace forever.
Hence, for each $j\in \N$, one infinitely often gets the same
fixed small positive chance $q_j >0$ to visit $A_j$, and thus one visits each $A_j$
infinitely often almost surely.
Transient runs eventually leave every finite set forever, and thus
visit the $i$-th ring for every $i$. So, for each $i \in \N$, there
is a chance $\ge 1/2$ of visiting $A_i$.

In either of the two cases,
the set of runs that don't satisfy $\bigcap_{i \in \N} \G \F \, A_i$
is a nullset, because the sets $A_i$ form a monotone decreasing chain.
\end{proof}
  
\begin{proof}
W.l.o.g.\ we can assume that $\mdp$ is finitely branching by \Cref{lem:branchingreplacement}.
Also w.l.o.g.\ we can assume that every state in $\mdp$ has an almost surely
winning strategy, because otherwise it would suffice to consider a sub-MDP as follows:
Consider an almost surely winning strategy
$\zstrat'$ from $\state_{0}$ in $\mdp$.
Let $\states' \subseteq \states$ be the subset of states that are visited
under $\zstrat'$. By restricting $\mdp$ to $\states'$, we obtain a sub-MDP
$\mdp'$ where all states are almost surely winning for $\bigcap_{i \in \N} \G \F \, A_i$.
It now suffices to construct an almost surely winning
randomized positional strategy from $\state_0$ in $\mdp'$, since the same
strategy is also almost surely winning from $\state_0$ in $\mdp$.
Thus, in the rest of the proof we can assume that every state in $\mdp$ has an almost surely
winning strategy.
  
  Let $d(\state_1,\state_2)$ be the length of the shortest path from state
  $\state_1$ to state $\state_2$ in $\mdp$. For every finite set of states
  $\hat{\states} \subseteq \states$ let
  $\bubble{n}{\hat{\states}} \eqdef \{\state' \mid \exists \state \in \hat{\states}.\ d(\state,\state') \leq n \}$.
  Since $\mdp$ is finitely branching, $\bubble{n}{\hat{\states}}$ is finite
  for every finite set $\hat{\states}$ and $n \in \N$. 

All states in $\mdp$ are almost surely winning for $\bigcap_{i \in \N} \G \F \, A_i$
and $\bigcap_{i \in \N} \G \F \, A_i \subseteq \F\, A_j$ for every $j \ge 1$,
thus
it follows that all states in $\mdp$ are almost surely winning
for $\F\, A_j$.
By \cite{KMST2020c}, for every $j \ge 1$, there exists a \emph{uniform}
deterministic positional strategy $\xstrat^j$ (that thus does not
depend on the start state) that is a.s.\ winning for the reachability objective $\F\, A_j$ in $\mdp$, i.e.,
\begin{equation}\label{eq:thm:mr-opt-pp-uniform}
\forall \state\, \probm_{\mdp,\state,\xstrat^j}(\F\, A_j) = 1
\end{equation}

We inductively construct an increasing (w.r.t.\ set inclusion) sequence 
of sets of states $\states_i$ with $i = -1,0,1,2,\dots$ such that
$\states_{-1} \eqdef \emptyset$,
$\states_0 \eqdef \{\state_0\}$,
$\states_{i+1} \eqdef \bubble{n_{i+1}}{\states_i}$
(for $i \ge 0$ and suitably chosen numbers $n_i$),
randomized positional strategies $\zstrat_i$
and MDPs $\mdp_0 \eqdef \mdp$ and $\mdp_{i+1}$ derived from $\mdp_i$
by fixing the choices according to $\zstrat_{i+1}$ inside $\states_{i+1}$
such that the following two properties hold:
\begin{equation}\label{eq:thm:mr-opt-pp-1}
\forall i\, \forall \state\, \exists \zstrat'\ \probm_{\mdp_i,\state,\zstrat'}(\bigcap_{j \in \N} \G \F \, A_j) = 1
\end{equation}
i.e., all states in $\mdp_i$ are still a.s.\ winning, and
\begin{equation}\label{eq:thm:mr-opt-pp-2}
\forall \state \in \states_{i-1}
\ \probm_{\mdp_{i-1},\state,\zstrat_{i}}(\F^{\le n_{i}} A_{i})
\ge 1/4
\end{equation}
For the base case of $i=0$ we have $\states_0 = \{\state_0\}$
and $\mdp_0 = \mdp$. \Cref{eq:thm:mr-opt-pp-1} holds by our assumption on
$\mdp$ and \Cref{eq:thm:mr-opt-pp-2} for $i=0$ is vacuously true, since $\states_{-1} = \emptyset$.

For the inductive step $i \to i+1$ consider $\mdp_i$ and let $\state \in \states$.
By the induction hypothesis and \eqref{eq:thm:mr-opt-pp-1} there is a strategy $\zstrat'$ such that
$
\probm_{\mdp_i,\state,\zstrat'}(\F\, A_{i+1}) \ge 
\probm_{\mdp_i,\state,\zstrat'}(\bigcap_{j \in \N} \G \F \, A_j) = 1
$.
I.e., all states in $\mdp_i$ are a.s.\ winning for the reachability objective
$\F\, A_{i+1}$.

By \cite{KMST2020c} there exists a \emph{uniform}
deterministic positional
strategy $\zstrat_i^{i+1}$ (that thus does not
depend on the start state) that is a.s.\ winning for the reachability objective $\F\, A_{i+1}$ in $\mdp_i$, i.e.,
$
\forall \state \in \states\ \probm_{\mdp_i,\state,\zstrat_i^{i+1}}(\F\, A_{i+1}) = 1.
$
By continuity  of measures, we have $\forall \state \in \states$
\[
1 = \probm_{\mdp_i,\state,\zstrat_i^{i+1}}(\F\, A_{i+1}) 
  = \probm_{\mdp_i,\state,\zstrat_i^{i+1}}\left(\cup_{m\ge 0}\,\F^{\le m} A_{i+1}\right)
  = \lim_{m \to \infty} \probm_{\mdp_i,\state,\zstrat_i^{i+1}}(\F^{\le m} A_{i+1}).
\]
Thus we can pick a number $m(\state) \in \N$ such that
$$\probm_{\mdp_i,\state,\zstrat_i^{i+1}}(\F^{\le m(\state)} A_{i+1}) \ge 1/2.$$
(Note that $m(\state)$ does depend on the start state $\state$, unlike the
strategy $\zstrat_i^{i+1}$.)

Since $\states_i$ is finite, we can pick a finite number
$n_{i+1} \eqdef \max\{m(\state) \mid \state \in \states_i\} \in \N$
and obtain
\begin{equation}\label{eq:thm:mr-opt-pp-3}
\forall \state \in \states_i
\ \probm_{\mdp_i,\state,\zstrat_i^{i+1}}(\F^{\le n_{i+1}} A_{i+1}) \ge 1/2.
\end{equation}
We let $\states_{i+1} \eqdef \bubble{n_{i+1}}{\states_i}$, and thus the above
is attained inside $\states_{i+1}$.

We define the randomized positional strategy $\zstrat_{i+1}$ on the set of states $\states_{i+1}$
of $\mdp_i$ as follows.
(Note that in $\mdp_i$ the choices in states $\states_i$ are already fixed.)
At every state $\state$ in $\states_{i+1} \setminus \states_i$ play like $\zstrat_i^{i+1}$ with
some large probability $p_{i+1}$ (to be determined) and otherwise,
play like the
deterministic positional
strategy $\xstrat^j$
(from \eqref{eq:thm:mr-opt-pp-uniform} above)
with a small
probability $(1-p_{i+1})\cdot 2^{-j}$ for all $j \ge 1$.
Since we have $p_{i+1} + (1-p_{i+1})\sum_{j \ge 1} 2^{-j} = 1$, this is a
distribution and the strategy $\zstrat_{i+1}$ is well defined.

Intuitively, $\zstrat_{i+1}$ focuses on the objective $\F\, A_{i+1}$
with a large probability $p_{i+1}$,
but keeps every other reachability objective $\F\, A_j$ in sight with a
nonzero probability.
Note that the attainment of $\zstrat_{i+1}$ with respect to $\F^{\le n_{i+1}} A_{i+1}$
is a continuous function in $p_{i+1}$, which converges to a value $\ge 1/2$ for
every start state $\state \in \states_i$ by \eqref{eq:thm:mr-opt-pp-3}.
I.e.,
\[
\forall \state \in \states_i\ \lim_{p_{i+1} \to 1}
\ \probm_{\mdp_i,\state,\zstrat_{i+1}}(\F^{\le n_{i+1}} A_{i+1}) \ge 1/2.
\]
From the finiteness of $\states_i$ it follows that there exists
a probability $p_{i+1} < 1$ such that
\[
\forall \state \in \states_i
\ \probm_{\mdp_i,\state,\zstrat_{i+1}}(\F^{\le n_{i+1}} A_{i+1}) \ge 1/4
\]
and thus we obtain \eqref{eq:thm:mr-opt-pp-2} for $i+1$.

We obtain the MDP $\mdp_{i+1}$ from $\mdp_i$ by fixing all choices in
$\states_{i+1} \setminus \states_i$ according to $\zstrat_{i+1}$.

Now we show that \eqref{eq:thm:mr-opt-pp-1} holds for $\mdp_{i+1}$.
We construct a strategy $\zstrat'$ that is a.s.\ winning for $\bigcap_{i \in \N} \G \F \, A_i$ in
$\mdp_{i+1}$ from every start state.
Inside $\states_{i+1}$ all choices are already fixed in $\mdp_{i+1}$
according to strategies $\zstrat_{k}$ for $1 \le k \le i+1$. 
By definition of these $\zstrat_{k}$, for all $j \ge 1$ at each state $\state \in \states_{i+1}$
the strategy $\xstrat^j$ is played with some positive probability $p(\state,j) >0$.
Then $p^j \eqdef \min_{\state \in \states_{i+1}} p(\state,j) > 0$, since
$\states_{i+1}$ is finite.
The strategy $\zstrat'$ plays in phases $j=1,2,3,\dots$.
In each phase $j$ it plays like $\xstrat^j$ everywhere outside of $\states_{i+1}$
until $A_j$ is reached, and then proceeds to phase $j+1$, etc.
It suffices to show that in every phase $j$ we reach $A_j$ eventually almost surely,
i.e., we'll show that
\begin{equation}\label{eq:thm:mr-opt-pp-as-j}
\forall \state\ \probm_{\mdp_{i+1},\state,\xstrat^j}(\F\, A_j) = 1
\end{equation}
We partition the complement of the objective $\F\, A_j$
into two parts and show that each is a nullset.
We partition $\neg\F\, A_j$ into those runs that visit $S_{i+1}$ infinitely often
and those that don't.
\begin{equation}\label{eq:partition-neg-Tj}
\neg\F\, A_j = \G\neg A_j = (\G\neg A_j \cap \G\F S_{i+1})
\uplus (\G\neg A_j \cap \F\G \neg S_{i+1})
\end{equation}

For the first part of the partition, consider the strategy $\xstrat^j$ on $\mdp$.
By \eqref{eq:thm:mr-opt-pp-uniform} and continuity of measures, for every state $\state$ we have  
$$1 = \probm_{\mdp,\state,\xstrat^j}(\F\, A_j) =
\lim_{m \to \infty} \probm_{\mdp,\state,\xstrat^j}(\F^{\le m} A_j).$$
Thus for every state $\state$ we can pick a number $m(\state)$ such that
$\probm_{\mdp,\state,\xstrat^j}(\F^{\le m(\state)} A_j) \ge 1/2$.
Let $m \eqdef \max_{\state \in \states_{i+1}} m(\state)$. The number $m$ is finite, since
$\states_{i+1}$ is finite.
Recall that in $\mdp_{i+1}$ the fixed strategy inside the finite set $\states_{i+1}$ plays
$\xstrat^j$ with some probability $\ge p^j >0$ at each state in $\states_{i+1}$.
Outside of $\states_{i+1}$ we play $\xstrat^j$ with probability $1$ (and thus
also with probability $\ge p^j$).
It follows that, for every $\state' \in \states_{i+1}$, we have
$\probm_{\mdp_{i+1},\state',\xstrat^j}(\F^{\le m} A_j) \ge (p^j)^m
\cdot 1/2 > 0$.
Therefore
\begin{equation}\label{eq:partition-part-1}
\probm_{\mdp_{i+1},\state,\xstrat^j}(\G\neg A_j \cap \G\F\, S_{i+1})
\le (1 - (p^j)^m \cdot 1/2)^\infty = 0
\end{equation}
For the second part of the partition, consider an arbitrary state $\state'$. We
have
\begin{equation}\label{eq:partition-part-2-pre}
\begin{array}{ll}
\probm_{\mdp_{i+1},\state',\xstrat^j}(\G\neg A_j \cap \G\neg S_{i+1})
= \probm_{\mdp,\state',\xstrat^j}(\G\neg A_j \cap \G\neg S_{i+1})
\le \probm_{\mdp,\state',\xstrat^j}(\G\neg A_j)
= 0
\end{array}
\end{equation}
The first equality above holds because $\mdp$ and $\mdp_{i+1}$ coincide outside
$\states_{i+1}$.
The last equality is due to \eqref{eq:thm:mr-opt-pp-uniform}.
Thus for all $\state \in \states$ 
\begin{equation}\label{eq:partition-part-2}
\begin{array}{ll}
\probm_{\mdp_{i+1},\state,\xstrat^j}(\G\neg A_j \cap \F\G\neg S_{i+1})
\le 
\probm_{\mdp_{i+1},\state,\xstrat^j}(\F(\G\neg A_j \cap \G\neg S_{i+1}))
& \text{set inclusion}\\
\le
\sup_{\state'} \probm_{\mdp_{i+1},\state',\xstrat^j}(\G\neg A_j \cap \G\neg S_{i+1})
& \text{def.\ of $\F$} \\
= 0
& \text{by \eqref{eq:partition-part-2-pre}}
\end{array}
\end{equation}
By combining
\eqref{eq:partition-neg-Tj},
\eqref{eq:partition-part-1}
and
\eqref{eq:partition-part-2}
we obtain \eqref{eq:thm:mr-opt-pp-as-j} and thus
\eqref{eq:thm:mr-opt-pp-1}.
This concludes the inductive construction.

\smallskip
Now we construct the
randomized positional
strategy $\zstrat$ that is almost surely winning for
$\bigcap_{i \in \N} \G \F \, A_i$ from
$\state_0$ in $\mdp$.
The strategy $\zstrat$ plays like $\zstrat_{i+1}$ at all states in
$\states_{i+1} \setminus \states_i$ for all $i$, i.e., it corresponds exactly to the
choices that are fixed in the derived MDPs $\mdp_i$.
We consider the dual objective $\neg\bigcap_{i \in \N} \G \F \, A_i$ and show that it is a
nullset under $\zstrat$.
First we note that 
\begin{equation}\label{eq:thm:mr-opt-pp-limsup-char}
\neg\bigcap_{i \in \N} \G \F \, A_i = \bigcup_{i \in \N} \F\G\ \neg A_i.
\end{equation}
We show that each part of this union $\F\G\ \neg A_i$ is a
nullset under $\zstrat$.
To this end, we partition $\F\G\ \neg A_i$ into two parts:
\begin{equation}\label{eq:thm:mr-opt-pp-part-transience}
\F\G\ \neg A_i = 
((\F\G\ \neg A_i) \cap \transience) \uplus
((\F\G\ \neg A_i) \cap \neg\transience)
\end{equation}
where
$\transience \eqdef \bigcap_{\state \in \states} \F\G \neg \state$.

For the first part of the partition of
\eqref{eq:thm:mr-opt-pp-part-transience},
we start by considering the slightly different objective 
$(\G\ \neg A_i) \cap \transience$.
Let $\state''$ be an arbitrary state in $\mdp$ that is reachable from
$\state_0$. Thus there exists some minimal index $k(\state'')$ such that
$\state'' \in \states_{k(\state'')}$. Now consider only the set of runs $\Runs{\mdp}{\state''}$
that start from $\state''$.
Transient runs eventually leave every finite set
forever. Thus
$\Runs{\mdp}{\state''} \cap (\G\ \neg A_i) \cap \transience \subseteq
\Runs{\mdp}{\state''} \cap (\G\ \neg A_i) \cap \bigcap_{j \ge \max(i,k(\state''))} \F(\states_j \setminus \states_{j-1})$.
On the other hand, by \eqref{eq:thm:mr-opt-pp-2} and the definition of
$\zstrat$ we have
$\forall \state \in \states_j
\ \probm_{\mdp,\state,\zstrat}(\F^{\le n_{j+1}} A_{j+1}) \ge 1/4$.
I.e., $A_{j+1}$ is visited with probability $\ge 1/4$ inside $\states_{j+1}$.
For the infinitely many $j \ge \max(i,k(\state''))$
we have the inclusion $A_i \supseteq A_j$
and thus for every state $\state''$ we have
\[
\begin{array}{ll}
\probm_{\mdp,\state'',\zstrat}((\G\ \neg A_i) \cap \transience)\\
\le
\probm_{\mdp,\state'',\zstrat}((\G\ \neg A_i) \cap \bigcap_{j \ge \max(i,k(\state''))} \F(\states_j \setminus \states_{j-1})) 
\le
(1 - 1/4)^\infty 
= 0
\end{array}
\]
Finally,
\begin{equation}\label{eq:thm:mr-opt-pp-transient}
\begin{array}{ll}
\probm_{\mdp,\state_0,\zstrat}((\F\G\ \neg A_i) \cap \transience)
\le \sup_{\state''} \probm_{\mdp,\state'',\zstrat}((\G\ \neg A_i) \cap \transience)
= 0.
\end{array}
\end{equation}

For the second part
of the partition of \eqref{eq:thm:mr-opt-pp-part-transience},
for some arbitrary state $\state$,
consider the event $(\F\G\ \neg A_i) \cap (\G\F\ \state)$.
Continuity of measures and \eqref{eq:thm:mr-opt-pp-uniform} yield  
$1 = \probm_{\mdp,\state,\xstrat^i}(\F\, A_i) =
\lim_{m \to \infty} \probm_{\mdp,\state,\xstrat^i}(\F^{\le m} A_i).$
Thus we can pick a number $m(\state)$ such that
$\probm_{\mdp,\state,\xstrat^i}(\F^{\le m(\state)} A_i) \ge 1/2.$
Recall that $\zstrat$ plays $\xstrat^i$ with some probability
$p^i(\state') >0$ at each state
$\state' \in \bubble{m(\state)}{\{\state\}}$
that is reachable from $\state$
in $\le m(\state)$ steps. Since $\mdp$ is finitely branching,
$\bubble{m(\state)}{\{\state\}}$ is finite
and thus $p^i \eqdef \min_{\state' \in \bubble{m(\state)}{\{\state\}}} p^i(\state') > 0.$
It follows that
\[q_i \eqdef \probm_{\mdp,\state,\zstrat}(\F^{\le m(\state)} A_i) \ge (p^i)^{m(\state)} \cdot 1/2 >0\]
i.e., after every visit to $\state$ we have a positive probability $q_i$ of reaching $A_i$.
Therefore for all states $\state''$ we have
$\probm_{\mdp,\state'',\zstrat}((\G\ \neg A_i) \cap
(\G\F\ \state)) \le (1-q_i)^\infty = 0.$
Thus,
\begin{equation}\label{eq:thm:mr-opt-pp-not-transient-part}
\begin{array}{ll}
\probm_{\mdp,\state_0,\zstrat}((\F\G\ \neg A_i) \cap
(\G\F\ \state)) 
\le
\sup_{\state''} \probm_{\mdp,\state'',\zstrat}((\G\ \neg A_i) \cap
(\G\F\ \state)) 
= 0.
\end{array}
\end{equation}
From
the definition of $\transience$, a union bound and \eqref{eq:thm:mr-opt-pp-not-transient-part} 
we obtain
\begin{equation}\label{eq:thm:mr-opt-pp-not-transient}
\begin{array}{ll}
\probm_{\mdp,\state_0,\zstrat}((\F\G\ \neg A_i) \cap \neg\transience)
= \probm_{\mdp,\state_0,\zstrat}((\F\G\ \neg A_i) \cap
\bigcup_{\state \in \states} \G\F\ \state) \\
\le 
\sum_{\state \in \states} \probm_{\mdp,\state_0,\zstrat}((\F\G\ \neg A_i) \cap
\G\F\ \state)
 = 0
\end{array}
\end{equation}

Finally, we show that $\zstrat$ wins almost surely.
\[
\begin{array}{llr}
\probm_{\mdp,\state_0,\zstrat}(\bigcap_{i \in \N} \G \F \, A_i)
\ge
1 - \probm_{\mdp,\state_0,\zstrat}(\neg\bigcap_{i \in \N} \G \F \, A_i) & \ \text{duality}\\
=
1 - \probm_{\mdp,\state_0,\zstrat}(\bigcup_i \F\G\ \neg A_i)
& \ \text{\eqref{eq:thm:mr-opt-pp-limsup-char}}\\
\ge
1 - \sum_i \probm_{\mdp,\state_0,\zstrat}(\F\G\ \neg A_i) &
\ \text{union bound}\\
= 1 - \sum_i \probm_{\mdp,\state_0,\zstrat}((\F\G\ \neg A_i)
\cap \transience) & \ \text{case split} \\
\quad - \sum_i \probm_{\mdp,\state_0,\zstrat}((\F\G\ \neg A_i) \cap\neg\transience) \\
= 1 & \ \text{\eqref{eq:thm:mr-opt-pp-transient}, \eqref{eq:thm:mr-opt-pp-not-transient}} & 
\end{array}
\] 
\end{proof}

\begin{corollary}\label{cor:GFmr}
  Consider a countable MDP $\mdp$ and a
  $\bigcap_{i \in \N} \G \F \, A_i$ objective.
  There exists a single randomized positional
  strategy that is optimal from every state that has an optimal strategy.
\end{corollary}
\begin{proof}
  Since the objective $\bigcap_{i \in \N} \G \F \, A_i$ is shift invariant in every MDP,
  the result follows from \Cref{thm:nestedbuchimr} and
  \Cref{as-to-opt}.
\end{proof}

\subsubsection{\texorpdfstring{Strategy Complexity of $\eps$-optimal  Strategies}
{Strategy Complexity of epsilon-optimal  Strategies}
}

First we show a general result about a combined objective that includes
$\transience$. It generalizes
the result on a combined $\transience \cap \always\eventually A$ objective of
\cite[Lemma 4]{KMST:Transient-arxiv}.

\bigskip
\begin{theorem}\label{thm:epsnestedbuchi}
Consider a countable MDP $\mdp$ with initial state $\state_{0}$ and the objective
$\formula \eqdef \transience \cap \bigcap_{i \in \N} \G \F \, A_i$.
For every $\eps > 0$ there exists an $\eps$-optimal deterministic 1-bit
strategy from $\state_0$.
\end{theorem}
\begin{proof}
By \Cref{lem:branchingreplacement}, w.l.o.g.\ we can assume that $\mdp$ is finitely branching. 
Since $\mdp$ is finitely branching, we can define
$d: S \to \N$ such that $d(s) \eqdef \min \{ \vert w \vert : w \text{ is a path from $s_{0}$ to $s$} \}$. 
Let $\bubble{n}{s_{0}} \eqdef \{ t \in \transition_{\mdp} \mid t= (s \transition_{\mdp} s'),\, d(s') \leq n \}$.
Since $\mdp$ is finitely branching, $\bubble{n}{s_{0}}$ is finite for every
$n \in \N$.

Let $A_{i}^{n} \eqdef A_{i} \cap \bubble{n}{s_{0}}$ and 
let $A_{i}^{>n} \eqdef A_{i} \setminus A_{i}^{n}$.
(In particular, $A_{i}^{>0} = A_i$.)

Transient runs eventually leave every finite set forever, and in particular
each of the finite sets $A_{x}^{n}$ for every $x,n \in \N$.
Therefore, the objective $\formula$ implies $\F (A_{x}^{>n})$
and thus $\formula$ can be written as follows.
\begin{equation} \label{eq:acyclicity}
\formula
=
\transience \cap \bigcap_{i \in \N} \G \F \, A_i
=
\transience \cap \bigcap_{i \in \N} \G \F \, A_i
\cap \F (A_{x}^{>n}) 
\end{equation}

\paragraph{First step of the proof.}
We consider an $\eps$-optimal strategy for
$\formula$ and show that it satisfies a certain stronger objective with
a probability that is almost as high (just losing one $\eps$).

Let $\tau$ be a general $\eps$-optimal strategy for $\formula$ from $s_0$, i.e.,
\begin{equation}\label{eq:pointepsupper-eps}
\probm_{\mdp,s_{0}, \tau}\left(\formula \right) \ge \valueof{\mdp, \formula}{s_{0}}- \eps.
\end{equation}
We now show that 
for a suitably chosen increasing sequence of natural numbers $0=n_0 < n_1 < n_2 \dots$ we have
\[
\probm_{\mdp, s_{0}, \tau} \left(  
\left( 
\bigcap_{j=1}^{\infty} \left( \F(A_{j}^{>n_{j-1}}) \cap \F(A_{j}^{n_{j}}) \right)
\right)  \cap \formula
\right) 
\geq \valueof{\mdp,\formula}{s_{0}} - 2 \eps.
\]
To this end, we prove by induction on $i$ that
the following property holds for all $i\ge 0$ and 
$\eps_j \eqdef \eps \cdot 2^{-j}$.

{\bf\noindent Induction hypothesis:}
\begin{equation}\label{eq:epsnestedbuchi-induction}
\probm_{\mdp, s_{0}, \tau} \left(
\left( 
\bigcap_{j=1}^{i} \left( \F(A_{j}^{>n_{j-1}})  \cap \F(A_{j}^{n_{j}}) \right)
\right)  \cap  \formula
\right) 
\geq \valueof{\mdp,\formula}{s_{0}} - \eps - \sum_{j=1}^{i} \eps_{j}.
\end{equation}

{\bf\noindent Base case:}
For the base case of $i=0$, the property follows directly from
\eqref{eq:pointepsupper-eps}.
(Empty index sets for intersection and sum.)

{\bf\noindent Induction step from $i$ to $i+1$:}
By the induction hypothesis we have
\[
\probm_{\mdp, s_{0}, \tau} \left(
\left( 
\bigcap_{j=1}^{i} \left( \F(A_{j}^{>n_{j-1}})  \cap \F(A_{j}^{n_{j}}) \right)
\right)  \cap  \formula
\right) 
\geq \valueof{\mdp,\formula}{s_{0}} - \eps - \sum_{j=1}^{i} \eps_{j}.
\]
Since, by \eqref{eq:acyclicity}, $\formula = \formula \cap \F(A_{i+1}^{>n_{i}})$, we obtain that 
\[
\probm_{\mdp, s_{0}, \tau} \Bigg(
\Bigg( 
\bigcap_{j=1}^{i} \left( \F(A_{j}^{>n_{j-1}}) \cap \F(A_{j}^{n_{j}}) \right) \Bigg)
 \cap \formula \cap \F(A_{i+1}^{>n_{i}})
\Bigg) 
\geq \valueof{\mdp,\formula}{s_{0}} - \eps - \sum_{j=1}^{i} \eps_{j}.
\]
We have $\F (A_{i+1}^{>n_{i}}) = \bigcup_{k \in \N}
\F^{\leq k} (A_{i+1}^{>n_{i}})= \bigcup_{k > n_{i}} \F^{\leq
k} (A_{i+1}^{>n_{i}})$. 
Hence, by continuity of measures,
\begin{align*}
\lim_{n_i < k \to \infty}
\probm_{\mdp, s_{0}, \tau} & \Bigg(
\Bigg( 
\bigcap_{j=1}^{i} \left( \F(A_{j}^{>n_{j-1}}) \cap \F(A_{j}^{n_{j}}) \right) \Bigg)
 \cap \formula \cap \F^{\le k}(A_{i+1}^{>n_{i}})
\Bigg) \\
& \geq \valueof{\mdp,\formula}{s_{0}} - \eps - \sum_{j=1}^{i} \eps_{j}.
\end{align*}
So there exists an $n_{i+1} > n_{i}$ such that 
\begin{align*}
\probm_{\mdp, s_{0}, \tau} & \Bigg( \Bigg( 
\bigcap_{j=1}^{i} \left( \F(A_{j}^{>n_{j-1}})  \cap \F(A_{j}^{n_{j}}) \right)
\Bigg) \cap
\formula \cap                              
\F^{\le n_{i+1}}(A_{i+1}^{>n_{i}}) 
\Bigg) \\
& \geq \valueof{\mdp,\bigcap_{i \in \N} \G \F \, A_i}{s_{0}} - \eps - \left(\sum_{j=1}^{i} \eps_{j}\right) - \eps_{i+1}
\end{align*}
I.e., we get 
\[
\probm_{\mdp, s_{0}, \tau} \left( \left( 
\bigcap_{j=1}^{i+1} \left( \F(A_{j}^{>n_{j-1}}) \cap \F(A_{j}^{n_{j}}) \right)  \right)
 \cap \formula \right) 
 \geq \valueof{\mdp,\formula}{s_{0}} - \eps - \sum_{j=1}^{i+1} \eps_{j}
\]
which completes the induction step.

By applying continuity of measures
to \eqref{eq:epsnestedbuchi-induction},
we obtain the desired result, i.e.,
\begin{equation}\label{eq:tauachievement}
\begin{split}
\probm_{\mdp, s_{0}, \tau} & \left(  
\left( 
\bigcap_{j=1}^{\infty} \left( \F(A_{j}^{>n_{j-1}}) \cap \F(A_{j}^{n_{j}}) \right)
\right)  \cap \formula
\right) \\
\geq & \ \valueof{\mdp,\formula}{s_{0}} - \eps - \sum_{j=1}^{\infty} \eps_{j} \\
= & \ \valueof{\mdp,\formula}{s_{0}} - 2 \eps. \\
\end{split}
\end{equation}

\smallskip
\paragraph{Second step of the proof.}
Here we consider a stronger objective
$\transience \cap \G \F \, {\sf Good} \subseteq \formula$
and show how it relates to $\formula$.

Let ${\sf Good}_{i} \eqdef A_i^{n_i}$
be those transitions inside $\bubble{n_{i}}{s_{0}}$ which are also in $A_i$. 
Let ${\sf Good} \eqdef \bigcup_{i \in \N} {\sf Good}_{i}$. 

The following claim shows that the inclusion
$\transience \cap \G \F \, {\sf Good} \subseteq \formula$
holds.

\begin{claim}\label{claim:goodislimsup}
  $\transience \cap \G \F \, {\sf Good} \subseteq \transience \cap
  \bigcap_{i \in \N} \G \F \, A_i \eqdef \formula$.
\end{claim}
\begin{proof}
Consider a run $\rho$ that satisfies $\transience \cap \G \F \, {\sf Good}$.
In particular, $\rho$ must visit ${\sf Good}$ infinitely often.
However, ${\sf Good} = \bigcup_{i \in \N} A_i^{n_i}$ where each set
$A_i^{n_i}$ is finite.
Since $\rho$ must also satisfy $\transience$, it can visit each of these
finite sets $A_i^{n_i}$ only finitely often.
Therefore, $\rho$ must visit sets $A_i^{n_i}$ for arbitrarily large numbers
$i$.
That is, $\rho$ satisfies $\bigcap_{i \in \N} \bigcup_{j \ge i} \F \, A_j^{n_j}$
and thus also $\bigcap_{i \in \N} \bigcup_{j \ge i} \F \, A_j$.
Since the sets $\{A_i\}_{i\in \N}$ are a decreasing chain, each visit to $A_j$
with $j \ge i$ is also a visit to $A_i$.
Hence $\rho$ also satisfies $\bigcap_{i \in \N} \G \F \, A_i$.
\end{proof}

Next we show that the value of $s_0$ for
$\transience \cap \G \F \, {\sf Good}$
is almost as high as its value for $\formula$ (losing only $2\eps$).
Due to the inclusions
\[
\left( \bigcap_{j=1}^{\infty} \F(A_{j}^{>n_{j-1}}) \cap \F(A_{j}^{n_{j}})
\right)
\cap \formula \subseteq \G \F \, {\sf Good} \cap \formula
\subseteq
\transience \cap \G \F \, {\sf Good}
,
\]
it follows immediately from \eqref{eq:tauachievement} that
\begin{equation}\label{eq:buchi-almost-limsup}
  \probm_{\mdp, s_{0}, \tau} (\transience \cap \G \F \, {\sf Good})
  \geq
\valueof{\mdp,\formula}{s_{0}} - 2 \eps,
\end{equation}
and thus
\begin{equation}\label{eq:buchigoodvalue}
  \valueof{\mdp,\transience \cap \G \F \, {\sf Good}}{s_{0}}
  \geq \probm_{\mdp, s_{0}, \tau} (\transience \cap \G \F \, {\sf Good}) 
\geq \valueof{\mdp,\formula}{s_{0}} - 2 \eps.
\end{equation}

\smallskip
\paragraph{Third step of the proof.}
Here we put the previous steps together and show the statement of the theorem.
By \cite[Lemma 4]{KMST:Transient-arxiv},
there exists an $\eps$-optimal deterministic 1-bit strategy $\sigma^{*}$ for
$\transience \cap \G \F \, {\sf Good}$ from $s_0$ in $\mdp$. Thus
\begin{align*}
  & \probm_{\mdp, s_{0}, \sigma^{*}}(\formula) \\
  & = \probm_{\mdp,s_{0},\sigma^{*}}(\transience \cap \bigcap_{i \in \N} \G \F \, A_i) & \text{by def.\ of $\formula$}\\  
  & \ge \probm_{\mdp, s_{0}, \sigma^{*}} (\transience \cap \G \F \, {\sf Good}) & \text{\Cref{claim:goodislimsup}} \\
  & \geq \valueof{\mdp,\transience \cap \G \F \, {\sf Good}}{s_{0}} - \eps & \text{$\sigma^{*}$ is $\eps$-opt.}\\
  & \geq \valueof{\mdp,\formula}{s_{0}} - 3 \eps.      & \text{by \eqref{eq:buchigoodvalue}}
\end{align*}
Since $\eps>0$ can be made arbitrarily small, the result follows.
\end{proof}

\begin{corollary}\label{cor:epsGF}
Consider a countable MDP $\mdp$ with initial state $\state_{0}$ and a
$\bigcap_{i \in \N}\G \F \, A_i$ objective.
\begin{enumerate}
\item
If $\mdp$ is universally transient, then there exists a deterministic 1-bit strategy that is $\eps$-optimal from $s_0$.
\item
In general, there exists a Det(SC+1-bit) $\eps$-optimal strategy from
$\state_0$ (even if $\mdp$ is not universally transient).
\end{enumerate}
\end{corollary}
\begin{proof}
  Towards item 1., if $\mdp$ is universally transient then, under every
  strategy $\sigma$, we have
  $\probm_{\mdp, s_{0}, \sigma}(\bigcap_{i \in \N} \G \F \, A_i)
  =
  \probm_{\mdp, s_{0}, \sigma}(\transience \cap \bigcap_{i \in \N} \G \F\, A_i)$
  and the result follows directly from \Cref{thm:epsnestedbuchi}.
  
  Towards item 2., consider the MDP $S(\mdp)$ that is derived from $\mdp$ by encoding the step
  counter from $\state_0$ into the states; cf.~\Cref{def:encodestep}.
  So $S(\mdp)$ is trivially universally transient.
Then we can use item 1.\ on $S(\mdp)$ to obtain a deterministic 1-bit strategy
on $S(\mdp)$. Finally, \Cref{steptopoint}
yields an $\eps$-optimal Det(SC+1-bit) strategy from $\state_0$ in $\mdp$.
\end{proof}

\subsection{Lower Bounds}\label{sec:limsupthreshold}
We present the lower bounds on the strategy complexity in terms of
the objective $\limsupppobj$, because it is much more intuitive to think of counterexamples
with rewards rather than thinking about which transitions are labeled as belonging to which set $A_i$. 
The MDPs in this section all have transition based rewards.

We begin by presenting the results for optimal strategies.

\begin{figure}
  \begin{center}
    \scalebox{1.2}{
    \begin{tikzpicture}
    
    \node[draw, inner sep=1pt] (S2) at (0,0) {$s_{0}$};

    \node (I1) at (8,0) {};
    \node (I2) at (8,-2) {};
    \node (I3) at (0,-1.5) {};
    
	\coordinate[shift={(0mm,0mm)}] (n) at (I3);    
    
    \node[draw, inner sep=1pt] (L1) at (1,0) {$r_{1}$};
    \node[draw, inner sep=1pt] (L2) at (2.5,0) {$r_{2}$};
    \node[draw, inner sep=1pt] (L3) at (4,0) {$r_{3}$};
    \node[draw, inner sep=1pt] (L4) at (5.5,0) {$r_{i}$};
    \node[draw, inner sep=1pt] (L5) at (7,0) {$r_{i+1}$};

    \node (HI1) at (1,-2) {};
    \node (HI2) at (2.5,-2) {};
    \node (HI3) at (4,-2) {};
    \node (HI4) at (5.5,-2) {};
    \node (HI5) at (7,-2) {};

    \draw[->,>=latex] (S2) edge node[above, midway]{$-1$} (L1)
    (L1) edge node[above, midway]{$-1$} (L2)
    (L2) edge node[above, midway]{$-1$} (L3)
    (L4) edge node[above, midway]{$-1$} (L5);
    
    
    \draw[->,>=latex, dotted, thick] (L3) edge node[above, midway]{} (L4);
    \draw[->,>=latex, dotted, thick] (L5) edge node[above, midway]{} (I1);

    \draw[->,>=latex,rounded corners=5pt] (L1) |- (n) -- (S2);   
    \draw[->,>=latex,rounded corners=5pt] (L2) |- (n) -- (S2);
    \draw[->,>=latex,rounded corners=5pt] (L3) |- (n) -- (S2);
    \draw[->,>=latex,rounded corners=5pt] (L4) |- (n) -- (S2);
    \draw[->,>=latex,rounded corners=5pt] (L5) |- (n) -- (S2);

    \node (P1) at (1.3, -0.75) {\scriptsize $-1$};
    \node (P2) at (2.8, -0.75) {\scriptsize $-\dfrac{1}{2}$};
    \node (P3) at (4.3, -0.75) {\scriptsize $-\dfrac{1}{3}$};
    \node (P4) at (6.1, -0.75) {\scriptsize $-\dfrac{1}{i}$};
    \node (P5) at (7.6, -0.75) {\scriptsize $-\dfrac{1}{i+1}$};
    
  \end{tikzpicture}
  }
    \caption{Det(F) strategies are not enough for almost sure $\limsupppobj$ in finitely branching MDPs. Note that this is very similar to \cite[Example 1]{Sudderth:2020}.
     }
    \label{limsuplowerdetf}
\end{center}    
    \end{figure}
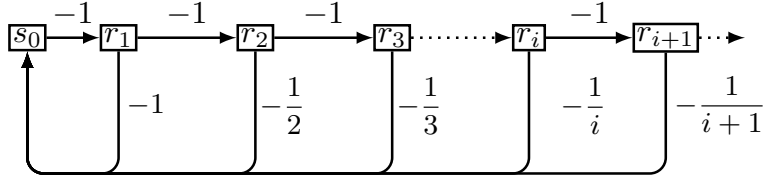

\bigskip
\begin{proposition}\label{thm:limsuplowerdetf}
There exists a finitely branching countable MDP $\mdp$ with rational
rewards as in \Cref{limsuplowerdetf}
with initial state $\state_0$ such that
\begin{enumerate}
\item\label{thm:ppalmostlowerstep-1}
$\exists \hat{\zstrat}\, \probm_{\mdp, \state_{0},\hat{\zstrat}}(\limsupppobj) = 1$, 
i.e., the state $\state_{0}$ is almost surely winning for $\limsupppobj$.
\item\label{thm:ppalmostlowerstep-2}
For every Det(F) strategy $\zstrat$ we have $\probm_{\mdp, \state_{0},\zstrat}(\limsupppobj) = 0.$
\end{enumerate}
So almost surely winning strategies, when they exist, cannot be chosen Det(F).
\end{proposition}

\begin{proof}
Towards \Cref{thm:ppalmostlowerstep-1}, let
$\hat{\zstrat}$ be the deterministic strategy that turns off the ladder
(i.e., goes back to $\state_0$) at state $r_{i}$ upon visiting it for the
  first time and then never again thereafter. Note that $\hat{\zstrat}$
  uses infinite memory.
  Let $\playset$ be the set of all runs generated by $\hat{\zstrat}$.
  We want to show that $\probm_{\mdp, \hat{\zstrat}, \state_{0}}(\playset \wedge \limsupppobj)=1$.
  Since $\hat{\zstrat}$ is deterministic and all states in $\mdp$ are
  controlled, we have $\vert \playset \vert = 1$.
  Therefore it suffices to show that the unique $\rho \in \playset$ is winning.
  By definition of $\hat{\zstrat}$, $\rho$ sees every reward $\frac{1}{i}$
  exactly once for $i=1,2,3,\dots$.
  Hence the $\limsup$ of the payoffs that $\rho$ sees is $\limsup_{i \to \infty} -\dfrac{1}{i}= 0$
  and thus $\rho \in \limsupppobj$ and
  $\probm_{\mdp, \state_{0},\hat{\zstrat}}(\limsupppobj) = 1$.

Towards \Cref{thm:ppalmostlowerstep-2},  
let $\zstrat$ be a deterministic strategy with $k$ memory modes $\{0,1,\dots,k-1\}$. 
For each memory mode $\memconf \in \{0,1,\dots,k-1\}$,
consider $\zstrat$'s behavior from $\state_{0}$.
Since $\zstrat$ is deterministic, there are only two cases.
In the first case we never visit $\state_0$ again.
In the second case we deterministically take a transition $r_{i} \transition \state_0$ for some
particular $i(\memconf)$ that depends on the initial memory mode $\memconf$.
For runs of the first case the $\limsup$ of the payoffs is $-1$.
For runs of the second case, we never visit states $r_j$ for
$j > i_{\max} \eqdef \max_{\memconf \in \{0,\dots,k-1\}} i(\memconf)$ and thus the $\limsup$
of the payoffs is $\le -\frac{1}{i_{\max}} < 0$.
Thus no runs induced by $\zstrat$ satisfy $\limsupppobj$
and $\probm_{\mdp, \zstrat, \state_{0}}(\limsupppobj)=0$.
\end{proof}

\bigskip

Now we present the lower bound results for $\eps$-optimal strategy complexity.

The $\limsupppobj$ objective generalizes the B\"uchi objective $\G \F \, A$, even for
integer rewards, hence the lower bounds for $\eps$-optimal strategies for $\G \F \, A$ from
\cite{KMST:ICALP2019,KMSW2017} carry over. The following \Cref{thm:ppepstep} shows that
randomized strategies with finite memory are not sufficient.

\begin{figure}
\begin{center}
\scalebox{1.2}{
    \begin{tikzpicture}
    
    \node[draw] (S2) at (0,0) {$s_{0}$};

    \node (I1) at (9,0) {};
    \node (I2) at (9,-2.5) {};
    \node (I3) at (0,-2.5) {};
    
	\coordinate[shift={(0mm,0mm)}] (n) at (I3);    
    
    \node[draw] (L1) at (1.5,0) {};
    \node[draw] (L2) at (3,0) {};
    \node[draw] (L3) at (4.5,0) {};
    \node[draw] (L4) at (6,0) {};
    \node[draw] (L5) at (7.5,0) {};
    
    \node[draw, circle] (R1) at (1.5,-1.5) {};
    \node[draw, circle] (R2) at (3,-1.5) {};
    \node[draw, circle] (R3) at (4.5,-1.5) {};
    \node[draw, circle] (R4) at (6,-1.5) {};
    \node[draw, circle] (R5) at (7.5,-1.5) {}; 
    
    \node[draw] (T1) at (1.5,-2.5) {};
    \node[draw] (T2) at (3,-2.5) {};
    \node[draw] (T3) at (4.5,-2.5) {};
    \node[draw] (T4) at (6,-2.5) {};
    \node[draw] (T5) at (7.5,-2.5) {}; 

    \node[draw, circle, inner sep=1pt] (B1) at (2,-0.75) {\scriptsize $\perp$};
    \node[draw, circle, inner sep=1pt] (B2) at (3.5,-0.75) {\scriptsize $\perp$};
    \node[draw, circle, inner sep=1pt] (B3) at (5,-0.75) {\scriptsize $\perp$};
    \node[draw, circle, inner sep=1pt] (B4) at (6.5,-0.75) {\scriptsize $\perp$};
    \node[draw, circle, inner sep=1pt] (B5) at (8,-0.75) {\scriptsize $\perp$};

    \draw[->,>=latex] (S2) edge node[above, midway]{$-1$} (L1)
    (L1) edge node[above, midway]{$-1$} (L2)
    (L2) edge node[above, midway]{$-1$} (L3)
    (L4) edge node[above, midway]{$-1$} (L5);
    
    \draw[->,>=latex] (L1) edge node[left, midway]{$+0$} (R1)
    (L2) edge node[left, midway]{$+0$} (R2)
    (L3) edge node[left, midway]{$+0$} (R3)
    (L4) edge node[left, midway]{$+0$} (R4)
    (L5) edge node[left, midway]{$+0$} (R5);
    
    \draw[->,>=latex] (R1) edge node[midway, below, xshift=7pt, yshift=2pt]{\scriptsize $2^{-1}$} (B1)
    (R2) edge node[midway, below, xshift=7pt, yshift=2pt]{\scriptsize $2^{-2}$} (B2)
    (R3) edge node[midway, below, xshift=7pt, yshift=2pt]{\scriptsize $2^{-3}$} (B3)
    (R4) edge node[midway, below, xshift=7pt, yshift=2pt]{\scriptsize $2^{-i}$} (B4)
    (R5) edge node[midway, below, xshift=12pt, yshift=2pt]{\scriptsize $2^{-i-1}$} (B5);
    
    \draw[->,>=latex] (R1) edge node[left, midway]{$+0$} (T1)
    (R2) edge node[left, midway]{$+0$} (T2)
    (R3) edge node[left, midway]{$+0$} (T3)
    (R4) edge node[left, midway]{$+0$} (T4)
    (R5) edge node[left, midway]{$+0$} (T5);
    
    \draw[->,>=latex] (I2) edge[dotted, thick] (T5)
    (T5) edge node[below, midway]{$+0$} (T4)
    (T4) edge[dotted, thick] (T3)
    (T3) edge node[below, midway]{$+0$} (T2)
    (T2) edge node[below, midway]{$+0$} (T1);
    
    \node (V1) at (0.75, -2.75) {$+0$};

    \draw[->,>=latex, dotted, thick] (L3) edge node[above, midway]{} (L4);
    \draw[->,>=latex, dotted, thick] (L5) edge node[above, midway]{} (I1);

    \draw[->,>=latex,rounded corners=5pt] (T1) -| (S2); 
    
    \end{tikzpicture}
    }
    \caption{$\eps$-optimal strategies for $\limsupppobj$ cannot be chosen Rand(F).
      In this example $\limsupppobj$ and $\G \F \, \{\state_0\}$ coincide.}
    \label{ppepsstep}
\end{center}    
    \end{figure}

\bigskip
\begin{proposition}\label{thm:ppepstep}
  There exists an MDP $\mdp$ as in \Cref{ppepsstep} such that
  $\valueof{\mdp, \limsupppobj}{\state_0} = 1$ and any Rand(F) strategy $\zstrat$
  is such that $\probm_{\mdp,\state_0,\zstrat}(\limsupppobj) = 0$.
\end{proposition}
\begin{proof}
  The result follows immediately from
  ~\cite[Proposition 2]{Krcal:Thesis:2009,KMSW2017},
  since in $\mdp$, $\limsupppobj$ and $\G \F \, \{\state_0\}$ coincide by construction.
\end{proof}

In \Cref{thm:buchippextract} we show a lower bound that is orthogonal to that of
\Cref{thm:ppepstep}, namely that randomized Markov strategies are not
sufficient.
First, we recall a result from \cite{KMST:ICALP2019}.

\bigskip
\begin{theorem}(\cite[Theorem 3]{KMST:ICALP2019})\label{buchiexample}
There exists a finitely branching countable MDP $\mdp=\mdptuple$ as in
\cite[Figure 3]{KMST:ICALP2019}
with initial state $\state_{0}$ and a subset of states $F \subseteq \states$
where the step counter from $\state_0$ is implicit in the current state such that 
\begin{enumerate}
\item
$\valueof{\mdp,\G \F \, F}{\state_{0}} = 1$ and
\item
for every randomized Markov strategy $\zstrat$, we have 
$\probm_{\mdp,\state_{0},\zstrat}(\G \F \, F)=0$.
\end{enumerate}
\end{theorem}

\begin{definition}\label{buchipp}
Let $\mdp = \mdptuple$ and $F \subseteq \states$
be as in \Cref{buchiexample}.
We define transition rewards by $r: \states \times \states \to \R$ such that on $\mdp$ the objectives $\limsupppobj$ and $\G \F \, F$ coincide. I.e.\ 
\begin{equation*}
    r(\state \longrightarrow \state') \eqdef \begin{cases}
               -1               & \mbox{if}\ \state \notin F\\
               +1               & \mbox{if}\ \state \in F
           \end{cases}
\end{equation*}
\end{definition}

The following proposition shows that randomized Markov strategies are not
sufficient for the $\limsupppobj$ threshold objective.

\begin{proposition}\label{thm:buchippextract}
There exists a finitely branching countable MDP $\mdp$ as in \Cref{buchipp}
with integer rewards in $\{+1,-1\}$,
initial state $\state_{0}$ and the step counter from $\state_0$ implicit in the current state
such that
\begin{enumerate}
\item
$\valueof{\mdp,\limsupppobj}{\state_{0}} = 1$, but
\item
for every randomized Markov strategy
$\zstrat$, we have
$
\probm_{\mdp,\state_{0},\zstrat}(\limsupppobj)=0
$.
\end{enumerate}
\end{proposition}
\begin{proof}
This follows directly from \Cref{buchipp} and \Cref{buchiexample}.
\end{proof}

\section{\texorpdfstring{Strategy complexity of $\bigcap_{i \in \N} \F \G A_i$
    and $\liminfppobj$}{Strategy complexity of lim inf}}\label{sec:nestedcobuchi}

\subsection{\texorpdfstring{Upper bound for $\bigcap_{i \in \N} \F \G \, A_i$ in infinitely branching MDPs}
{Upper bound in infinitely branching MDPs}}

In order to prove the main results of this section, we use the following result
on the $\transience$ objective. Recall that, given an MDP $\mdp=\mdptuple$, 
$\transience \eqdef \bigwedge_{s \in S} \F \G \neg s$.

\bigskip
\begin{theorem}(\cite[Theorem 8]{KMST:Transient-arxiv})
\label{epstransience}
In every countable MDP there exist uniform $\eps$-optimal MD strategies for
$\transience$.
\end{theorem}

\bigskip
\begin{lemma}\label{lem:transience-plus-safety}
Let $\mdp=\mdptuple$ be a countable MDP and $X \subseteq \transition$ a subset of the transitions.
For every $\eps >0$ there exists a uniform $\eps$-optimal MD strategy for the objective
$\transience \cap \G X$.
\end{lemma}
\begin{proof}
  Starting with $\mdp$, we construct a modified MDP $\mdp'$
  by adding a new state $s_\bot$ and a self-loop $s_\bot \to s_\bot$.
  Moreover, all transitions of $\mdp$ that are \emph{not} in $X$ are
  re-directed to $s_\bot$.
  Thus in $\mdp'$ all runs that visit some transition $\notin X$
  will eventually loop in $s_\bot$ and are not transient.
  Hence in $\mdp'$ we have that $\transience \cap \G X = \transience$.
  On the other hand, all runs from states $\state \in \states$ that visit only
  transitions in $X$ are unaffected by the differences between $\mdp$ and
  $\mdp'$.
  Therefore, for all states $\state \in \states$ and strategies $\tau$ from $\state$,
  \begin{equation}\label{eq:transience-plus-safety-one}
    \probm_{\mdp,\state,\tau}(\transience \cap \G X) = \probm_{\mdp',\state,\tau}(\transience \cap \G X)
  \end{equation}  
  By \Cref{epstransience}, for every $\eps >0$, there exists a uniform $\eps$-optimal MD strategy
  $\sigma$ for $\transience$ in $\mdp'$.
  This $\sigma$ is also uniform $\eps$-optimal for $\transience \cap \G X$ in $\mdp'$.
  
  We now carry $\sigma$ back to $\mdp$ and show that it is also
  uniform $\eps$-optimal for $\transience \cap \G X$ in $\mdp$.
  Let $\state_0 \in \states$ be an arbitrary start state in $\mdp$.
  \begin{align*}
    & \probm_{\mdp,\state_0,\sigma}(\transience \cap \G X) \\
    & = \probm_{\mdp',\state_0,\sigma}(\transience \cap \G X) &    \eqref{eq:transience-plus-safety-one} \\ 
    & \ge \valueof{\mdp',\transience \cap \G X}{\state_0} - \eps & \mbox{uniform $\eps$-optimality in $\mdp'$}\\ 
    & = \valueof{\mdp,\transience \cap \G X}{\state_0} - \eps    & \eqref{eq:transience-plus-safety-one}
  \end{align*}
\end{proof}
  
Moreover, we need an auxiliary lemma (inspired by \cite[Lemma 18]{KMST:Transient-arxiv}).

\begin{lemma}\label{lem:transientdrift}
Let $\mdp = \mdptuple$ be a countable MDP with initial state $s_0$. 
Let $\psi$ be any objective which implies $\transience$ and let $\sigma$ be a
strategy from the start state $s_0$. Let $X \subseteq S$ be a finite set of
states and $\delta > 0$. The following properties hold:
\begin{enumerate}
\item
  There is an $\ell \in \N$ such that $ 
\probm_{\mdp,s_0,\sigma} (\psi \cap G^{\geq \ell}(\neg X)) \geq \probm_{\mdp,s_0,\sigma} (\psi) - \delta$.
\item
  For each $n \in \N$, there exists a finite set $Y \subseteq \states$
  such that $\probm_{\mdp, s_0, \sigma}(\G^{\leq n} \, Y) \geq 1 - \delta$.
\end{enumerate}
\end{lemma}
\begin{proof}
\ \\
\begin{enumerate}
\item
  We know that $\psi$ implies $\transience$ and that
  $\transience = \bigcap_{s \in \states} \F \G \, \neg s \subseteq \bigcap_{s
    \in X} \F \G \, \neg s$. Therefore
\begin{equation}\label{eq:transience equalities}
\psi = \psi \cap \transience \subseteq \psi \cap \bigcap_{s \in X} \eventually\always \neg s \subseteq \psi.
\end{equation} 
This allows us to write:
\begin{align*}
\psi = \psi \cap \transience & = \psi \cap \bigcap_{s \in S} \F \G \, \neg s & \text{def.\ of $\transience$} \\
& = \psi \cap \bigcap_{s \in X} \F \G \, \neg s & \text{$X \subseteq S$ and \eqref{eq:transience equalities}} \\
& = \psi \cap \F \G \, \neg X & \text{$X$ is finite}\\
& = \psi \cap \bigcup_{k \in \N} \F^{\leq k} \G \, \neg X & \text{def.\ of $\F$}\\
& = \psi \cap \bigcup_{k \in \N} \G^{\geq k} \, \neg X & \text{$\F^{\leq k} \G \, = \G^{\geq k}$}\\
\end{align*}
Hence, by applying continuity of measures, we obtain that 
\[\probm_{\mdp,s_0,\sigma}(\psi) = 
\probm_{\mdp,s_0,\sigma}(\psi \cap \bigcup_{k \in \N} \G^{\geq k} \, \neg X)
= \lim_{k \to \infty} \probm_{\mdp,s_0,\sigma}(\psi \cap \G^{\geq k} \, \neg X).
\]
Thus, given $\delta$, using the definition of a limit, we know that there must be an $\ell \in \N$ such that 
$\probm_{\mdp,s_0,\sigma}(\psi \cap \G^{\geq \ell} \, \neg X) \geq \probm_{\mdp,s_0,\sigma}(\psi) - \delta$ as required.
\item
Consider the Markov chain induced by playing $\sigma$ from $s_0$.
In each round $i \le n$ we cut infinite tails off the distributions
such that we lose only $\le \delta/n$ probability. Thus we remain inside
some finite set of states $Y_i$ with probability
$\ge 1-\frac{i \cdot\delta}{n}$ after the $i$-th round.
By taking $Y \eqdef Y_n$, the result follows. \qedhere
\end{enumerate}
\end{proof}

Now we show a general result about a combined objective that includes
$\transience$.

\bigskip
\begin{theorem}\label{thm:liminfunivtrans}
  Consider a countable MDP $\mdp$ with initial state $s_{0}$ and
  the objective $\formula \eqdef \transience \cap \bigcap_{i \in \N} \F \G A_i$.
  For every $\eps >0$ there exists an $\eps$-optimal MD strategy from $s_0$.
\end{theorem}
\begin{proof}[Proof outline]
  We show that the objective $\formula$ can be
  sufficiently closely approximated by the objective
  $\transience \cap \G \, {\sf Good}$ for some suitably
  defined set of transitions ${\sf Good}$, and then use
  \Cref{lem:transience-plus-safety}.
\end{proof}
\begin{proof}
  Let $\sigma$ be a general $\eps$-optimal strategy for $\formula$
  from $s_0$ in $\mdp$, i.e. 
  \begin{equation}\label{eq:liminfunivtrans-one}
    \probm_{\mdp, s_0, \sigma}(\formula) \geq
    \valueof{\mdp,\formula}{s_0} - \eps .
  \end{equation}
Let $\eps_i \eqdef \eps \cdot \dfrac{2^{-i}}{3}$ for $i \geq 1$. 
We construct a sequence of increasing natural numbers $\{ n_i \}_{i \in \N}$ and finite sets $\{ S_i \}_{i \in \N}$ such that the following holds:

\begin{equation}\label{eq:transientepsdrift}
\probm_{\mdp, s_0, \sigma}\left( 
\formula
\cap \bigcap_{i \in \N} ( \F^{\leq n_i} \, (\G A_i \cap \G\neg S_{i-1}) \cap \G^{\leq n_i}\, S_i ) 
\right) \geq \valueof{\mdp,\formula}{s_0} - 2 \eps
\end{equation}

To this end, we prove by induction on $k$ that the following property holds for $k \geq 0$.
\textbf{Induction hypothesis:}
\begin{equation}\label{eq:indhypothesis}
\probm_{\mdp, s_0, \sigma}\left( 
\formula
\cap \bigcap_{i = 1}^{k} ( \F^{\leq n_i} \, (\G A_i \cap \G \neg S_{i-1}) \cap \G^{\leq n_i}\, S_i ) 
\right) \geq \valueof{\mdp,\formula}{s_0} - \eps - 3 \sum_{i = 1}^k \eps_i
\end{equation}

\textbf{Base case $k=0$:} Let $n_0 \eqdef 0$ and let $S_0 \eqdef \{ s_0 \}$ and $S_{-1} \eqdef \emptyset$.
Then setting $k=0$ yields empty index sets for the intersection and sum, reducing the base case to 
\eqref{eq:liminfunivtrans-one}.

\textbf{Induction step from $k$ to $k+1$:}

Assume that for some $k$ we have \eqref{eq:indhypothesis}.
To simplify the notation, let 
\[
V \eqdef \valueof{\mdp,\formula}{s_0} - \eps - 3 \sum_{i=1}^k \eps_i.
\]
By definition, $\formula$ implies $\F \G \, A_{k+1}$, giving us 
\begin{multline*}
\probm_{\mdp, s_0, \sigma}\left( 
\varphi
\cap \bigcap_{i=1}^k ( \F^{\leq n_i} \, (\G A_i \cap \G\neg S_{i-1}) \cap \G^{\leq n_i}\, S_i )
\right) 
= \\
\probm_{\mdp, s_0, \sigma}\left( 
\varphi
\cap \bigcap_{i=1}^k ( \F^{\leq n_i} \, (\G A_i \cap \G \neg S_{i-1}) \cap \G^{\leq n_i}\, S_i )
\cap \F \G \, A_{k+1}
\right).
\end{multline*}
Since $\formula$ implies $\transience$,
we can instantiate \Cref{lem:transientdrift}(1) with parameters $\psi = \formula
\cap \bigcap_{i=1}^k ( \F^{\leq n_i} \, (\G A_i \cap \G \neg S_{i-1}) \cap \G^{\leq n_i}\, S_i )
\cap \F \G \, A_{k+1}$,
$X = S_k$ and $\delta = \eps_{k+1}$ to obtain an $\ell \in \N$ such that
\[
\probm_{\mdp, s_0, \sigma}\left( 
\formula
\cap \bigcap_{i=1}^k ( \F^{\leq n_i} \, (\G A_i \cap \G \neg S_{i-1}) \cap \G^{\leq n_i}\, S_i )
\cap \F \G \, A_{k+1} \cap \G^{\geq \ell}\, \neg S_k
\right) \geq V - \eps_{k+1}.
\]
Now we use continuity of measures and the definition of $\F$ to obtain that 
\begin{equation*}
\begin{split}
& \probm_{\mdp, s_0, \sigma}\left( 
\formula
\cap \bigcap_{i=1}^k ( \F^{\leq n_i} \, (\G A_i \cap \G\neg S_{i-1}) \cap \G^{\leq n_i}\, S_i )
\cap \F \G \, A_{k+1} \cap \G^{\geq \ell}\, \neg S_k
\right)\\
& = 
\lim_{j \to \infty} \probm_{\mdp, s_0, \sigma}\left( 
\formula
\cap \bigcap_{i=1}^k ( \F^{\leq n_i} \, (\G A_i \cap \G \neg S_{i-1}) \cap \G^{\leq n_i}\, S_i )
\cap \F^{\leq j} \G \, A_{k+1} \cap \G^{\geq \ell}\, \neg S_k
\right).
\end{split}
\end{equation*}
Then, using the definition of a limit, we can choose an $m \in \N$ such that 
\begin{equation*}
\begin{split}
& \probm_{\mdp, s_0, \sigma}\left( 
\formula
\cap \bigcap_{i=1}^k ( \F^{\leq n_i} \, (\G A_i \cap \G\neg S_{i-1}) \cap \G^{\leq n_i}\, S_i )
\cap \F^{\leq m} \G \, A_{k+1} \cap \G^{\geq \ell}\, \neg S_k
\right) 
\\ 
& \geq  V - 2\eps_{k+1}.
\end{split}
\end{equation*}
We know that $\G^{\geq \ell}$ is equivalent to $\F^{\leq \ell} \G$, so 
$\G^{\geq \ell}\, \neg S_k = \F^{\leq \ell} \G \, \neg S_k$. Then, setting $n_{k+1} \eqdef \max \{ m, \ell \}$
we obtain that
\begin{equation*}
\begin{split}
& \probm_{\mdp, s_0, \sigma}\left( 
\formula
\cap \bigcap_{i=1}^k ( \F^{\leq n_i} \, (\G A_i \cap \G\neg S_{i-1}) \cap \G^{\leq n_i}\, S_i )
\cap \F^{\leq n_{k+1}} \, (\G A_{k+1} \cap \G\neg S_k )
\right) \\ 
& \geq  V - 2\eps_{k+1}.
\end{split}
\end{equation*}
Finally, we instantiate \Cref{lem:transientdrift}(2) with $n = n_{k+1}$ and $\delta = \eps_{k+1}$ to obtain 
a finite set of states $S_{k+1}$ such that
\begin{multline*}
\probm_{\mdp, s_0, \sigma}\Bigg( 
\formula
\cap \bigcap_{i=1}^k ( \F^{\leq n_i}\, (\G A_i \cap \G\neg S_{i-1}) \cap \G^{\leq n_i}\, S_i )
\cap \\ \F^{\leq n_{k+1}} \, (\G A_{k+1} \cap \G\neg S_k )
\cap \G^{\leq n_{k+1}} \, S_{k+1}
\Bigg) \geq V - 3 \eps_{k+1}
\end{multline*}
This yields \eqref{eq:indhypothesis} for $k+1$, thus concluding the induction step.
By using continuity of measures and the definition of the $\eps_i$, we
obtain \eqref{eq:transientepsdrift} from \eqref{eq:indhypothesis}. 

\bigskip
Note that $\bigcap_{i \in \N} ( \F^{\leq n_i} \, (\G A_i \cap \G \neg S_{i-1}) \cap \G^{\leq n_i}\, S_i )$
is equal to $\bigcap_{i \in \N} \G^{\geq n_i} A_i \cap \G^{\geq n_i} \neg S_{i-1} \cap \G^{\leq n_{i+1}} \, S_{i+1}$.
(Since $n_0 = 0$, the term $\G^{\leq n_0} S_0$ holds trivially.)
Hence $\sigma$ also satisfies
\begin{equation}\label{eq:goodforgoodi}
\probm_{\mdp, s_0, \sigma}\left( 
\varphi
\cap
\bigcap_{i \in \N} \G^{\geq n_i} A_i \cap \G^{\geq n_i} \neg S_{i-1} \cap \G^{\leq n_{i+1}} \, S_{i+1}
\right)
\geq \valueof{\mdp,\formula}{s_0} - 2 \eps
\end{equation}
Let $B_i \eqdef \{(x \to y) \mid x,y \in S_{i+1} \setminus S_{i-1}\}$
be the set of transitions with source and target inside $S_{i+1} \setminus S_{i-1}$.
Let ${\sf Good}_i \eqdef A_i \cap B_i$ for all $i$ and 
${\sf Good} \eqdef \bigcup_{i \in \N} {\sf Good}_i$. 
Our choice of ${\sf Good}_i$ is such that 
$\bigcap_{i \in \N} \G^{\geq n_i} A_i \cap \G^{\geq n_i} \neg S_{i-1} \cap \G^{\leq n_{i+1}} \, S_{i+1}$
implies
$\bigcap_{i \in \N} \G^{[n_{i},n_{i+1}]} \, {\sf Good}_i$ for all $i$.
Furthermore, $\bigcap_{i \in \N} \G^{[n_{i},n_{i+1}]} \, {\sf Good}_i \subseteq \G \, {\sf Good} $.
Hence we obtain from \eqref{eq:goodforgoodi} that 
\begin{equation}\label{eq:goodforgood}
\valueof{\mdp,\transience \cap \G \, {\sf Good}}{s_0}
\geq
\probm_{\mdp, s_0, \sigma}\left( 
\transience \cap \G \, {\sf Good} \right) 
\geq 
\valueof{\mdp,\formula}{s_0} - 2 \eps
\end{equation}
Now we show that 
\begin{equation}\label{eq:GoodinAi}
\transience \cap \G \, {\sf Good} \subseteq \bigcap_{i \in \N} \F \G \, A_i
\end{equation}
Consider any run $\rho \in \transience \cap \G \, {\sf Good}$.
Since $S_{i+1}$ is finite, $B_i$ is finite, and thus ${\sf Good}_i$ is finite.
By transience, $\rho$ can visit each finite set ${\sf Good}_i$
only finitely often.
Hence $\rho$ stays in ${\sf Good}$ forever whilst visiting sets ${\sf Good}_i$
for ever greater $i$.
Since ${\sf Good}_i \subseteq A_i$, $\rho \in \G \, {\sf Good}$
and the sets $A_i$ are monotone decreasing (wrt.\ set inclusion),
it follows that $\rho$ satisfies $\bigcap_{i \in \N} \F \G \, A_i$ as desired.

By \Cref{lem:transience-plus-safety},
there exists a memoryless deterministic (MD)
$\eps$-optimal strategy $\sigma^*$ from $s_0$ for 
the objective $\transience \cap \G \, {\sf Good}$.
Now we show that $\sigma^*$ is $3\eps$-optimal for $\formula$.
\begin{align*}
\probm_{\mdp, s_0, \sigma^*}(\formula) & = \probm_{\mdp, s_0, \sigma^*}(\transience \cap \bigcap_{i \in \N} \F \G \, A_i) \\
& \ge \probm_{\mdp, s_0, \sigma^*}(\transience \cap \G \, {\sf Good}) & \eqref{eq:GoodinAi} \\
& \geq \valueof{\mdp,\transience \cap \G \, {\sf Good}}{s_0} - \eps & \text{$\sigma^*$ is $\eps$-optimal} \\
& \geq \valueof{\mdp, \formula}{s_0} - 3 \eps & \eqref{eq:goodforgood}
\end{align*}
Since $\eps$ can be chosen arbitrarily small, the result follows.
\end{proof}

\begin{corollary}\label{infepsupperpp}
Given an MDP $\mdp$ and initial state $s_{0}$, $\eps$-optimal strategies for $\bigcap_{i \in \N} \F \G A_i$
can be chosen
\begin{enumerate}
\item MD if $\mdp$ is universally transient.
\item Det(SC) in general (even if $\mdp$ is not universally transient).
\end{enumerate}
\end{corollary}
\begin{proof}
  Towards item 1., if $\mdp$ is universally transient then
  (modulo a nullset under every strategy) $\bigcap_{i \in \N} \F \G A_i$
  coincides with $\transience \cap \bigcap_{i \in \N} \F \G A_i$, and
  thus the result follows from \Cref{thm:liminfunivtrans}.

Towards item 2., consider the MDP $S(\mdp)$ that is derived from $\mdp$ by encoding the step
counter from $\state_0$ into the states; cf.~\Cref{def:encodestep}.
$S(\mdp)$ is trivially universally transient.
Then item 1. yields an $\eps$-optimal MD strategy in $S(\mdp)$.
Finally, \Cref{steptopoint} yields an $\eps$-optimal Det(SC) strategy from $\state_0$ in $\mdp$.
\end{proof}

\begin{corollary}\label{infoptupperpp}
Given an MDP $\mdp$ and initial state $s_{0}$, optimal strategies for $\bigcap_{i \in \N} \F \G A_i$, 
where they exist, can be chosen
\begin{enumerate}
\item MD if $\mdp$ is universally transient.
\item Det(SC) in general (even if $\mdp$ is not universally transient).
\end{enumerate}
\end{corollary}
\begin{proof}
  Towards item 1., assume that $\mdp$ is universally transient.
  We apply \Cref{infepsupperpp}.1 to obtain an $\eps$-optimal MD strategy from $s_0$.
  Since $\bigcap_{i \in \N} \F \G A_i$ is a shift invariant objective,
  \Cref{epsilontooptimal}.\ref{item:epstooptimal} yields an MD
strategy that is optimal from every state of $\mdp$ that has an optimal
strategy.

Towards item 2., if $\mdp$ is not universally transient,
then we work in $S(\mdp)$ which is universally transient and 
apply item 1.\ to obtain optimal MD strategies from every
state of $S(\mdp)$ that has an optimal strategy.
By \Cref{steptopoint} we can translate
this MD strategy on $S(\mdp)$ back to a Det(SC) strategy in $\mdp$,
which is optimal for $\bigcap_{i \in \N} \F \G A_i$ from $s_{0}$ (provided that $s_0$ admits
any optimal strategy at all).
\end{proof}

\subsection{\texorpdfstring{Upper bound for $\bigcap_{i \in \N} \F \G \, A_i$  in finitely branching MDPs}
  {Upper bound in finitely branching MDPs}
}

We show that, in the special case of finitely branching MDPs, MD strategies
suffice for $\bigcap_{i \in \N} \F \G \, A_i$.
First we need the following auxiliary lemma, which holds only for finitely branching MDPs.

\bigskip
\begin{lemma}\label{lem:fbavoid}
Given a finitely branching countable MDP $\mdp$, a subset $T \subseteq \to$ of
the transitions and a state $\state$, we have
\[
\valueof{\mdp,\neg\F T}{\state} < 1
\ \Rightarrow\ \exists k \in \N.\,
\valueof{\mdp,\neg\F^{\le k} T}{\state} < 1 
\]
i.e., if it is impossible to completely avoid $T$ then
there is a bounded threshold $k$ and a fixed nonzero
chance of seeing $T$ within $\le k$ steps, regardless of the strategy.
\end{lemma}

\begin{proof}
If suffices to show that
$\forall k \in \N.\, \valueof{\mdp,\neg\eventually^{\le k} T}{\state} =1$
implies $\valueof{\mdp,\neg\eventually T}{\state} = 1$.
Since $\mdp$ is finitely branching, the state $\state$ has only finitely many
successors $\{\state_1,\dots,\state_n\}$.

Consider the case where $\state$ is a controlled state.
If we had the property for all $i$ with $1 \le i\le n$ there exists a $ k_i \in \N$ such that
$\valueof{\mdp,\neg\eventually^{\le k_i} T}{\state_{i}} < 1$,
then we would have
$\valueof{\mdp,\neg\eventually^{\le k} T}{\state} < 1$
for $k=(\max_{1 \le i \le n} k_i)+1$
which contradicts our assumption.
Thus there must exist an $i \in \{1,\dots,n\}$ with
$\forall k \in \N.\, \valueof{\mdp,\neg\eventually^{\le k} T}{\state_i} =1$.
We define a strategy $\zstrat$ that chooses the successor state $s_i$ when in
state $\state$.

Similarly, if $\state$ is a random state, we must have
$\forall k \in \N.\, \valueof{\mdp,\neg\eventually^{\le k} T}{\state_i} =1$
for all its successors $s_i$.

By using our constructed strategy $\zstrat$, we obtain
$\probm_{\mdp,\state,\zstrat}(\neg\eventually T)=1$ and therefore
$\valueof{\mdp,\neg\eventually T}{\state} = 1$ as required.
\end{proof}


The following theorem is very similar to \cite[Theorem 27]{MM:CONCUR2021}.
Here we present a much shorter proof that uses \Cref{thm:liminfunivtrans}.

\bigskip
\begin{theorem}\label{finpointpayoff}
  Consider a finitely branching countable MDP $\mdp =\mdptuple$
  with initial state $s_{0}$ and a $\bigcap_{i \in \N} \F \G A_i$ objective.
  For every $\eps >0$ there exists an $\eps$-optimal MD strategy from $s_0$.
\end{theorem}
\begin{proof}[Proof outline]
The main idea is to do a case distinction between transient and non-transient runs.
Under non-transience, the objective $\bigcap_{i \in \N} \F \G A_i$
can only be satisfied if one eventually enters a certain totally safe subspace
where one can win almost surely with an MD strategy.
For the other case, under transience, we obtain an $\eps$-optimal MD strategy
from \Cref{thm:liminfunivtrans}.
The proof handles this case distinction via the construction of a
modified MDP $\mdp'$ that folds the first case into the second.
\end{proof}
\begin{proof}
Let $\eps >0$.
We begin by partitioning the state space into two sets, $S_{\text{safe}}$ and
$S \setminus S_{\text{safe}}$.
The set $S_{\text{safe}}$ is the subset of states which is surely winning for
the safety objective of only using transitions in $\bigcap_{i \in \N} A_i$.
Since $\mdp$ is finitely branching, there exists a uniformly optimal MD
strategy $\sigma_{\text{safe}}$ for this safety objective
\cite{Puterman:book,KMSW2017}.

We construct a new MDP $\mdp'$ by modifying $\mdp$. We create a gadget
$G_{\text{safe}}$ composed of a sequence of new controlled states
$x_0 , x_1 , x_2 , \dots$ with transitions 
$x_0 \transition x_1 ,\, x_1 \transition x_2 ,$ etc. 
Let $X \eqdef \bigcup_{i \in \N} \{ x_i \transition x_{i+1} \}$. 
We now define sets $B_i \eqdef A_i \cup X$. Hence any run entering $G_{\text{safe}}$ is winning for $\bigcap_{i \in \N} \F \G \, B_i$. 
We insert $G_{\text{safe}}$ into $\mdp$ by replacing all incoming transitions
to $S_{\text{safe}}$ with transitions that lead to $x_{0}$. The idea behind
this construction is that when playing in $\mdp$, once you reach a state in
$S_{\text{safe}}$, you can win surely by playing the
optimal MD strategy $\sigma_{\text{safe}}$ for safety.
So we replace $S_{\text{safe}}$ with the surely winning gadget $G_{\text{safe}}$.
Thus
\begin{equation}\label{eq:valuesmmprime}
\valueof{\mdp,\bigcap_{i \in \N} \F \G A_i}{s_0} = \valueof{\mdp',\bigcap_{i \in \N} \F \G \, B_i}{s_0} 
\end{equation}
and if an $\eps$-optimal MD strategy exists in $\mdp$,
then there exists a corresponding one in $\mdp'$, and vice-versa.

In the next step we argue that under \emph{every} strategy
$\sigma'$ from $s_0$ in $\mdp'$ the attainment for
$\bigcap_{i \in \N} \F \G \, B_i$ and
$\transience \cap \bigcap_{i \in \N} \F \G \, B_i$ coincide, i.e.,
\begin{claim}\label{eqliminfpptransience}
\[
\forall \sigma'.\, \probm_{\mdp',s_0,\sigma'}(\bigcap_{i \in \N} \F \G \, B_i)
= \probm_{\mdp',s_0,\sigma'}(\transience \cap \bigcap_{i \in \N} \F \G \, B_i).
\]
\end{claim}
\begin{proof}
The $\ge$ inequality holds trivially, since 
$\transience \cap \bigcap_{i \in \N} \F \G \, B_i \subseteq \bigcap_{i \in \N}\F \G \, B_i$.

Towards the $\le$ inequality, it suffices to show that
\begin{equation}\label{eq:liminf-part-transience}
\forall \sigma'.\ \probm_{\mdp',s_0,\sigma'}(\bigcap_{i \in \N} \F \G \, B_i \cap \overline{\transience})=0.
\end{equation}
Let $\sigma'$ be an arbitrary strategy from $s_0$ in $\mdp'$ and
$\playset$ be the set of all runs induced by it.
For every $s \in S$, let $\playset_{s} \defeq \{\rho \in \playset \mid \rho \text{ satisfies } \G \F (s) \}$
be the set of runs seeing state $s$ infinitely often.
In particular, any run $\rho \in \playset_{s}$ is not transient.
Indeed, $\overline{\transience} = \bigcup_{s \in S} \playset_{s}$.
We want to show that for every state $s \in S$ and strategy $\sigma'$
\begin{equation}\label{eq:liminf-part-transience-s}
\probm_{\mdp',s_0,\sigma'}(\bigcap_{i \in \N} \F \G \, B_i \cap \playset_{s}) = 0.
\end{equation}
Since all runs visiting a state in $G_{\text{safe}}$ are transient, any
$\playset_{s}$ with $s$ in $G_{\text{safe}}$ must be empty, and thus
\eqref{eq:liminf-part-transience-s} holds for these cases.

Now we consider the remaining cases of $\playset_{s}$ where $s$ is not in $G_{\text{safe}}$.
Let $T \eqdef \{t \in \transition_{\mdp'} \mid\  t \notin \bigcap_{i \in \N} B_i \}
= \bigcup_{i \in \N} \overline{B_i}$.

We now show that $\valueof{\mdp',\neg\F T}{s} < 1$
by assuming the opposite and deriving a contradiction.
Assume that $\valueof{\mdp',\neg\F T}{s} = 1$.
The objective $\neg\F T$ is a safety objective.
Thus, since $\mdp'$ is finitely branching, there exists a strategy
from $s$ that surely avoids $T$ (always pick an optimal
move) \cite{Puterman:book,KMSW2017}.
(This would not hold in infinitely branching MDPs where optimal moves might
not exist.)
However, by construction of $S_{\text{safe}}$ and $\mdp'$, this implies that 
$s$ is in $G_{\text{safe}}$. Contradiction.
Thus $\valueof{\mdp',\neg\F T}{s} < 1$.

Since $\mdp'$ is finitely branching, we can apply \Cref{lem:fbavoid}
and obtain that there exists a threshold $k_s$ such that
$\valueof{\mdp',\neg\F^{\le k_s} T}{s} < 1$.
Therefore
$\delta_s \eqdef 1 - \valueof{\mdp',\neg\F^{\le k_s} T}{s} >0$.
Thus, under every strategy, upon visiting $s$ there is a chance $\ge \delta_s$
of seeing a transition in $T = \bigcup_{i \in \N} \overline{B_i}$
within the next $\le k_s$ steps.
Let $T^s \subseteq T$ be the subset
of transitions in $T$ that can be reached
in $\le k_s$ steps from $s$. Since $\mdp'$ is finitely branching, $T^s$ is finite.
Since the sequence of sets $\{B_i\}_{i\in \N}$ is monotone decreasing, the sequence
$\{\overline{B_i}\}_{i\in \N}$ is monotone increasing.
Hence, for every transition $t \in T^s$, there is a minimal index $i$ such
that $t \in \overline{B_i}$.
Moreover, since $T^s$ is finite, the maximum (over $t \in T^s$) of these
minimal indices is bounded, i.e.,
$\ell_s \eqdef \max_{t \in T^s}\ \min\{ i \ \mid\ t \in T^s \text{ and } t \in \overline{B_i} \} < \infty$.

Thus, under \emph{every} strategy, upon visiting $s$ there is a chance $\ge \delta_s$
of seeing a transition in $\overline{B_{\ell_s}}$ within the next $\le k_s$
steps, i.e.,
\begin{equation}\label{eq:finpointpayoff-maxmin}
\forall \sigma'\ \probm_{\mdp',s,\sigma'}(\F^{\le k_s} \overline{B_{\ell_s}})
\ge \delta_s > 0
\end{equation}
Define $\playset_{s}^{i} \defeq \{ \rho \in \playset \mid \rho \text{ sees $s$
at least $i$ times} \}$, so we get $\playset_{s} = \bigcap_{i \in \N} \playset_{s}^{i}$.
We obtain
\begin{align*}
& \sup_{\sigma'}\probm_{\mdp',s_0,\sigma'}(\bigcap_{i \in \N} \F \G \, B_i \cap \playset_{s}) \\
& \le \sup_{\sigma'}\probm_{\mdp',s_0,\sigma'}(\F\G \, B_{\ell_s} \cap \playset_{s}) & \text{set inclusion}\\
&
= \sup_{\sigma'}\lim_{n \to \infty}\probm_{\mdp',s_0,\sigma'}(\F^{\le n}\G \, B_{\ell_s} \cap \playset_{s}) &  \text{continuity of measures}\\
& \le \sup_{\sigma''}\probm_{\mdp',s,\sigma''}(\G \, B_{\ell_s} \cap \playset_{s}) & \text{$s$ visited after $>n$ steps}\\
& = \sup_{\sigma''} \probm_{\mdp',s,\sigma''}(\G \, B_{\ell_s} \cap \bigcap_{i \in \N} \playset_{s}^{i}) & \text{def.\ of $\playset_{s}^{i}$} \\
& = \sup_{\sigma''} \lim_{i \to \infty}\probm_{\mdp',s,\sigma''}(\G \, B_{\ell_s} \cap \playset_{s}^{i})
&  \text{continuity of measures} \\
& \le \lim_{i \to \infty}(1-\delta_s)^i = 0 & \text{by def.\ of $\playset_{s}^{i}$ and \eqref{eq:finpointpayoff-maxmin}}
\end{align*}
and hence \eqref{eq:liminf-part-transience-s}.
From this we obtain
\begin{equation*}
\begin{split}
\probm_{\mdp',s_0,\sigma'}(\bigcap_{i \in \N} \F \G \, B_i \ \cap \ \overline{\transience}) & = 
\probm_{\mdp',s_0,\sigma'}(\bigcap_{i \in \N} \F \G \, B_i \ \cap \ \bigcup_{s \in S} \playset_{s}) \\ 
& \le 
\sum_{s \in S} \probm_{\mdp',s_0,\sigma'}(\bigcap_{i \in \N} \F \G \, B_i \ \cap \ \playset_{s})=0
\end{split}
\end{equation*}
and thus \eqref{eq:liminf-part-transience} and \Cref{eqliminfpptransience}.
\end{proof}

By \Cref{thm:liminfunivtrans}, there exists an $\eps$-optimal MD
strategy $\widehat{\sigma}$ from $s_0$ for $\transience \cap \bigcap_{i \in \N} \F \G B_i$ in $\mdp'$, i.e.,
\begin{equation}\label{eq:transience-eps}
\probm_{\mdp',s_0,\hat{\sigma}}(\transience \cap \bigcap_{i \in \N} \F \G B_i)
\ge \valueof{\mdp',\transience \cap \bigcap_{i \in \N} \F \G B_i}{s_0}
- \eps.
\end{equation}
We construct an MD strategy $\sigma^{*}$ in $\mdp$ which plays
like the MD strategy $\sigma_{\text{safe}}$ in $S_{\text{safe}}$ and plays
like the MD strategy $\widehat{\sigma}$ everywhere else.
\begin{align*}
\probm_{\mdp,\state_0,\zstrat^{*}}(\bigcap_{i \in \N} \F \G \, A_i) 
&=  \probm_{\mdp',\state_0,\hat{\zstrat}}(\bigcap_{i \in \N} \F \G \, B_i) &
\text{def.\ of $\mdp'$, $\zstrat^{*}$ and $\sigma_{\text{safe}}$}\\
& = \probm_{\mdp',\state_0,\hat{\zstrat}}(\transience \cap \bigcap_{i \in \N} \F \G \, B_i)  & \text{by \Cref{eqliminfpptransience}}\\
& \ge \valueof{\mdp',\transience \cap \bigcap_{i \in \N} \F \G \, B_i}{s_0} - \eps  & \text{by \eqref{eq:transience-eps}}\\
& = \valueof{\mdp',\bigcap_{i \in \N} \F \G \, B_i}{s_0} - \eps   & \text{by \Cref{eqliminfpptransience}}\\
& = \valueof{\mdp,\bigcap_{i \in \N} \F \G \, A_i}{s_0} - \eps & \text{by \eqref{eq:valuesmmprime}}
\end{align*}
Hence $\sigma^{*}$ is an $\eps$-optimal MD strategy for $\bigcap_{i \in \N} \F \G \, A_i$ from
$s_0$ in $\mdp$. 
\end{proof}


\begin{corollary}\label{finoptupperpp}
Given a finitely branching MDP $\mdp$ and a 
$\bigcap_{i \in \N} \F \G A_i$ objective, 
there exists a single MD strategy that is optimal from every state that has an
optimal strategy.
\end{corollary}
\begin{proof}
Since $\bigcap_{i \in \N} \F \G A_i$ is shift invariant, the result follows
from \Cref{finpointpayoff} and \Cref{epsilontooptimal}.
\end{proof}

\subsection{Lower bound}

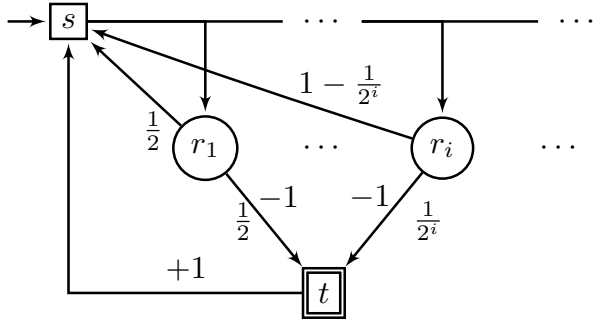
\begin{figure}[tbp]
\begin{center}
\scalebox{1.2}{
\begin{tikzpicture}[>=latex',shorten >=1pt,node distance=1.9cm,on grid,auto,
roundnode/.style={circle, draw,minimum size=1.5mm},
squarenode/.style={rectangle, draw,minimum size=2mm},
diamonddnode/.style={diamond, draw,minimum size=2mm}]

\node [squarenode,initial,initial text={}] (s) at(0,0) [draw]{$s$};

\node[roundnode] (r1)  [below right=1.4cm and 1.5cm of s] {$r_1$};
\node[roundnode] (r3)  [draw=none,right=1.3 of r1] {$\cdots$};
\node[roundnode] (r4)  [right=1.3cm of r3] {$r_i$};
\node[roundnode] (r5)  [draw=none,right=1.3cm of r4] {$\cdots$};

\node [squarenode,double,inner sep = 4pt] (t)  [below=1.6cm of r3] {$ t $};

\draw [->] (s) -- ++(1.5,0) -- (r1);
\node[roundnode] (d)  [draw=none,right=2.8 of s] {$\cdots$};
\draw[-] (s) -- (d);
\draw [->] (d) -- ++(1.3,0) -- (r4);
\node[roundnode] (dd)  [draw=none,right=2.8 of d] {$\cdots$};
\draw[-] (d) -- (dd);

\path[->] (r1) edge node [midway,left] {$\frac{1}{2}$} node[near start, right] {$-1$} (t);
\path[->] (r4) edge node [midway,right=.2cm] {$\frac{1}{2^{i}}$} node[near start, left] {$-1$} (t);
\path[->] (r1) edge  node[pos=0.3,below] {$\frac{1}{2}$} (s);   
\path[->] (r4) edge [bend left=1] node [pos=0.2,above] {$1-\frac{1}{2^{i}}$} (s);  

\draw [->] (t) -- node[midway, above] {$+1$} ++(-2.8,0) -- (s);
\end{tikzpicture}
}
\caption{ 
We present an infinitely branching MDP adapted from \cite[Figure 3]{KMSW2017} and augmented with a reward structure.
All of the edges carry reward $0$ except the edges entering $t$ that carry reward $-1$ and the edge from $t$ to $s$ carries reward $+1$.
A strategy is therefore optimal for $\liminfppobj$ if and only if it satisfies $\cobuchi{t}$ almost surely. This requires infinite memory.
Note that in the context of \Cref{infepsupperpp}, this example only works because it is not universally transient.
}
\label{liminfinfbranchlower}
\end{center}
\end{figure}

We present the lower bound for the strategy complexity in terms of $\liminfppobj$, because reasoning about counterexamples with transition rewards is very natural.

\begin{proposition}\label{infbranchsteplower}
There exists an infinitely branching MDP $\mdp$ as in \Cref{liminfinfbranchlower} with initial state $s$ such that
\begin{itemize}
\item every FR strategy $\sigma$ is such that $\probm_{\mdp, s, \sigma} (\liminfppobj) = 0$
\item there exists a strategy $\sigma$ such that $\probm_{\mdp, s, \sigma} (\liminfppobj) = 1$.
\end{itemize}
Hence, optimal (and even almost-surely winning) strategies and $\eps$-optimal
strategies for $\liminfppobj$ require infinite memory.
\end{proposition}
\begin{proof}
  This follows directly from \cite[Theorem 4]{KMSW2017}, since
  in \Cref{liminfinfbranchlower},
  $\liminfppobj$ coincides with the co-B\"{u}chi objective to visit state $t$
  only finitely often.
\end{proof}

\section{Expected Payoff Objectives}\label{sec:expected}

We show how upper and lower bounds on the strategy complexity of
optimal strategies for $\limsupppexp$ and $\liminfppexp$ follow directly
from results on the $\limsupppobj$ and $\liminfppobj$ objectives, respectively. 
This allows us to solve two open problems from
\cite[p.43 and p.53]{Sudderth:2020}.

\begin{table}[tbp]
\begin{tabular}{|l||l|l|}
\hline
                     & Optimal $\limsupppexp$ & Optimal $\liminfppexp$ \\ \hline
Finitely Branching   & Rand(Positional) or Det(SC) \ref{thm:explimsupupper}   \ref{thm:explimsuplower}      & MD    \ref{thm:expliminffinbranch}                 \\ \hline
Infinitely Branching & Rand(Positional) or Det(SC) \ref{thm:explimsupupper} \ref{thm:explimsuplower}        & Det(SC) \ref{thm:expliminfinfbranch}  \ref{thm:expliminflower}             \\ \hline
\end{tabular}
\label{table:results}
\caption{Strategy complexity of optimal strategies for the expected $\limsup$
  and $\liminf$.}
\end{table}

\subsection{Optimal Strategies for the Expected \texorpdfstring{$\limsup$}{lim
  sup}}\label{sec:limsupexpected}

\begin{theorem}\label{thm:explimsupupper}
Let $\mdp$ be a countable MDP with initial state $\state_0$.
If an optimal strategy for $\limsupppexp$ exists, then there also exists an
optimal memoryless randomized (MR) strategy and an optimal deterministic
Markov (Det(SC)) strategy.
\end{theorem}
\begin{proof}
This follows from \Cref{cor:GFsc,cor:GFmr} and \Cref{thm:explimsuptothreshold}.
\end{proof}

In \cite[p.53]{Sudderth:2020} Sudderth poses the question whether `the existence of an optimal $\limsup$ strategy at every 
state always implies that there exists an optimal, possibly randomized, stationary strategy'. 
With \Cref{thm:explimsupupper}, we answer this question in the affirmative.
It is worth noting that Dubins and Savage
\cite[Example 4, p.59]{DubbinsSavage:2014} present a counterexample in the framework of finitely
additive probability theory.
Indeed our proof of \Cref{thm:nestedbuchimr} requires countably additive
probability theory, since the memoryless randomized (MR) strategy is constructed as a countably infinite
combination of memoryless deterministic (MD) strategies.
The next result gives us a lower bound on the strategy complexity. 

\bigskip
\begin{proposition}\label{thm:explimsuplower}
There exists a finitely branching countable MDP $\mdp$ with rational
rewards as in \Cref{limsuplowerdetf}
with initial state $\state_0$ such that
\begin{enumerate}
\item
$\exists \hat{\zstrat}\, \expectval_{\mdp, s_0, \hat{\zstrat}}(\limsup_{PP}) = 0 = \valueof{\limsupppexp}{s_0}$, 
i.e., $\hat{\sigma}$ is optimal.
\item
For every Det(F) strategy $\zstrat$ we have 
$\expectval_{\mdp, \state_{0},\zstrat}(\limsup_{PP}) < 0 = \valueof{\limsupppexp}{s_0}.$
\end{enumerate}
So optimal strategies, when they exist, cannot be chosen Det(F).
\end{proposition}

\begin{proof}
This follows directly from \Cref{thm:limsuplowerdetf}.
\end{proof}

\subsection{Optimal Strategies for the Expected \texorpdfstring{$\liminf$}{lim
     inf}}

The following theorem solves the open question from
\cite[p.43, last par.]{Sudderth:2020} about the strategy complexity of optimal
strategies for the expected $\liminf$.

\medskip
\begin{theorem}\label{thm:expliminfinfbranch}
Let $\mdp$ be a countable (possibly infinitely branching) MDP.
Optimal strategies for $\liminfppexp$, when they exist, can be chosen as
  \begin{itemize}
  \item
    MD, if $\mdp$ is universally transient.
  \item
   Deterministic Markov (Det(SC)) in general.
\end{itemize}
\end{theorem}
\begin{proof}
This follows from \Cref{thm:explimsuptothreshold} and \Cref{infoptupperpp}.
\end{proof}

In the special case of \emph{finitely branching} MDPs,
optimal strategies can be simpler.

\bigskip
\begin{theorem}\label{thm:expliminffinbranch}
  Let $\mdp$ be a countable \emph{finitely branching} MDP.
  Optimal strategies for $\liminfppexp$, when they exist, can be chosen MD.
\end{theorem}
\begin{proof}
This follows from \Cref{thm:explimsuptothreshold} and \Cref{finoptupperpp}.
\end{proof}

For infinitely branching MDPs, the following lower bound holds.

\smallskip
\begin{proposition}\label{thm:expliminflower}
There exists an infinitely branching MDP $\mdp$ as in \Cref{liminfinfbranchlower} with reward implicit in the state and initial state $s$ such that
\begin{itemize}
\item
  there is an optimal strategy $\sigma$ with $\expectval_{\mdp, s,
    \sigma} (\liminf_{PP}) = 1 = \valueof{\liminfppexp}{s}$.
\item
  every FR strategy $\sigma$ is such that $\expectval_{\mdp, s, \sigma} (\liminf_{PP}) = -1$
\end{itemize}
Hence, optimal strategies for $\liminfppexp$ require infinite memory.
\end{proposition}
\begin{proof}
This follows directly from \Cref{infbranchsteplower}.
\end{proof}

\subsection{Epsilon-optimal Strategies for the Expected
  \texorpdfstring{$\limsup$ and  $\liminf$}{lim sup and lim inf}}

The results above concern the strategy complexity of \emph{optimal} strategies
for the expected $\limsup$ (resp.\ $\liminf$), where they exist.
Optimal strategies need not always exist, but if they do, then their strategy
complexity might be lower than that of $\eps$-optimal strategies.

The strategy complexity of $\eps$-optimal strategies for $\limsupppexp$ and
$\liminfppexp$ in countable MDPs is an open question.
However, it is known for the special case of countable MDPs where all
daily rewards are either $0$ or $1$,
and the lower bounds for this special case 
trivially carry over to the general case.

In the special case with daily rewards either $0$ or $1$,
the $\limsupppexp$ objective corresponds to the B\"uchi objective (maximize
the probability of seeing transitions with reward $1$ infinitely often) and
$\liminfppexp$ corresponds to the co-B\"uchi objective (maximize
the probability of seeing transitions with reward $0$ only finitely often).
For the B\"uchi objective, $\eps$-optimal strategies can be chosen as 
Det(SC + 1-bit), while Markov strategies (Rand(SC)) or finite memory
strategies (Rand(F)) are not sufficient \cite{KMST:ICALP2019}.
For the co-B\"uchi objective, $\eps$-optimal strategies can be chosen as 
Det(SC), but not Rand(F), in general. However, if the MDP is finitely
branching, then $\eps$-optimal strategies for the co-B\"uchi objective can be chosen
as Det(Positional) \cite{KMSW2017,KMST2020c}.
These upper bounds for the co-B\"uchi objective do not carry over from MDPs to 2-player
stochastic games, where infinite memory (beyond a step counter)
is required instead \cite[Remark 1]{KMST-DGA:2024}.

\section{Conclusion}\label{sec:conclusion}

Our results provide a complete picture of the strategy complexity
of $\limsup$ and $\liminf$ threshold objectives, and the corresponding
problem for optimal strategies for the expected $\limsup$ and $\liminf$.
They also highlight fundamental differences between the $\limsup$ and
$\liminf$ objectives.
Unlike for the $\liminf$ case, 
\begin{itemize}
\item
The memory requirements of strategies for $\limsup$ objectives
depend on whether the transition rewards are integers
or rationals/reals (\Cref{rem:integer-rewards} and \Cref{tab:limsup}).
\item
Randomization does make a difference for $\limsup$ objectives, e.g., strategies can sometimes trade a step
counter for randomization (\Cref{cor:GFsc} and \Cref{cor:GFmr}).
\item
For $\limsup$ objectives, the memory requirements of $\eps$-optimal
strategies differ from those of optimal strategies
(\Cref{tab:limsup}).
\item
For $\limsup$ objectives, the memory requirements of strategies do \emph{not} depend on whether the MDP is infinitely
branching or finitely branching. Nor does it depend on a particular branching
degree $\ge 2$ (\Cref{lem:branchingreplacement}).
\end{itemize}
Finally, as shown in \Cref{sec:expected}, the strategy complexity may depend on
whether one works in the finitely additive probability theory or in the countably additive one.

\backmatter

\section*{Declarations}

\begin{itemize}
\item \textbf{Funding.}
  This work has been supported by
  the Royal Society, grant IES$\backslash$R3$\backslash$213110.
\item \textbf{Competing interests.}
The authors declare they have no financial, or non-financial interests, and have no 
potential conflicts of interest to declare.
\end{itemize}

\newpage
\appendix
\begin{appendices}
  \section{Memory-based strategies}\label{app-def}

A \emph{memory-based strategy} $\zstrat$ of Maximizer is a strategy that
can be described by a tuple $(\memconfset, \memconf_0, \zstrat_\alpha, \zstrat_\memconf)$
where $\memconfset$ is the set of memory modes, $\memconf_0 \in \memconfset$
is the initial memory mode, and
the functions $\zstrat_\alpha$ and $\zstrat_\memconf$ describe
how successor states are chosen (at controlled states)
and how memory modes are updated (generally), respectively.

A play $\play = \state_0e_0\state_1e_1\cdots$ according to $\zstrat$ 
generates a sequence of memory
modes $\memconf_0, \dots, \memconf_t, \memconf_{t+1}, \dots$ from the given set
of memory modes $\memconfset$, where $\memconf_t$ is the memory mode at
time $t$.

If the current state $\state_t$ is a controlled state then
the strategy $\zstrat$ selects a distribution over the available successor
states of $s_t$ via function $\zstrat_\alpha$ that
depends only on the current state $\state_t$ and the memory $\memconf_t$,
i.e., $\zstrat_\alpha(\state_t, \memconf_t) \in \dist(\successors{\state_t})$.
The next state $\state_{t+1}$ is then chosen according to this distribution.

If the current state $\state_t$ is a random state then
the successor state $\state_{t+1}$ is chosen according to the pre-defined
distribution over $\successors{\state_t}$ of the MDP. However, the strategy can still observe this and
update its memory.

In either case, the next memory mode $\memconf_{t+1}$ of Maximizer is chosen
from the distribution given by function
$\zstrat_m$, that depends on the current memory mode and
on the observed outcome of the step from $\state_t$
to $\state_{t+1}$, i.e.,  
$\zstrat_m(\memconf_t, \state_t, \state_{t+1}) \in \dist(\memconfset)$.

A \emph{finite-memory strategy} is one where $\card{\memconfset} < \infty$.
A \emph{$k$-mode strategy} is a memory-based strategy with at most $k$
memory modes, i.e., $\card{\memconfset} \le k$.
A $2$-mode strategy is also called a \emph{1-bit strategy}.
A strategy is \emph{memoryless} (also called positional or stationary)
if $\card{\memconfset}=1$.
A strategy is called \emph{Markov} if it uses only a step counter but no
additional memory, i.e., $\memconfset = \N_0$ and $\memconf_n = n$.
A strategy is \emph{deterministic} (also called pure) if the distributions
chosen by $\zstrat_\alpha$ and $\zstrat_m$ are Dirac.
Otherwise, it is called \emph{randomized} (aka mixed).
Pure stationary strategies are also called MD (memoryless deterministic)
and mixed stationary strategies are also called MR (memoryless randomized).
Similarly, deterministic (aka pure) finite-memory strategies are also called FD,
and randomized (aka mixed) finite-memory strategies are also called FR.

  \section{Transition Rewards vs. State Rewards}\label{app:states-vs-transitions}

Transition based rewards and state based
rewards can be encoded into each other.

\begin{definition}
Given an MDP $\mdp = (\states,\zstates,\rstates,\transition,\probp,r)$ with state based rewards bounded between $+m$ and $-m$, $m \in \R$, we construct a modified MDP with transition based rewards $\mdp^t = (\states^t,\zstates^t,\rstates^t,\transition^t,\probp^t,r^t)$ as follows.
For every $s \in S$, we construct two new states $s_{\text{in}}$ and $s_{\text{out}}$ and a transition $t^s$ with reward $r(t^s) = r(s)$. I.e.\ we replace every instance of

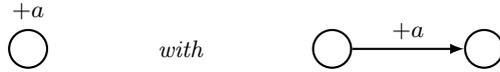
\begin{center}
    \begin{tikzpicture}
    
    \node[draw, circle, minimum height=5mm] (S1) at (0,0){};
    
    \node (T1) at (2,0){with};
    \node (T2) at (0,0.5){$+a$};
    
    \node[draw, circle, minimum height=5mm] (S2) at (4,0){};
    \node[draw, circle, minimum height=5mm] (S3) at (6,0){};

     \draw[->,>=latex] (S2) edge node[above, midway]{$+a$}  (S3);

    \end{tikzpicture}
\end{center}    

Formally, we construct two copies of $S$, $S_{\text{in}} \eqdef \{ s_{\text{in}} \mid s \in S \}$ 
and $S_{\text{out}} \eqdef \{ s_{\text{out}} \mid s \in S \}$. 
Similarly, define  $S_{\Box \text{in}} \eqdef \{ s_{\text{in}} \mid s \in \zstates \} $, 
$S_{\Box \text{out}} \eqdef \{ s_{\text{out}} \mid s \in \zstates \}$ and 
$S_{\ocircle \text{in}} \eqdef S_{\text{in}} \setminus S_{\Box \text{in}}$,
$S_{\ocircle \text{out}} \eqdef S_{\text{out}} \setminus S_{\Box \text{out}}$.
This allows us to define: 
\begin{itemize}
\item $S^t \eqdef S_{\text{in}} \cup S_{\text{out}}, 
\zstates^t \eqdef S_{\Box \text{in}} \cup S_{\Box \text{out}}, 
\rstates^t \eqdef S_{\ocircle \text{in}}\cup S_{\ocircle \text{out}}$, 
\item $\transition^t \eqdef \{ s_{\text{in}} \transition^t s_{\text{out}} \mid s\in S \} \cup 
\{ s_{\text{out}} \transition^t s^{'}_{\text{in}} \mid s \transition s' \}$
\item \[r^t(s \transition^t s') \eqdef 
	\begin{cases*}
     -m \text{ (resp. } +m) \text{ if } s \in S_{\text{out}}  \\
     r(s)  \text{ if } s \in S_{\text{in}}
    \end{cases*} \]
\item For $s \in \rstates$ and $s' \in S$, define $P^t(s_{\text{in}})(s'_{\text{out}}) \eqdef 1$ and 
$P^t(s_{\text{out}})(s'_{\text{in}}) \eqdef P(s)(s')$. 
\end{itemize}

Note that the definitions of $r^t$ and $P^t$ are complete,
since there are no transitions from $S_{\text{in}}$ to itself or from $S_{\text{out}}$ to itself.
\end{definition}

\smallskip
\begin{lemma}
Given an MDP $\mdp$ with state based rewards bounded between $+m$ and $-m$, $m \in \R$,
for every optimal $\limsupppobj$ 
(resp. $\liminfppobj$) 
strategy $\sigma$ in $\mdp^t$, 
there exists an optimal $\limsupppobj$ 
(resp. $\liminfppobj$) 
strategy $\sigma'$ in $\mdp$ with the same memory as $\sigma$ such that 
$\probm_{\mdp^t, s_{0_{\text{in}}}, \sigma}(\limsupppobj) = \probm_{\mdp, s_0, \sigma'}(\limsupppobj)$
(resp.\ $\probm_{\mdp^t, s_{0_{\text{in}}}, \sigma}(\liminfppobj) = \probm_{\mdp, s_0, \sigma'}(\liminfppobj)$).
\end{lemma}
\begin{proof}
Consider an optimal strategy $\sigma$ for $\limsupppobj$ (resp. $\liminfppobj$) in $\mdp^t$. 
We will use $\sigma$ to construct a new strategy $\sigma'$ which is optimal $\limsupppobj$ (resp. $\liminfppobj$) in $\mdp$.
For $s,s' \in S$, let $\sigma'(s)(s') \eqdef \sigma(s_{\text{out}})(s'_{\text{in}})$. 
Note that $\sigma'$ ignores all of the actions that $\sigma$ takes from an \textit{in} state to an \textit{out} state.
This is because by construction of $\mdp^t$ there are no decisions to be made in those cases.
In the case where $\sigma$ is Det(F) or Rand(F), $\sigma'$ makes all of the same decisions as $\sigma$, and thus the probability of a given sequence of rewards is the same between the two strategies (modulo the buffer rewards $\pm m$ in $\mdp^t$). 
Thus the attainment of $\sigma$ and $\sigma'$ must be the same.
In the case where $\sigma$ uses a step counter in its memory, notice that path lengths in $\mdp^t$ are doubled relative to path lengths in $\mdp$.
However, since there is no decision to be made in alternating states, we adjust for this by making $\sigma'$'s step counter count twice as fast so that all decisions are made at even step counter values, mirroring the step counter $\sigma$ uses. Hence $\sigma$ and $\sigma'$ must have the same attainment.
\end{proof}

\begin{remark}
The reverse construction also clearly works. Given an MDP $\mdp$ with transition based rewards, 
there exists an MDP $\mdp^s$ with state based rewards
such that for every optimal $\limsupppobj$ (resp. $\liminfppobj$) strategy $\sigma$ in $\mdp^s$, 
there exists an optimal $\limsupppobj$ (resp. $\liminfppobj$) strategy $\sigma'$ in $\mdp.$

The proof follows from a very similar construction which replaces all instances of
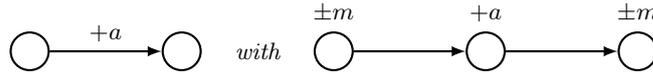
\begin{center}
    \begin{tikzpicture}
    
    \node[draw, circle, minimum height=5mm] (S1) at (0,0){};
    \node[draw, circle, minimum height=5mm] (S2) at (2,0){};
    
    \node (T1) at (3,0){with};
    \node (T2) at (6,0.5){$+a$};
    \node (T3) at (4,0.5){$\pm m$};
    \node (T4) at (8,0.5){$\pm m$};

    \node[draw, circle, minimum height=5mm] (S3) at (4,0){};
    \node[draw, circle, minimum height=5mm] (S4) at (6,0){};
    \node[draw, circle, minimum height=5mm] (S5) at (8,0){};
      
     \draw[->,>=latex] (S1) edge node[above, midway]{$+a$}  (S2)
     (S3) edge  (S4)
     (S4) edge  (S5);

    \end{tikzpicture}
\end{center}    
where the $\pm m$ means $-m$ for $\limsupppobj$ and $+m$ for $\liminfppobj$.
\end{remark}

\end{appendices}  

\newpage
\bibliography{conferences,refs}

\end{document}